%% file: paper.tex
\RequirePackage{fix-cm}
\documentclass[smallextended]{svjour3}       
\smartqed  
\usepackage{graphicx}
\usepackage{grffile}  
\usepackage[mode=buildnew]{standalone}

\input{macros.tex}

\usepackage{ifthen}
\newboolean{compilefromscratch}
\setboolean{compilefromscratch}{false}

\begin{document}

\title{Parallel-in-time high-order multiderivative IMEX solvers
}


\author{Jochen~Sch\"utz \and David~C.~Seal \and Jonas~Zeifang}


\institute{J. Sch\"utz \at
            Faculty of Sciences, Universiteit Hasselt, \\ Agoralaan Gebouw D, BE-3590 Diepenbeek, Belgium \\
              \email{jochen.schuetz@uhasselt.be}
           \and
           D. C. Seal\at
          United States Naval Academy, Department of Mathematics, \\ 572C Holloway Road, Annapolis, MD 21402, USA \\ \email{seal@usna.edu} 
          \and 
          J. Zeifang \at
          Institute of Aerodynamics and Gas Dynamics, University of Stuttgart, \\ Pfaffenwaldring 21, DE-70569 Stuttgart, Germany \\ \email{zeifang@iag.uni-stuttgart.de}
}

\date{Received: date / Accepted: date}

\maketitle

\begin{abstract}
    In this work, we present a novel class of parallelizable high-order time integration schemes for the approximate solution of additive ODEs. The methods achieve high order through a combination of a suitable quadrature formula involving multiple derivatives of the ODE's right-hand side and a predictor-corrector ansatz. The latter approach is designed in such a way that parallelism in time is made possible. We present thorough analysis as well as numerical results that showcase scaling opportunities of methods from this class of solvers.
\end{abstract}

\section{Introduction}\label{sec:Introduction}
\input{sec_introduction}

\section{Numerical algorithm}\label{sec:DescriptionAlg}
\input{sec_algorithm}

\section{Convergence analysis}\label{sec:ConvergenceAnalysis}
\input{sec_convergence_analysis}

\section{Parallel implementation}\label{sec:ParallelStrategy}
\input{sec_parallel}

\section{Numerical results}\label{sec:NumRes}
\input{sec_numerics}

\section{Improved algorithm and scaling results}\label{sec:ModificationsAlg}
\input{sec_results_modifications}

\section{Conclusions and outlook}\label{sec:Conclusion}
\input{sec_conclusion}

%
%

\begin{acknowledgements}
This study was initiated during a research stay of D.C.~Seal at the University of Hasselt, which was supported by
the Special Research Fund (BOF) of Hasselt University. Additional funding came from the Office of Naval Research,
grant number N0001419WX01523 and N0001420WX00219.
J.~Zeifang was partially supported by the German Research Foundation (DFG) through the project GRK 2160/1 ``Droplet Interaction Technologies".
The HPC-resources and services used in this work were provided by the VSC (Flemish Supercomputer Center), funded by the Research Foundation - Flanders (FWO) and the Flemish Government.
\end{acknowledgements}

\bibliographystyle{spmpsci}      
\bibliography{ListPaper}   

\end{document}

%% file: macros.tex
\usepackage{amsmath,amssymb,graphicx,stmaryrd}
\usepackage[dvipsnames]{xcolor}
\usepackage{tikz}
\usepackage{hyperref}
\hypersetup{
	colorlinks,
	linkcolor={red!50!black},
	citecolor={blue!50!black},
	urlcolor={blue!80!black}
}
\usepackage{pgfplots}
\usepackage{pgfplotstable}
\usepackage{placeins}
\usepackage{color}
\usepackage{pdfpages}
\usetikzlibrary{shapes,arrows,calc,external,matrix,positioning,patterns}
\usepackage{rotating}

\DeclareMathOperator{\R}{\mathbb{R}}
\DeclareMathOperator{\N}{\mathbb{N}}

\newcommand{\Dt}{\mathrm{d}t}

\newcommand{\II}{\mathcal I}

\renewcommand{\O}{\mathcal{O}}

\newcommand{\eps}{\varepsilon}

\newcommand{\dt}{\Delta t}

\DeclareMathOperator{\Id}{Id}
\DeclareMathOperator{\EE}{\mathfrak{E}}

\newcommand{\Phie}{\Phi_{\text{E}}}
\newcommand{\Phii}{\Phi_{\text{I}}}
\newcommand{\dPhie}{\dot\Phi_{\text{E}}}
\newcommand{\dPhii}{\dot\Phi_{\text{I}}}

\newcommand{\Tend}{T_{end}}

\newcommand{\km}{k_{\max}}

\newtheorem{assumption}{Assumption}
\newtheorem{algorithm}{Algorithm}

\newcommand{\algs}[4]{#1^{#2,[#3],{#4}}}

\renewcommand{\dot}[1]{\overset{\boldsymbol .}{#1}}

\newcommand{\method}[2]{\text{HBPC}(#1,#2)}
\newcommand{\methodt}[2]{\text{HBPC*}(#1,#2)}

\pgfplotscreateplotcyclelist{epslist}{%
	{red,thick,mark=o},
	{orange,thick,mark=pentagon},
	{green!60!black!90!,thick,mark=square},
	{blue!90!white!50,thick,mark=triangle},
	{blue!60!black!80,thick,mark=x},
	{violet	,thick	,mark=diamond},	
	{red,thick,dashed,mark=o},
	{orange,thick,dashed,mark=pentagon},
	{green!60!black!90!,thick,dashed,mark=square},
	{blue!90!white!50,thick,dashed,mark=triangle},
	{blue!60!black!80,thick,dashed,mark=x},
	{violet	,thick	,dashed,mark=diamond},
	{red,thick,dotted,mark=o},
	{orange,thick,dotted,mark=pentagon},
	{green!60!black!90!,thick,dotted,mark=square},
	{blue!90!white!50,thick,dotted,mark=triangle},
	{blue!60!black!80,thick,dotted,mark=x},
	{violet	,thick	,dotted,mark=diamond}}
\pgfplotscreateplotcyclelist{green}{%
	{green!60!black!10!,thick,mark=o},
	{green!60!black!20!,thick,mark=pentagon},
	{green!60!black!30!,thick,mark=square},
	{green!60!black!40!,thick,mark=triangle},
	{green!60!black!50!,thick,mark=x},
	{green!60!black!60!,thick,mark=otimes},
	{green!60!black!70!,thick,mark=o},
	{green!60!black!80!,thick,mark=pentagon},
	{green!60!black!90!,thick,mark=square},
	{green!60!black!100!,thick,mark=triangle},
	{blue!60!black!10!,thick,mark=o},
	{blue!60!black!20!,thick,mark=pentagon},
	{blue!60!black!30!,thick,mark=square},
	{blue!60!black!40!,thick,mark=triangle},
	{blue!60!black!50!,thick,mark=x},
	{blue!60!black!60!,thick,mark=otimes},
	{blue!60!black!70!,thick,mark=o},
	{blue!60!black!80!,thick,mark=pentagon},
	{blue!60!black!90!,thick,mark=square},
	{blue!60!black!100!,thick,mark=triangle}}

%% file: sec_introduction.tex
In this work, we consider numerical approximations of nonlinear differential equations of the form
\begin{align}
 \label{eq:ode}
 w'(t) &= \Phi(w) \equiv \Phii(w) + \Phie(w), \quad t \in (0, \Tend), \\ 
 \notag w(0) &= w_0. 
\end{align}
It is assumed that the right-hand side $\Phi$ is split into the contributions $\Phii$ (`stiff')  and $\Phie$ (`non-stiff'), to account for different scales in the solution. The rationale is that the term $\Phii$ will be treated implicitly to obtain a stable method, while the term $\Phie$ will be treated explicitly, typically for efficiency reasons by reducing the complexity of the algebraic solves required to address the implicit terms per time step. 

Differential equations that can be put into this form arise frequently, often resulting from the discretization of singularly perturbed partial differential equations. Some examples include problems in meteorology~\cite{giraldo2010semi}, geophysics~\cite{BispenIMEXSWE}, aerodynamics~\cite{BoQiuRussoXiong19,HaJiLi12,RSIMEXFullEuler}, molecular dynamics~\cite{filbet2010class} and many more. The list is by no means exhaustive. Here, we assume a splitting has already been performed.  Our goal is to develop a high-order implicit-explicit (IMEX) method that lends itself to time-parallelism. 

At the core of this work is the novel discretization method of Eq.~\ref{eq:ode} developed in~\cite{SealSchuetz19}. To the best of the authors' knowledge, this method is, together with the very recent publication~\cite{AbdiHojjati2020}, 
the first one to combine the IMEX (implicit/explicit, see, e.g.,~\cite{Ascher1997,Ascher1995,Bos09,pareschi2000implicit,christopher2001additive}) and the multiderivative paradigms~\cite{AiOk2019,TC10,Seal2015b,KastlungerWanner1972,TurTur2017,Seal13}. 
In this context, the term \emph{multiderivative} refers to the fact that higher derivatives of the unknown solution $w$ are used in the numerical method. While more established methods (Runge-Kutta, Adams methods, backward differentiation formulas (BDF) and the like) rely on only evaluating $\Phie$ and $\Phii$, multiderivative methods also take the temporal derivatives into account. 
To be more precise, in this work, we use two derivatives, i.e., also the quantities\footnote{The dot here ($\cdot$) refers to the time derivative $d/dt$, whereas the prime ($'$) refers to the Jacobian of the vector valued $\Phii$ and $\Phie$.}
\begin{align*} 
 \dPhii(w) := \Phii'(w) \Phi(w), \qquad \text{and} \qquad \dPhie(w) := \Phie'(w) \Phi(w)
\end{align*}
are taken into account. Note that there holds $w''(t) = \dot \Phi(w) \equiv \dPhii(w) + \dPhie(w)$. This has of course the potential to increase the order of consistency without adding more stages or steps, respectively. The limitation to two derivatives is only for the ease of presentation. 
The method developed in~\cite{SealSchuetz19} has already been extended to cope with higher derivatives~\cite{alex2020highorder}; and also in the context of this present work, this is easily possible. 

There are three unique features to consider for the class of methods presented here: 
\begin{itemize} 
 \item The methods we present use a \textbf{predictor-corrector} strategy for increasing the overall order of accuracy. This means that in each timestep, a predicted value  (we will indicate values corresponding to the predictor with $[0]$), as well as several corrector steps (indicated with $[k]$, $1 \leq k \leq \km$) are computed. Here, $\km$ is a user-defined parameter that defines the maximum number of corrections in a single time step. Similiar to deferred correction, or DC-type methods (cf.~\cite{CausleySeal19,DuttGreenRokh00,OngSpiteri} and the references therein), the corrected steps pick up an order of accuracy each time, until some maximum accuracy is reached based on the underlying quadrature rule.
 
 \item The methods we construct are derived from two-derivative Runge-Kutta methods, however, they are modified in such a way that they can be applied to any \textbf{implicit-explicit (IMEX) splitting} $\Phi(w) = \Phii(w) + \Phie(w)$ of the ODE. This ultimately leads to the fact that the methods are not one-step, but two-step methods, with a very mild dependency on the previous step. As there are multiple stages involved, this leads to general linear methods with associated convergence theory~\cite{Butcher1966}, but we have higher derivatives of the unknown involved in our analysis.
 
 \item The methods we produce can all leverage \textbf{parallel-in-time} implementations to leverage distributed or multicore computer architectures.
 
\end{itemize}

The ability to parallelize an ODE solver in time has become more and more important in recent years with the number of processors available growing constantly. Because of the causality principle, parallelization in time poses different challenges then parallelization in space. For an overview of temporally parallel methods, we refer to Gander's work~\cite{Gander2015} and the very recent overview article by Ong and Schroder~\cite{OngSchroder2020}. Also the website \cite{PintURL} is a beautiful source of information on the topic. Temporally parallel methods can be cast into different classes. The scheme we present in this work uses the concept of pipelining, i.e., the information and data flow is directed in such a way that the correction iterations can be put on different processors. We have been inspired by the RIDC (revisionist integral deferred correction) schemes~\cite{ChriOng11,ChriOng10,ChriOng2015}, however, similar ideas have been around for a long time, see~\cite{MirankerLiniger1967}. The essence of pipelining is that, given one has enough processors at one's disposal, the higher order achieved through the corrections can be obtained in roughly the same wallclock-time that it takes to compute the lower-order predictor. In our case, the predictor's convergence order is two, while the corrections can increase this order up to the order of the underlying Runge-Kutta quadrature. In our examples, this is four, six and eight, respectively, but conceptually, the order is not limited to this.

The paper is organized as follows: In Sec.~\ref{sec:DescriptionAlg}, we describe the underlying algorithm in detail and introduce the necessary notation. Thereafter, in Sec.~\ref{sec:ConvergenceAnalysis}, we analyze stability and convergence of the method through writing it in a one-step fashion. The method is inherently time-parallel, which is laid out in Sec.~\ref{sec:ParallelStrategy}. Also, two alternative approaches are shown and later compared. In Sec.~\ref{sec:NumRes} we show numerical results for well-known test cases. We find some improvements that one can make to the numerical algorithm described earlier; these are shown in Sec.~\ref{sec:ModificationsAlg}. Subsequently, scaling results are presented. As usual, the last section, Sec.~\ref{sec:Conclusion} offers conclusion and outlook.

%% file: sec_algorithm.tex
\label{sec:algorithm}
We start with setting the notation that is used in this work. For expository purposes, we work with a fixed timestep $\dt$, and therefore with a fixed total number of timesteps, denoted by $N$, we must have
\begin{align*} 
 \dt := \frac{\Tend}{N}.
\end{align*}
The discrete time levels are defined as
\begin{align*} 
 t^n := n \dt, \qquad 0 \leq n \leq N.
\end{align*}
Note that the uniform timestep we assume here is not necessary in practical computations, but it simplifies the analysis.
Throughout the whole publication, we assume that $\Phi$ and the split functions $\Phii$, $\Phie$ and the solution $w$ are sufficiently smooth to warrant picking up high-order accuracy from the numerical method.  Formally, we make the following assumption:
\begin{assumption}\label{ass:lipschitz}
 All occuring functions $\Phi$, $\Phii$, $\Phie$ and their temporal derivatives $\dot\Phi,\dPhii$, $\dPhie$ are assumed to be smooth and Lipschitz continuous. 
\end{assumption}

In this work, we are interested in multiderivative Runge-Kutta methods. These are a relatively straightforward extension of well-known Runge-Kutta methods that include extra derivatives of the right hand side function \cite[Definition 1]{Seal13}).  To set notation, we include a full definition of a \emph{two}-derivative Runge-Kutta method.  Higher derivatives can be included by including more terms in the expansion for each of the stages.
\begin{definition}\label{def:RKlimit}
Let $w^n$ denote an approximate solution to $w$ at point $t^n$, and let $w^{n,l}$, $1 \leq l \leq s$ denote the stage values to be computed. A \emph{two-derivative Runge-Kutta method} is any method that can be cast as
\begin{align*} 
 w^{n,l} := w^n + \dt \sum_{j=1}^s B_{lj}^{(1)} \Phi(w^{n,j}) + \dt^2 \sum_{j=1}^s B_{lj}^{(2)} \dot\Phi(w^{n,j}), \qquad 1 \leq l \leq s,
\end{align*}
and update given by
\begin{equation*}
w^{n+1} = w^n + \dt \sum_{j=1}^s b_j^{(1)} \Phi(w^{n,j}) + \dt^2 \sum_{j=1}^s b^{(2)}_j \dot\Phi(w^{n,j}),
\end{equation*}
where the $B^{(1)}$, $B^{(2)}$, $b^{(1)}$, and $b^{(2)}$ are the Butcher tableaux that define the scheme.  We say the method is \emph{globally stiffly accurate} if
$b^{(1)}_j = B^{(1)}_{sj}$ and
$b^{(2)}_j = B^{(2)}_{sj}$, i.e., the last stage defines the update:
\begin{align*} 
 w^{n+1} := w^{n,s}.
\end{align*}
\end{definition}

Note that this is a fully coupled system of (potentially nonlinear) equations whenever the tableaux are not lower triangular. In this form, the Runge-Kutta method is hence typically not suited for high-dimensional ODEs, nonetheless, they can serve as excellent background methods for creating more efficient solvers.

Of particular note are the two-derivative Runge-Kutta \emph{collocation} methods.  These methods are derived by fitting a Hermite-Birkhoff polynomial interpolant through the time instances given by $c \dt$ and integrating the result.  
In this work, we use equidistant collocation points, which implies:
\begin{itemize}
	\item the method is of order $q$, with $q = 2s$ being twice the number of stages (hence, number of elements in $c$);
	\item the stage order is also $q$; 
	\item the last stage is equal to the update step, which corresponds to the globally stiffly accurate or first-same-as-last property already known in standard Runge-Kutta methods \cite{BOSCARINO201760,HaiWan1,HaiWan} and will have an impact on asymptotic properties of the method \cite{SKN15}. 
\end{itemize}

\begin{example}[Two-derivative Hermite-Birkhoff collocation methods]
\label{expl:rk}
In this work, we use the following two-derivative Runge-Kutta methods: 
 \begin{itemize} 
  \item A fourth-order method ($q=4$) with two stages ($s=2$, one being fully explicit), which exactly corresponds to the method used in \cite{SealSchuetz19}: 
  \begin{align}\label{eq:butcher4}
        c = \begin{pmatrix}0 \\ 1\end{pmatrix}, \quad
        B^{(1)} = 
        \begin{pmatrix}
        0 & 0 \\[0.5em]
        \frac 1 2& \frac 1 2
        \end{pmatrix}, \quad
        B^{(2)} = 
        \begin{pmatrix}
        0 & 0   \\[0.5em]
        \frac 1 {12} & \frac{-1}{12}
        \end{pmatrix}.
        \end{align}
  \item A sixth-order method  ($q=6$) with three stages ($s=3$, one being fully explicit), as also used in \cite{SSJ2017}:
  \begin{align}\label{eq:butcher6}
        c = \begin{pmatrix}0\\ \frac 1 2\\ 1\end{pmatrix}, \quad
        B^{(1)} = 
        \begin{pmatrix}
        0 & 0 & 0 \\[0.5em]
        \frac{101}{480} & \frac{8}{30} & \frac{55}{2400} \\[0.5em]
        \frac{7}{30} & \frac{16}{30} & \frac{7}{30} \\
        \end{pmatrix}, \quad
        B^{(2)} = 
        \begin{pmatrix}
        0 & 0 & 0  \\[0.5em]
        \frac{65}{4800} & -\frac{25}{600} &  -\frac{25}{8000} \\[0.5em]
        \frac{5} {300} & 0 &  -\frac{5}{300} 
        \end{pmatrix}.
        \end{align}
  \item An eigth-order method ($q=8$) with four stages ($s=4$, one being fully explicit):
     \begin{align}\label{eq:butcher8}
            c = \begin{pmatrix}0\\[0.5em] \frac 1 3\\[0.5em] \frac 2 3\\[0.5em] 1\end{pmatrix}, \quad
            B^{(1)} &= 
            \begin{pmatrix}
            0 & 0 & 0 & 0\\[0.5em]
            \frac{6893}{54432}& \frac{313}{2016}& \frac{89}{2016}& \frac{397}{54432} \\[0.5em]
            \frac{223}{1701}&    \frac{20}{63}&   \frac{13}{63}&   \frac{20}{1701} \\[0.5em]
            \frac{31}{224}&   \frac{81}{224}&  \frac{81}{224}&    \frac{31}{224} 
            \end{pmatrix}, \\
            B^{(2)} &= 
            \begin{pmatrix}
            0 & 0 & 0 & 0  \\[0.5em]
            \frac{1283}{272160}& -\frac{851}{30240}& -\frac{269}{30240}& -\frac{163}{272160} \\[0.5em]
            \frac{43}{8505}&    -\frac{16}{945}&    -\frac{19}{945}&     -\frac{8}{8505} \\[0.5em]
            \frac{19}{3360}&    -\frac{9}{1120}&     \frac{9}{1120}&    -\frac{19}{3360}
            \end{pmatrix}.
            \end{align}
\end{itemize}
\end{example}

We choose to use the above-mentioned class of methods with equispaced abscissa for the sake of a clearer presentation of results. Other multiderivative Runge-Kutta methods could also be used, of course with the necessary modifications. Higher derivatives can also be incorporated as an alternative option to further increase the overall order.

With each Runge-Kutta method, there is an underlying quadrature rule connected to the solver.  That is, given the quantities $\phi^j$, $1 \leq j \leq s$, we define the ``quadrature rule" as the multiderivative Runge-Kutta flux update, given by 
	\begin{align*} 
	 \II_l(\phi^1, \ldots \phi^s ) := 
        \dt \sum_{j=1}^s B^{(1)}_{lj}  \phi^j  + \dt^2 \sum_{j=1}^s B^{(2)}_{lj} {\dot \phi}^j. 
	\end{align*}
These terms are precisely the terms required to produce each stage in the Runge-Kutta method.  They may or may not be high order quantities, unless further assumptions are made about the solver. In the case of the Hermite-Birkhoff-type approach, they are of higher order:

\begin{lemma} 
For each Hermite-Birkhoff Runge-Kutta method, there holds,
\begin{align*} 
 \int_{t^n}^{t^{n+1}} \Phi(w(t)) \Dt = \II_l\left(\Phi(w(t^n + c_1 \dt)), \ldots, \Phi(w(t^n + c_s \dt))\right)  + \O(\dt^{q+1}).
\end{align*}
\end{lemma}
\begin{proof}
This is due to the construction: The Hermite-Birkhoff polynomial is of order $2s-1=q-1$, which gives an integral approximation of order $q+1$. (Note that the length of the integration area is $\dt$.)

\qed
\end{proof}

In the sequel, we exploit precisely this quadrature to extend the method shown in \cite{SealSchuetz19} to higher orders and parallelism in time. 
Before we actually define the algorithm, let us clarify the notation used. The algorithm to be presented relies on the approximated quantities:
\begin{align} 
 \label{eq:wnkl}
 \algs w  n {k} l \approx w(t^{n} + c_l \dt), \quad 0 \leq n \leq N, \quad 0 \leq k \leq \km, \quad 1 \leq l \leq s.
\end{align}
Here, $n$ is the usual discrete time level, and the $l$ refers to the stages that are present within the Runge-Kutta method. The index $k$ is a parameter that is associated to `correction' steps (to be explained below). The final index $\km$ is a fixed parameter chosen by the user to describe the total number of iterations sought.  For non-stiff problems, it is typically dictated by the maximum order that can be reached, while for stiff problems, it can be advantageous to use a larger $\km$ to overcome convergence issues \cite{SealSchuetz19}, or even allow the Algorithm to choose this adaptively. 

Finally, we make the abbreviations 
\begin{align*} 
 \algs \Phi n k l := \Phi\left(\algs w n k l\right), \quad
 \algs \Phii n k l := \Phii\left(\algs w n k l\right), \quad
 \algs \Phie n k l := \Phie\left(\algs w n k l\right),
\end{align*}
to simplify the notation.

We are now ready to formulate the following algorithm.  Given the Butcher tableaux $B^{(1)}$ and $B^{(2)}$ and the time instances $c$, (e.g., any of the methods from Example~\ref{expl:rk}), the Hermite-Birkhoff Predictor Corrector method is as follows:
\begin{algorithm}[$\method{q,\km}$]\label{alg:mdode}
To advance the solution to Eq.~\eqref{eq:ode} from time level $t^{n}$ to time level $t^{n+1}$, fill the values $\algs w n 0 l$ using a second-order IMEX-Taylor method:

\begin{enumerate}
	\item \textbf{Predict.} Solve the following expression for $\algs w n 0 l$ and $1 \leq l \leq s$: 
	 \begin{align} 
	 \begin{split}
	 \label{eq:predict}
	 \algs w {n} 0 l := {\color{VioletRed}\algs w {n-1} 0 s} & + c_l\dt \left(\algs {\Phii} n 0 l + {\color{VioletRed}\algs{\Phie} {n-1} 0 s}\right) \\ & + \frac{(c_l\dt)^2}{2} \left({\color{VioletRed}\algs {\dPhie} {n-1} 0 s} - \algs {\dPhii} n 0 l\right). 
    \end{split}
	\end{align}
	Subsequently:
	\item \textbf{Correct.} Solve the following for $\algs w {n} {k+1} l$, for each $2 \leq l \leq s$ and each $0 \leq k < k_{\max}$: 
	\begin{align}
	\begin{split}
	\label{eq:correct}
	\algs w {n} {k+1} 1 &:= {{\color{VioletRed}\algs w {n-1} {k+2} s}}_{}, \\
	\algs w {n} {k+1} l &:= {\color{VioletRed}\algs w {n-1} {k+2} s}  \\ 
	&
	+ {\dt} \left( \algs {\Phii} {n} {k+1} l - \algs {\Phii} n {k} l\right)
	- \frac{\dt^2}{2} \left(\algs {\dPhii} {n} {k+1} l- \algs {\dPhii} {n} {k} l\right)
	\\ & + \II_l(\algs \Phi{n}{k} 0, \algs \Phi n {k} 1, \ldots, \algs \Phi n {k} s).
	\end{split}
	\end{align}
	In order to close the recursion, if $k = k_{\max} -1$, then replace each of the $k+2$ with $\km$.
	  
	\item \textbf{Update.} In order to preserve the first-same-as-last property, we put
	\begin{align*} 
	        w^{n+1} := \algs w n {k_{\max}} s.
	\end{align*}

\end{enumerate}

Finally, to seed initial conditions for this solver, we define 
\begin{align*}
	\algs w {-1} k s := w_0, \qquad 0 \leq k \leq \km.
\end{align*}

\end{algorithm}

Note that the starting value of $k=0$ denotes a straightforward second-order IMEX-Taylor scheme; each $k > 0$ represents a corrected version thereof, taking into account the quadrature rule obtained from the Runge-Kutta scheme. The reasoning is that with each additional correction, we see one extra order of convergence until the maximal order of the Runge-Kutta method, $q$, has been found.  

\begin{remark} 
	The difference between this algorithm and our previous one found in \cite{SealSchuetz19} lies in the highlighted red terms:
	\[
		\left\{{\algs w {n-1} 0 s}, {\algs{\Phie} {n-1} 0 s}, {\algs{\dPhie} {n-1} 0 s}, {\algs w {n-1} {k+2} s}\right\}.
	\]
	In our original work each of these would have been evaluated at $w^n$.
	Here, we work with different iterate numbers.  \textbf{This potent modification makes it possible to parallelize the method in time without sacrificing the order of accuracy.} The idea that will be layed out in this publication is very similar to -- and in fact inspired by -- the strategy proposed in \cite{ChriOng11,ChriOng10}.
\end{remark}


\begin{remark}
 Please note that Alg. \ref{alg:mdode} is not a Runge-Kutta method. 
\end{remark}

\begin{remark} 
There is no implicitness in the Runge-Kutta quadrature $\II_l$. The only implicitness is in the difference of the $\Phii$ and $\dPhii$. This term is used to introduce the necessary stability for stiff problems. 
\end{remark}

In the next three sections we present a thorough analysis of this solver followed by a discussion of a parallization strategy and then numerical results.  In Section \ref{sec:ModificationsAlg}, we suggest an alternative formulation that provides a low-storage alternative and even better scaling results.


%% file: sec_convergence_analysis.tex
Due to the changes made in the algorithm in comparison to the one proposed in \cite{SealSchuetz19}, this is no longer a one-step multiderivative time integrator, but a general linear method. It has both a multistage as well as a multistep flavour, and we need to follow the convergence analysis outlined in the seminal work by Butcher \cite{Butcher1966} to conduct our analysis. To ease the presentation, we will in the sequel assume that we are treating a scalar differential equation, i.e., we postulate that 
\begin{align*} 
 \Phi: \R \rightarrow \R.
\end{align*}

We start by rewriting Alg. \ref{alg:mdode} to make it more accessible to a convergence analysis. Define the vector $Y^{n+1} \in \R^{s\cdot(k_{\max}+1)}$ consisting of all stages and all correction steps at time instance $t^n = n \dt$,  $n > 0$, i.e., 
\begin{align*} 
 Y^{n+1} := \left(\begin{matrix}
                \algs w n 0 1 \\
                \colon \\
                \algs w n 0 s \\
                \algs w n 1 1 \\
                \colon \\
                \algs w n 1 s \\
                \colon  \\
                \algs w n {k_{\max}} s
          \end{matrix} \right) \in \R^{s\cdot (k_{\max}+1)}. 
\end{align*}
Note that we have made a shift in the $n$ for a clearer exposition. 
For $n = 0$, we define $Y^0$ to be the vector with entries consisting of the solution at time $t=0$:
\begin{align*}
 Y^0 := \left(\begin{matrix}
                w_0 \\
                \colon \\
                w_0
          \end{matrix} \right) \in \R^{s\cdot (k_{\max}+1)}. 
\end{align*}
Then, we can write the time integrator from Alg. \ref{alg:mdode} in a ``one-step-fashion'' as 
\begin{align*} 
 Y^{n+1} = A Y^n & + \dt \left(B_{E} \Phie(Y^{n+1}) + B_{I} \Phii(Y^{n+1})\right) \\ & + \frac{\dt^2}{2} \left(\tilde B_{E} {\dPhie}(Y^{n+1}) + \tilde B_{I} {\dPhii}(Y^{n+1})\right) 
\end{align*}
for matrices $A$, $B_I$, $B_E$, $\tilde B_E$ and $\tilde B_I$ in $\R^{s\cdot (k_{\max}+1)\times s \cdot (k_{\max}+1)}$. 

In the sequel, we use the well-known inf-norm $\|\cdot\|_{\infty}$ for both matrices and vectors. For the sake of readability, we write $\|\cdot\|$ instead of $\|\cdot\|_{\infty}$. That is, for $v \in \R^\kappa$, we define $\|v\| := \max_{1 \leq i \leq \kappa} |v_i|$; for matrices $C \in \R^{\kappa \times \kappa}$, there holds 
\begin{align*}
 \|C\| := \sup_{v \in \R^\kappa} \frac{\|Cv\|}{\|v\|}.
\end{align*}

\begin{lemma} 
 The matrix $A$ is power-bound by one, i.e., for each $n \in \N$, there holds
 \begin{align*} 
  \|A^n\| \leq 1.
 \end{align*}

\end{lemma}
\begin{proof}
 Define the matrix $\EE \in \R^{s\times s}$ that is zero except for the last column, which is filled with ones, i.e., 
 \begin{align*} 
  \EE = \left( \begin{matrix} 
                0 &0 & \cdots &0 & 1 \\
                0 &0 & \cdots &0 & 1 \\
                \vdots & \vdots &  & \vdots& \vdots \\
                0 & 0 & \cdots & 0 & 1
               \end{matrix} \right). 
 \end{align*}
 Then, $A$ is given as the block matrix 
 \begin{align} 
  \label{eq:A}
  A = \left( \begin{matrix} 
              \EE & 0 & 0 & 0 & \cdots & 0 & 0 \\
              0 & 0 & \EE & 0 & \cdots & 0 & 0 \\
              0 & 0 & 0 & \EE & \cdots & 0 & 0 \\
              \vdots & \vdots & \vdots & \vdots & \ddots & \vdots & \vdots \\
              0 & 0 & 0 & 0 & \cdots & \EE & 0 \\
              0 & 0 & 0 & 0 & \cdots & 0 & \EE \\
              0 & 0 & 0 & 0 & \cdots & 0 & \EE
             \end{matrix} \right).
 \end{align}
 There holds $\|Av\| \leq \|v\|$, which can be seen due to the fact that for each $1 \leq i \leq s\cdot (k_{\max}+1)$, there is a $1 \leq j \leq s \cdot (k_{\max}+1)$ such that $(Av)_i = v_j$. This concludes the proof.

\qed
\end{proof}

\begin{corollary}\label{cor:boundedness}
 There exists a constant $C \in \R$ such that for each $j \in \N^{\geq 0}$, there holds
 \begin{align*} 
  \|A^j B_E\| \leq C, \quad \|A^j B_I\| \leq C, \quad \|A^j \tilde B_E\| \leq C, \quad \|A^j \tilde B_I\| \leq C. 
 \end{align*}
\end{corollary}
Similarly to \cite{Butcher1966}, we define the vector $W^n$ as the vector of the exact solutions at time instant $t^n$, i.e., 
\begin{align*}
 W^{n+1} := \left(\begin{matrix}
                w(t^n + c_1 \dt) \\
                \colon \\
                w(t^n + c_s \dt) \\ 
                w(t^n + c_1 \dt) \\
                \colon \\
                w(t^n + c_s \dt)
          \end{matrix} \right) \in \R^{s\cdot (k_{\max}+1)}, \qquad n \geq 0. 
\end{align*}
We define $W^0 := Y^0$. 
For a convergence analysis, we are interested in the error, or the differnce between the approximate and the exact solution: $Z^n := Y^n - W^n$.  This quantity fulfills the following time-discrete equation: 

\begin{align}
 \label{eq:zn}
 \begin{split}
 Z^{n} &\equiv Y^{n} - W^{n} \\ = & 
         A Z^{n-1} + \dt B_{E} \left( \Phie(Y^{n})-\Phie(W^{n})\right) + 
                   \dt B_{I} \left( \Phii(Y^{n})-\Phii(W^{n})\right)  \\ 
                   & +\frac{\dt^2}{2} \tilde B_{E} \left( \dPhie(Y^{n})-\dPhie(W^{n})\right) + 
                   \frac{\dt^2}{2} \tilde B_{I} \left( \dPhii(Y^{n})-\dPhii(W^{n})\right) \\ &+ E_n,
 \end{split}
\end{align}
with $E^n$ the local consistency error 
\begin{align*}
 E^n := AW^{n-1} &+ \dt \left(B_{E} \Phie(W^{n}) +B_{I} \Phii(W^{n})\right) \\ &+ \frac{\dt^2}{2} \left(\tilde B_{E} \dPhie(W^{n}) +\tilde B_{I} \dPhii(W^{n})\right) - W^{n}.
\end{align*}
As in \cite{SealSchuetz19}, we can prove that the local consistency error at correction level $k$ is given by $\O\left(\dt^{\min\{2+k,q \}+1}\right)$. For $k=0$, this is straightforward -- this is second-order IMEX-Taylor.  For $k > 0$, techniques as in \cite{SealSchuetz19} (and long known to the spectral deferred correction (SDC) community, see, e.g., \cite{DuttGreenRokh00}) can be employed. 

\begin{lemma}\label{la:en}
 Subdivide $E^n$ into $\km+1$ blocks of size $s$ each; call the individual blocks $E^{n,[k]}$, $0 \leq k \leq \km$. Then, there holds  
 \begin{align*}
               E^{n,[k]} = O(\dt^{\min\{2+k,q\}+1}). 
              \end{align*}
\end{lemma}
\begin{proof}
The proof goes along the same lines as in \cite{SealSchuetz19} and is hence omitted. 

\qed
\end{proof}
The following lemma, although straightforward to prove, is of utmost importance to the convergence analysis. 
\begin{lemma}
 The recursion for $Z^n$ given in \eqref{eq:zn} can be explicitly written as: 
 \begin{align}
   \label{eq:znexact}
   \begin{split}
   Z^n = & \dt \sum_{j=0}^{n-1} A^j B_E \left(\Phie(Y^{n-j})-\Phie(W^{n-j})\right) 
      \\ &+ \dt \sum_{j=0}^{n-1} A^j B_I \left(\Phii(Y^{n-j})-\Phii(W^{n-j})\right) \\
      & + \frac{\dt^2}{2}   \sum_{j=0}^{n-1} A^j \tilde B_E \left(\dPhie(Y^{n-j})-\dPhie(W^{n-j})\right) 
      \\ &+ \frac{\dt^2}{2} \sum_{j=0}^{n-1} A^j \tilde B_I \left(\dPhii(Y^{n-j})-\dPhii(W^{n-j})\right)
    \\ &+ \sum_{j=0}^{n-1} A^j E^{n-j}. 
   \end{split}
 \end{align}
\end{lemma}
\begin{proof}
This can be easily proved using \eqref{eq:zn} and an induction over $n$. Note that one has to take into account that $Z^0 = 0$ by definition. 

\qed
\end{proof}

In order to show convergence, we need the following result on the analytical solution to a linear recursion equation: 
\begin{lemma}\label{la:recursion}
 Consider the following linear recursion: 
 \begin{align} 
  \label{eq:recursion}
  x_n = \alpha (x_0 + \ldots + x_{n-1}) + \beta_n, \qquad n \geq 1,
 \end{align}
with $\alpha, \beta_n \in \R$ given for all $n \in \N$. (As indicated, $\beta_n$ is allowed to depend on $n$, while $\alpha$ is not.) Also $x_0$ is supposed to be known. 
The analytical solution to this recursion is given by
 \begin{align} 
  \label{eq:analytics}
  x_n = \alpha (\alpha+1)^{n-1}x_0 + \sum_{j=1}^{n-1} \beta_j \alpha (1+\alpha)^{n-j-1} + \beta_n.
 \end{align}
\end{lemma}

\begin{proof} 
 We proceed by induction.  For $n=1$, we have $x_1 = \alpha x_0 + \beta_1$, which agrees with formula \eqref{eq:analytics}.
 Now assume that \eqref{eq:analytics} is correct for all $n < N$. Then, there holds: 
 \begin{align*} 
  x_N &= \alpha \sum_{n=0}^{N-1} x_n \\
  &= \alpha x_0 + \alpha \sum_{n=1}^{N-1} \left(\alpha (\alpha+1)^{n-1}x_0 + \sum_{j=1}^{n-1} \beta_j \alpha (1+\alpha)^{n-j-1} + \beta_n\right) + \beta_N \\
  &= \underbrace{\alpha x_0 + \alpha \sum_{n=1}^{N-1} \alpha (\alpha+1)^{n-1}x_0}_{=:I} + \underbrace{\alpha \sum_{n=1}^{N-1} \sum_{j=1}^{n-1} \beta_j \alpha (1+\alpha)^{n-j-1} + \alpha \sum_{n=1}^{N-1} \beta_n}_{=:II} + \beta_N. 
 \end{align*}
 We compute the terms separately: 
 \begin{align*} 
  I &= \alpha x_0 \left(1 + \sum_{n=0}^{N-2} \alpha (\alpha + 1)^n \right) = \alpha x_0 \left(1  + \alpha \frac{(\alpha+1)^{N-1}-1}{\alpha}\right) \\ &= \alpha x_0 (\alpha+1)^{N-1}.
 \end{align*}
 Now for $II$, there holds via a switching of the summation that
 \begin{align*} 
  II &= \alpha \sum_{n=1}^{N-1} \sum_{j=1}^{n-1} \beta_j \alpha (1+\alpha)^{n-j-1} + \alpha\sum_{n=1}^{N-1} \beta_n\\ 
     &= \alpha \sum_{j=1}^{N-2} \beta_j \sum_{n=j+1}^{N-1} \alpha (1+\alpha)^{n-j-1} + \alpha \sum_{n=1}^{N-1} \beta_n \\
     &= \alpha \sum_{j=1}^{N-2} \beta_j \left({(1+\alpha)^{N-j-1}}-1\right) + \alpha \sum_{n=1}^{N-1} \beta_n = \alpha \sum_{j=1}^{N-1} \beta_j (1 + \alpha)^{N-j-1}.
 \end{align*}
 All together, this shows that also $x_N$ can be formed according to \eqref{eq:analytics}. The statement is hence true. 

\qed
\end{proof}

Using the above results, we can immediately show that the method is second-order convergent: 
\begin{theorem}\label{thm:convergence}
 Given that $\dt$ is sufficiently small, and given the assumptions on the functions $\Phi$, $\Phii$ and $\Phie$ as done in Assumption \ref{ass:lipschitz}, there is a constant $C$ such that 
 \begin{align*} 
  Z^n \leq C \dt^2, 
 \end{align*}
 and hence, the method is second-order convergent. 
\end{theorem}
\begin{proof}
 Note that due to Corollary \ref{cor:boundedness}, we can estimate terms like $\|A^j B_E\|$ against fixed constants. 
 Using the expression \eqref{eq:znexact}, we can hence compute 
 \begin{align*} 
  \|Z^n\| \leq &\ \dt L_{\Phie} \sum_{j=0}^{n-1}  \|A^j B_E\|  \|Z^{n-j}\| 
             + \dt L_{\Phii}   \sum_{j=0}^{n-1}  \|A^j B_I\| \|Z^{n-j}\| 
            \\&+ \frac{\dt^2}{2} L_{\dPhie}   \sum_{j=0}^{n-1}  \|A^j \tilde B_E\|\|Z^{n-j}\| 
             + \frac{\dt^2}{2} L_{\dPhii} \sum_{j=0}^{n-1}  \|A^j \tilde B_I\|  \|Z^{n-j}\| 
             \\&+ \sum_{j=0}^{n-1} \|A^j E^{n-j}\| \\
          \leq \ &C\dt \sum_{j=0}^{n-1}  \|Z^{n-j}\| + \sum_{j=0}^{n-1} \|E^{n-j}\|.
 \end{align*}
 Given that $C \dt$ is smaller than one, we get the inequality 
 \begin{align*} 
  \|Z^n\| \leq \ &C\dt \sum_{j=1}^{n-1}  \|Z^{n-j}\| + \sum_{j=0}^{n-1} \|E^{n-j}\|,
 \end{align*}
 with another constant $C$ that does not depend on $\dt$ or $N$. Note that summation of the first term begins at $j=1$.
 Now there holds that $\|Z^n\| \leq \eps_n$, with $\eps_0 = 0$ and $\eps_n$, $n > 0$, defined as 
 \begin{align*} 
  \eps_n = C \dt (\eps_0 + \ldots + \eps_{n-1}) + \sum_{j=0}^{n-1} \|E^{n-j}\|. 
 \end{align*}
 This is exactly of form \eqref{eq:recursion}, and has hence the analytical solution 
 \begin{align*} 
  \eps_n = \sum_{j=1}^{n-1} C \dt \left(\sum_{k=0}^{j-1} \|E^{n-j}\|\right) (1 + C\dt)^{n-j-1} + \sum_{j=0}^{n-1} \|E^{n-j}\|. 
 \end{align*}
 We can estimate $(1+C\dt)^{n-j-1}$ against $e^{C \Tend}$, and because of La. \ref{la:en}, there then holds that 
 \begin{align*} 
  \eps_n \leq C \dt^2.
 \end{align*}
 Note that the property $\sum_{k=0}^{j-1} \|E^{n-j}\|$ can be estimated against $C\dt^2$. Note also that constants are subject to change. They do however never depend on $N$ or $\dt$. 

\qed
\end{proof}

Second order convergence is of course not completely what we expect from Alg. \ref{alg:mdode}. In the following, we explore the nature of the higher orders in some more detail.

\begin{remark} 
 The structure of the matrix $A$ as given in \eqref{eq:A} is very interesting: it ``pushes elements in vectors downwards." To facilitate the understanding of the following analysis, we show different exponents of $A$. Obviously, $A^0 = \Id$, the identiy matrix, and $A^1$ is given in \eqref{eq:A}. There holds, always given that $\km$ is large enough:
 \begin{align*} 
  A^2 = \left( \begin{matrix} 
              \EE & 0 & 0 & 0 & 0 & \cdots & 0& 0 & 0 \\
              0 & 0 & 0 & \EE & 0 & \cdots & 0& 0 & 0 \\
              0 & 0 & 0 & 0 & \EE & \cdots & 0& 0 & 0 \\
              \vdots & \vdots & \vdots & \vdots & \vdots & \ddots &\vdots & \vdots & \vdots \\
              0 & 0 & 0 & 0 & 0 & \cdots & \EE  & 0 & 0 \\
              0 & 0 & 0 & 0 & 0 & \cdots & 0 & \EE & 0 \\
              0 & 0 & 0 & 0 & 0 & \cdots & 0& 0 & \EE \\
              0 & 0 & 0 & 0 & 0 & \cdots & 0& 0 & \EE \\
              0 & 0 & 0 & 0 & 0 &  \cdots& 0 & 0 & \EE
             \end{matrix} \right), \qquad 
  A^3 = \left( \begin{matrix} 
              \EE & 0 & 0 & 0 & 0 & \cdots & 0& 0 & 0 \\
              0 & 0 & 0 & 0 & \EE & \cdots & 0& 0 & 0 \\
              0 & 0 & 0 & 0 & 0 & \cdots & 0& 0 & 0 \\
              \vdots & \vdots & \vdots & \vdots & \vdots & \ddots & \vdots & \vdots \\
              0 & 0 & 0 & 0 & 0 & \cdots & \EE& 0 & 0 \\
              0 & 0 & 0 & 0 & 0 & \cdots & 0& \EE & 0 \\
              0 & 0 & 0 & 0 & 0 & \cdots & 0& 0 & \EE \\
              0 & 0 & 0 & 0 & 0 & \cdots & 0& 0 & \EE \\
              0 & 0 & 0 & 0 & 0 & \cdots & 0& 0 & \EE \\
              0 & 0 & 0 & 0 & 0 &  \cdots & 0& 0 & \EE
             \end{matrix} \right), \qquad \ldots              
 \end{align*}
 Eventually, there is a limit matrix $A^{\infty}$, given as 
 \begin{align*} 
 A^{\infty} = \left( \begin{matrix} 
              \EE & 0 & 0 & 0 & 0 & \cdots & 0& 0 & 0 \\
              0 & 0 & 0 & 0 & 0 & \cdots & 0& 0 & \EE \\
              0 & 0 & 0 & 0 & 0 & \cdots & 0& 0 & \EE \\
              \vdots & \vdots & \vdots & \vdots & \vdots & \ddots & \vdots & \vdots \\
              0 & 0 & 0 & 0 & 0 & \cdots & 0& 0 & \EE \\
              0 & 0 & 0 & 0 & 0 & \cdots & 0& 0 & \EE \\
              0 & 0 & 0 & 0 & 0 & \cdots & 0& 0 & \EE \\
              0 & 0 & 0 & 0 & 0 & \cdots & 0& 0 & \EE \\
              0 & 0 & 0 & 0 & 0 & \cdots & 0& 0 & \EE \\
              0 & 0 & 0 & 0 & 0 &  \cdots & 0& 0 & \EE
             \end{matrix} \right).
 \end{align*}
 More formally, we have $A^j = A^{\infty}$ for all $j \geq \km -1$. This structure is important, because it significantly simplifies the convergence analysis. Essentially, it states that most contributions to the overall error come from higher corrections. Intuitively, this is desirable, because higher corrections are supposed to have smaller error contributions, at least asymptotically. 
\end{remark}

\begin{lemma}\label{la:epsil}
Subdivide $Z^n$ into $\km+1$ blocks of size $s$ each, similarly to La. \ref{la:en}. 
Define $\eps_k^n$, $0 \leq k \leq \km$ to be the infinity-norm of the $k-$th block. For $n = 0$, there is no error, and hence
\begin{align*} 
 \eps_k^0 = 0, \qquad 0 \leq k \leq \km. 
\end{align*}
Furthermore, for $k = 0$, the underlying consistency analysis is trivial, and hence $\eps_0^n$ can be bound by a positive constant $C$ times $\dt^2$, i.e., 
\begin{align*} 
 \eps_0^n \leq C \dt^2, \qquad 1 \leq n \leq N.
\end{align*}
The constant $C$ depends of course on the analytical solution, the fluxes $\Phi$, the initial conditions $w_0$ and the like, but it does not depend on $\dt$ or $n$. 
Furthermore, there is a constant $C$ such that there holds for $1 \leq k < \km$ and $1 \leq n \leq N$:
\begin{alignat}{2}
 \label{eq:epsnk}
 \eps^n_k    &\leq C \dt \left( \eps_k^n + \eps_{k-1}^n\right) + C \dt \sum_{j=1}^{n-1} \sum_{h=k}^{\km} \eps_h^{n-j} + C\dt^{\min\{2+k,q\}}, \\
 \label{eq:epsnkmax}
 \eps^n_{\km}&\leq C \dt \sum_{j=0}^{n-1} \left(\eps_{\km-1}^{n-j} + \eps_{\km}^{n-j}\right) + C\dt^{\min\{2+\km,q\}}.
\end{alignat}
\end{lemma}
\begin{remark} Before proving this lemma, we point out a couple of salient features:
\begin{itemize}
 \item Note that $\eps^n_{\km}$ has a global (i.e., over \emph{all} $n$) dependence on $\eps^n_{\km-1}$, while all the other $\eps^n_k$ do not have this kind of dependence on $\eps^n_{k-1}$. 
 \item We have chosen to take the same constant $C$ for both the expression of $\eps_0^n$ as well as for all the terms in the inequality for $\eps_k^n$. This can be done without loss of generality, as we take the largest of all these constants. 
\end{itemize}
 
\end{remark}
\begin{proof}
To prove \eqref{eq:epsnk}, we consider the analytical expression of $Z^n$ given in \eqref{eq:znexact}. Let us first treat the last term in \eqref{eq:znexact}, $\sum_{j=0}^{n-1} A^j E^{n-j}$. Due to La. \ref{la:en}, we know that there is a constant $C$ such that $E^{n-j,[k]} \leq C \dt^{\min\{2+k,q\}+1}$. There holds 
\begin{align*} 
 \left(\sum_{j=0}^{n-1} A^j E^{n-j}\right)^{[k]} \leq N C \dt^{\min\{2+k,q\}+1} \leq \Tend C \dt^{\min\{2+k,q\}}.
\end{align*}
By $(\cdot)^{[k]}$, we denote the $k-$the block of a vector, again, $0 \leq k \leq \km$. We have used the fact that $\dt = \frac{\Tend}{N}$ and that $A^j$ ``pushes elements downwards," so there will be no $\dt$ contributions of lower orders than the ones given here. This explains why we carefully chose the red terms in Alg. \ref{alg:mdode} the way we did.

In a similar way, we treat the other terms. They can all be handled alike, so we only consider the term 
\begin{align*} 
 &\dt \sum_{j=0}^{n-1} A^j B_I \left(\Phii(Y^{n-j}) - \Phii(W^{n-j)}\right) \\ = &\dt B_I \left(\Phii(Y^{n}) - \Phii(W^{n)}\right) + \dt \sum_{j=1}^{n-1} A^j B_I\left(\Phii(Y^{n-j}) - \Phii(W^{n-j)}\right).
\end{align*}
Note that the first term couples errors at correction level $k-1$ to those of $k$. Estimated, this gives, independent of whether $k=\km$ or not, a contribution of
\begin{align*} 
 C \dt \left(\eps_k^n + \eps_{k-1}^n\right).
\end{align*}
For $k < \km$, due to the peculiar structure of $A^j$, $A^jB_I$ remains a block-matrix; on the $k-$th block-row, only two blocks are filled.\footnote{Let us clarify what we mean by block-rows and block-columns: The matrices we are dealing with here are of size $s \cdot (\km+1) \times s \cdot (\km+1)$. They can hence be subdivided into $(\km+1)^2$ blocks of size $s\times s$. According to the notation in Alg. \ref{alg:mdode}, we begin counting by $k=0$. Block-row $k$ means hence in the big matrix the rows from $k(s+1) + 1$ to $(k+1)(s+1)$.}
The block-columns thereof are $j$ and $j+1$, with $j \geq k$. A generous estimate yields the contribution 
\begin{align*} 
 C \dt \sum_{j=1}^{n-1} \sum_{h=k}^{\km} \eps_h^{n-j}. 
\end{align*}
(In fact, for $j \geq \km -1$, the block columns are $\km-1$ and $\km$, so sharper estimates could be used.) 
For $k=\km$, the non-zero block-columns are $\km-1$ and $\km$. This then gives the contribution 
\begin{align*}
 C \dt \sum_{j=1}^{n-1} \left(\eps_{\km-1}^{n-j} + \eps_{\km}^{n-j}\right).
\end{align*}
%
%
%
%
%
%
This yields the desired result.

\qed
\end{proof}

We do not have a proof for the following statement, which is why we formulate it as a proposition. 
\begin{proposition}
\label{hyp:order}
From the results in La.~\ref{la:epsil}, we conjecture that the method is convergent with the $k-$dependent orders.  That is, we postulate that each of the correction steps satisfy
\begin{align*}
  e^n_k &= \O\left(\dt^{\min\{2+k,q\}}\right),   &\quad k &< \km,
\end{align*}
and that the final correction step satisfies
\begin{align*}
	e^n_{\km} &= \O\left(\dt^{\min\{1+\km,q\}}\right), &\quad k &= \km. 
\end{align*}
\end{proposition}

\begin{remark}
 Please note that the last correction step does not increase the order. This is different to the behavior of the algorithm in~\cite{SealSchuetz19}; it is a consequence of the fact that we have modified the values at time level $t^n$ to make the algorithm ready for parallelism. 
\end{remark}

\begin{remark}
 Although we only present numerical results that support Proposition~\ref{hyp:order} -- and we only know from Thm. \ref{thm:convergence} that the method is at-least second order convergent -- we can very well see the underlying structure from Eqs.~\eqref{eq:epsnk}-\eqref{eq:epsnkmax}: For $k < \km$, $\eps_k^n$ essentially possesses the contributions $\dt \eps_{k-1}^n$, $\dt \sum_{j=1}^{n-1}\sum_{h=k}^{\km} \eps_h^{n-j}$ and $\dt^{\min\{2+k,q\}}$. (We have neglected the constants here.) If we assume that for increasing $k$, the order in $\dt$ of $\eps_k^n$ would also increase or at least not decrease, the contribution $\dt \sum_{j=1}^{n-1}\sum_{h=k}^{\km} \eps_h^{n-j}$ scales as $\O(\eps_k^n)$. Furthermore, if the order between $\eps_{k}^n$ and $\eps_{k-1}^n$ differs by one, then $\dt \eps_{k-1}^n$ scales again as $\eps_k^n$. The final constant term $\dt^{\min\{2+k,q\}}$ then gives the overall order. 
\end{remark}

%% file: sec_parallel.tex
The method presented in Alg.~\ref{alg:mdode} has been carefully designed in such a way that it can be implemented in a parallel setting. The ideas we use are similar to the ones shown in~\cite{ChriOng10} -- and in fact, the much older ideas presented in~\cite{MirankerLiniger1967} -- and are roughly based on the fact that dependencies are chosen in such a way that pipelining becomes possible. The following observations are key to being able to parallelize this solver: 
\begin{itemize} 
 \item The predictor, i.e., the values $\algs w n 0 l$ that correspond to zeroth iteration number $k = 0$, do \emph{not} depend on the values of any other correction step.  They only depend on the final value $\algs w {n-1} 0 s$ obtained by the \emph{predictor} in the previous time step. 
 
 \item For $1 \leq k < \km$, the $(k)-$th corrector step, i.e., the values $\algs w n k l$, depend on the $(k-1)$-st corrector step at time level $n$, as well as on the next correction at the previous time step, $\algs w {n-1} {k+1} s$. 
 
 \item The $(\km)$-th corrector step, $\algs w n \km l$ depends only on the $(\km-1)$-st corrector step at time level $n$ as well as the final correction at the previous time level, $\algs w {n-1} {\km} s$.
\end{itemize}

These dependencies are illustrated in more detail on the left hand side of Fig.~\ref{fig:timeparallel}.  On the $x-$axis, the time instances $n$, $n+1, \ldots$, are indicated, while on the $y-$axis, the correction iterates $k$ and the predictor ($k=0$) are sketched. The circles at position $(n, k)$ correspond to the computation of $\algs w n k l$ for all $1 \leq l \leq s$. Circles with the same number on it can be computed in parallel, while those with a higher number have to wait for those with a lower number to finish. The arrows indicate direct dependencies, required for the calculation of $\algs{w}{n}{k}{l}$. 
Given the dependencies, it also makes sense to always group the correction iterates $2k$ and $2k+1$ on the same process, since one would otherwise create idle processor time. In addition, in order to use computational resources as efficiently as possible, we also do this for the predictor, so the predictor $(k=0)$ is grouped together with the first correction iterate $k=1$, which would strictly speaking not be necessary. 
Red arrows then imply communication over processes. Note that communication is always unidirectional. For comparison, the dependencies of the original serial algorithm from~\cite{SealSchuetz19} are visualized in the center of Fig.~\ref{fig:timeparallel}. One can clearly see that without the modifications proposed in Alg.~\ref{alg:mdode}, the predictor and all correction steps depend on the $(\km)$-th iteration at time level $n-1$.
\begin{figure}
	\centering
	\setlength{\tabcolsep}{-0.5em}
	\begin{tabular}{ccc}
	  \scriptsize HO-Parallel & \scriptsize Original & \scriptsize LO-Parallel\\[-0.5em]
  	\includegraphics[width=0.33\textwidth]{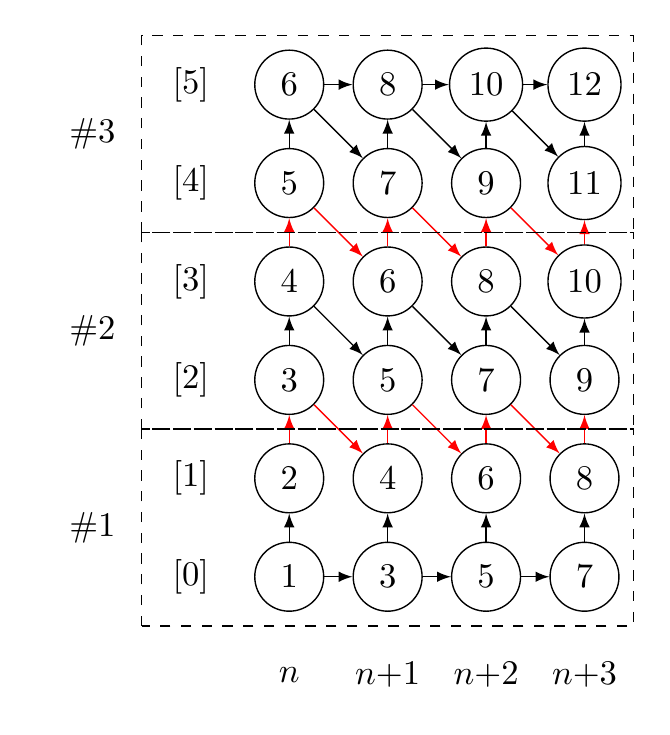}&
  	\includegraphics[width=0.33\textwidth]{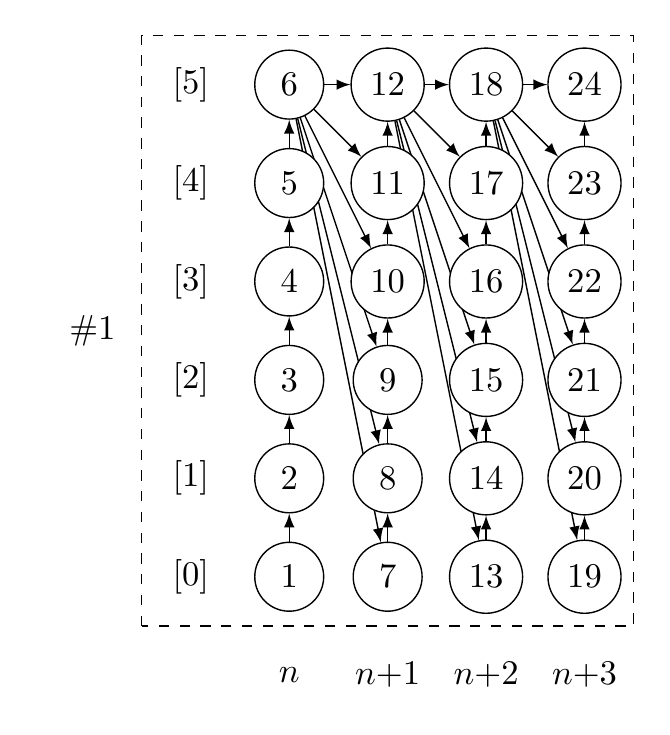}&
  	\includegraphics[width=0.33\textwidth]{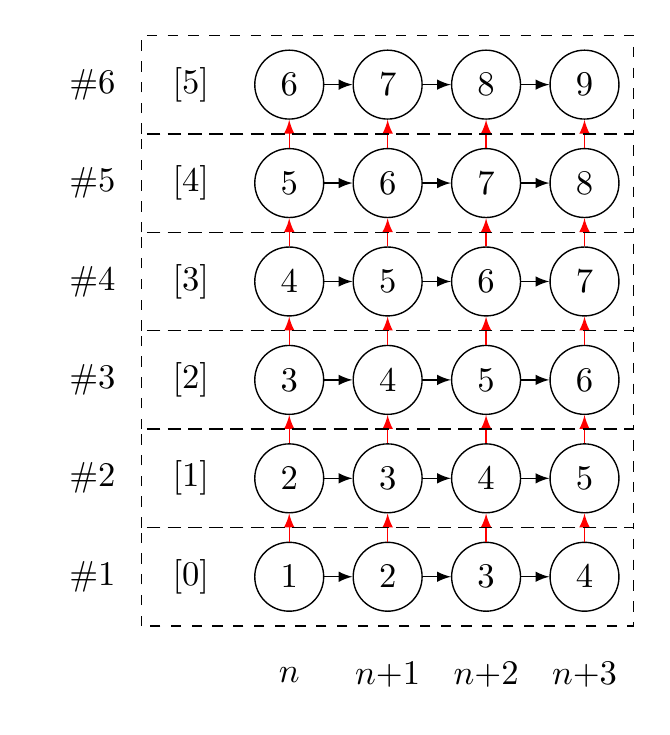}
	\end{tabular}
	\caption{Detailed view on the dependencies of the parallelization flavors. Here we include an illustrative example with a total of four time steps and $\km = 5$ total iterations.
	On the $x-$axis, time instances $n$, $n+1, \ldots$ are indicated, while on the $y-$axis, the correction iterates $k$ and the predictor ($k=0$) are sketched. The circles at position $(n, k)$ correspond to the computation of $\algs w n k l$ for all $1 \leq l \leq s$. Circles with the same number on it can be computed in parallel, while those with a higher number have to wait for those with a lower number to finish. The arrows indicate direct dependencies required for the calculation of $\algs{w}{n}{k}{l}$, with red arrows indicating that communication between processors is needed.
	Left: Higher order parallel-in-time MD scheme. 
	Middle: Original serial algorithm proposed in~\cite{SealSchuetz19}. Right: Low order parallel-in-time MD scheme, enabling the use of more processors. }\label{fig:timeparallel}
\end{figure}

The High-Order Parallel (HO-Parallel) version is designed to deliver parallel speedup while retaining high-order accuracy. This is the method described in Alg.~\ref{alg:mdode}, but here we describe one such grouping of processors that yields a parallel solver.  Theoretically, this speedup can be estimated as follows:
Denote the number of timesteps again by $N$. The unparallelized algorithm has to perform $N \cdot (\km+1)$ operation blocks (the computation of $\algs w n k l$ for given $n$, $k$ and for all $1 \leq l \leq s$). Each process in the parallelized version has to perform $2N$ operations, and the last process has to wait $\km - 1$ cycles before it can start. This gives, neglecting all other sources of imbalance and overhead such as communication and the like, a maximum theoretical speedup of
\begin{equation}
\label{eq:alg1-scaling}
	\frac{N \cdot (\km+1)}{2N + \km -1} \rightarrow \frac{\km+1}{2}, \qquad N \rightarrow \infty.
\end{equation}

Please note that this is the speedup in comparison to letting the algorithm run in its unparallelized version, it is not the theoretical speedup in comparison to the underlying Runge-Kutta method, as is done, e.g., in~\cite{EmmettMinion12}. This speedup is harder to determine, in particular in the IMEX setting that we follow here, because it depends on the convergence of the correction iterates as well as on the (nonlinear) algebraic solves, which can differ tremendously from the fully implicit Runge-Kutta method and the  semi-implicit correction iterates used here.

Grouping the processors has the obvious disadvantage that it reduces the parallelization capabilities, and therefore if attaining parallel speed is of chief concern, we suggest the Low-Order Parallel (LO-Parallel) method described in the right frame of Fig.~\ref{fig:timeparallel}.  In short, the basic idea is to dedicate an entire process to each correction step and let each correction run completely independent from the other corrections insofar as they need to wait for the previous correction to finish in a queue like fashion.  That is, the \textbf{LO-Parallel} version changes the red terms in Alg.~\ref{alg:mdode} from $\algs w {n-1} {k+2} s$ to $\algs w {n-1}{k+1} s$:
\begin{align}
\begin{split}
\label{eq:lo-parallel}
\algs w {n} {k+1} 1 &:= {\algs w {n-1} {k+1} s}, \\
\algs w {n} {k+1} l &:= {\algs w {n-1} {k+1} s}  \\ 
&
+ {\dt} \left( \algs {\Phii} {n} {k+1} l - \algs {\Phii} n {k} l\right)
- \frac{\dt^2}{2} \left(\algs {\dPhii} {n} {k+1} l- \algs {\dPhii} {n} {k} l\right)
\\ & + \II_l(\algs \Phi{n}{k} 0, \algs \Phi n {k} 1, \ldots, \algs \Phi n {k} s).
\end{split}
\end{align}
This means that on correction iterate $k+1$, we only take the previous $(k)-$th correction iterate into account, and not as in the novel algorithm, the $(k+2)$-nd iterate.  Provided we lag each of the corrections sufficiently through a startup procedure, then each of the correctors can follow up and clean up the previous iterate at the same cost of the predictor.  In comparison to the serial evaluation of Alg.~\ref{alg:mdode}, the speedup would be 
\begin{equation}
\frac{N \cdot (\km+1)}{N+(\km+1)} \to \km+1, \quad N \to \infty
\end{equation}
under ideal circumstances.  Unfortunately, this modfification reduces the overall method to a second-order method, and therefore we do not recommend this LO-Parallel strategy unless parallelization is of utmost importance.

%% file: sec_numerics.tex
In this section, we first experimentally validate the theoretical findings of Sec.~\ref{sec:ConvergenceAnalysis} with sample scalar ODE test cases. Here, we show that the desired orders of accuracy are reached after the predictor and each of the correction steps. We then increase the complexity of our test cases to systems and stiff problems. There we study the influence of choosing one of the algorithms presented in Fig.~\ref{fig:timeparallel}. It is shown that the modifications of the original algorithm from~\cite{SealSchuetz19} described in Alg.~\ref{alg:mdode} only have a small influence on the obtained solution. Finally, it is shown that the proposed method converges to the limiting Runge-Kutta method given in Def.~\ref{def:RKlimit} even for very stiff problems.

\subsection{ODE for Convergence Testing}
\pgfplotstableset{col sep=comma}
We start by considering a scalar model problem
\begin{align}\label{eq:ODE_conv}
	w'(t) = -w^{-\frac{5}{2}}, \qquad w_0 = 1,
\end{align}
to validate the order of convergence of the introduced methods.
The analytical solution of this ODE is $w(t)=\left(-\frac{7}{2}t+w_0^{7/2}\right)^{2/7}$; the error is evaluated at $\Tend=0.25$. For this test problem we introduce an artificial IMEX splitting via 
\begin{align*}
  \Phie(w)=\alpha w'(t), \qquad \Phii(w)=(1-\alpha) w'(t), 
\end{align*}
for $\alpha = 0.2$. 

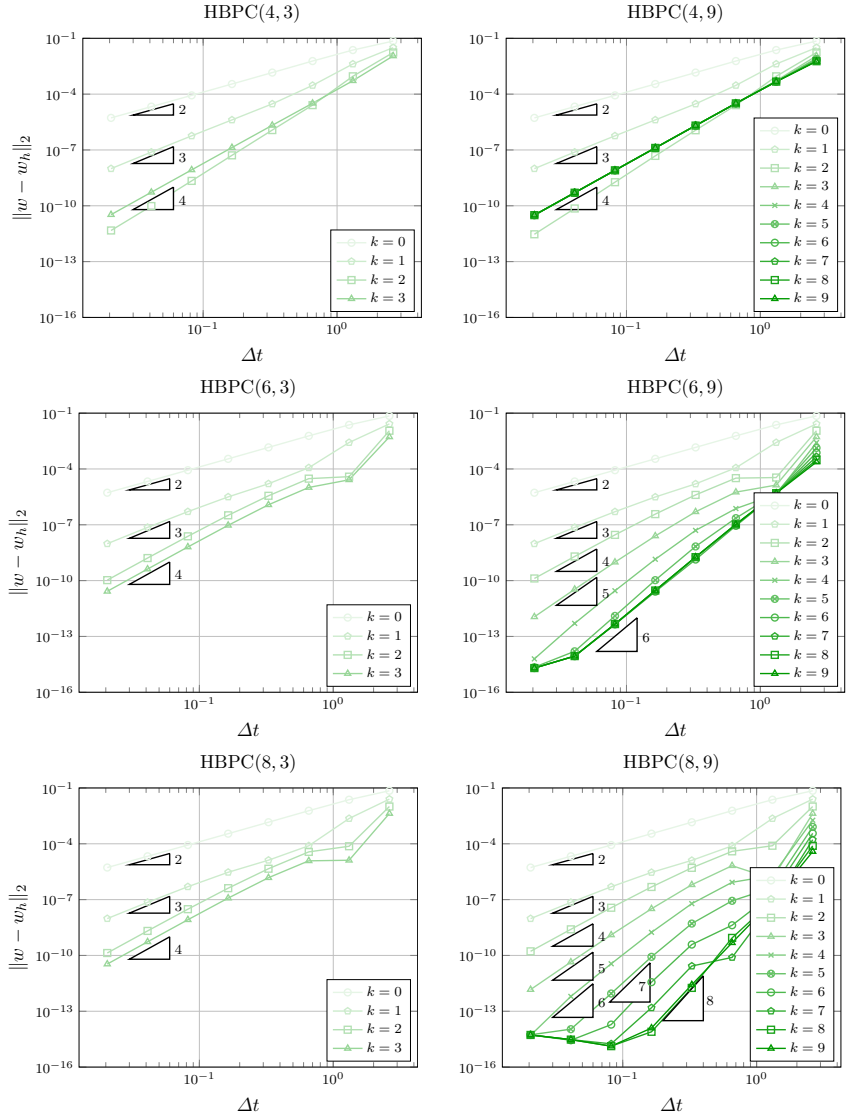
\begin{figure}
	\centering
	\begin{tabular}{cc}
	    \ifthenelse{\boolean{compilefromscratch}}
	    {
            \tikzsetnextfilename{rkrk4_kmax4_convtest_IMEX_eps1}
            \begin{tikzpicture}[scale=0.65]
            \begin{loglogaxis}[cycle list name=green,xlabel={$\Delta t$},ylabel={$\|w-w_h\|_2$} ,grid=major,legend style={at={(0.98,0.02)},anchor= south east,font=\footnotesize},title={$\method{4}{3}$},label style={font=\large},title style={font=\large},legend cell align={left},ymin=1e-16,ymax=1e-1]
            \addplot table[skip first n=1,x expr={10.5/\thisrowno{0}}, y expr={\thisrowno{1}}] {../Code/results/results_imexmd_parallel_rkrk4_kmax4_convtest_IMEX_eps1.csv};
            \addplot table[skip first n=1,x expr={10.5/\thisrowno{0}}, y expr={\thisrowno{2}}] {../Code/results/results_imexmd_parallel_rkrk4_kmax4_convtest_IMEX_eps1.csv};
            \addplot table[skip first n=1,x expr={10.5/\thisrowno{0}}, y expr={\thisrowno{3}}] {../Code/results/results_imexmd_parallel_rkrk4_kmax4_convtest_IMEX_eps1.csv};
            \addplot table[skip first n=1,x expr={10.5/\thisrowno{0}}, y expr={\thisrowno{4}}] {../Code/results/results_imexmd_parallel_rkrk4_kmax4_convtest_IMEX_eps1.csv};
            \addplot[color=black,thick,fill = white,forget plot] coordinates {(6e-2,3.0e-5) (6e-2,0.5^2*3.0e-5) (3e-2,0.5^2*3.0e-5) (6e-2,3.0e-5)}; \node at (axis cs: 0.07, 1.5e-5){$2$};
            \addplot[color=black,thick,fill = white,forget plot] coordinates {(6e-2,1.5e-7) (6e-2,0.5^3*1.5e-7) (3e-2,0.5^3*1.5e-7) (6e-2,1.5e-7)}; \node at (axis cs: 0.07, 4.0e-8){$3$};
            \addplot[color=black,thick,fill = white,forget plot] coordinates {(6e-2,1.0e-9) (6e-2,0.5^4*1.0e-9) (3e-2,0.5^4*1.0e-9) (6e-2,1.0e-9)}; \node at (axis cs: 0.07, 2.0e-10){$4$};
            \legend{$k=0$,$k=1$,$k=2$,$k=3$,$k=4$,$k=5$,$k=6$,$k=7$,$k=8$,$k=9$}
            \end{loglogaxis}
            \end{tikzpicture}&
            \tikzsetnextfilename{rkrk4_kmax10_convtest_IMEX_eps1}
            \begin{tikzpicture}[scale=0.65]
            \begin{loglogaxis}[cycle list name=green,xlabel={$\Delta t$},grid=major,legend style={at={(0.98,0.02)},anchor= south east,font=\footnotesize},title={$\method{4}{9}$},label style={font=\large},title style={font=\large},legend cell align={left},ymin=1e-16,ymax=1e-1]
            \addplot table[skip first n=1,x expr={10.5/\thisrowno{0}}, y expr={\thisrowno{1}}] {../Code/results/results_imexmd_parallel_rkrk4_kmax10_convtest_IMEX_eps1.csv};
            \addplot table[skip first n=1,x expr={10.5/\thisrowno{0}}, y expr={\thisrowno{2}}] {../Code/results/results_imexmd_parallel_rkrk4_kmax10_convtest_IMEX_eps1.csv};
            \addplot table[skip first n=1,x expr={10.5/\thisrowno{0}}, y expr={\thisrowno{3}}] {../Code/results/results_imexmd_parallel_rkrk4_kmax10_convtest_IMEX_eps1.csv};
            \addplot table[skip first n=1,x expr={10.5/\thisrowno{0}}, y expr={\thisrowno{4}}] {../Code/results/results_imexmd_parallel_rkrk4_kmax10_convtest_IMEX_eps1.csv};
            \addplot table[skip first n=1,x expr={10.5/\thisrowno{0}}, y expr={\thisrowno{5}}] {../Code/results/results_imexmd_parallel_rkrk4_kmax10_convtest_IMEX_eps1.csv};
            \addplot table[skip first n=1,x expr={10.5/\thisrowno{0}}, y expr={\thisrowno{6}}] {../Code/results/results_imexmd_parallel_rkrk4_kmax10_convtest_eps1.csv};
            \addplot table[skip first n=1,x expr={10.5/\thisrowno{0}}, y expr={\thisrowno{7}}] {../Code/results/results_imexmd_parallel_rkrk4_kmax10_convtest_IMEX_eps1.csv};
            \addplot table[skip first n=1,x expr={10.5/\thisrowno{0}}, y expr={\thisrowno{8}}] {../Code/results/results_imexmd_parallel_rkrk4_kmax10_convtest_IMEX_eps1.csv};
            \addplot table[skip first n=1,x expr={10.5/\thisrowno{0}}, y expr={\thisrowno{9}}] {../Code/results/results_imexmd_parallel_rkrk4_kmax10_convtest_IMEX_eps1.csv};
            \addplot table[skip first n=1,x expr={10.5/\thisrowno{0}}, y expr={\thisrowno{10}}] {../Code/results/results_imexmd_parallel_rkrk4_kmax10_convtest_IMEX_eps1.csv};
            \addplot[color=black,thick,fill = white,forget plot] coordinates {(6e-2,3.0e-5) (6e-2,0.5^2*3.0e-5) (3e-2,0.5^2*3.0e-5) (6e-2,3.0e-5)}; \node at (axis cs: 0.07, 1.5e-5){$2$};
            \addplot[color=black,thick,fill = white,forget plot] coordinates {(6e-2,1.5e-7) (6e-2,0.5^3*1.5e-7) (3e-2,0.5^3*1.5e-7) (6e-2,1.5e-7)}; \node at (axis cs: 0.07, 4.0e-8){$3$};
            \addplot[color=black,thick,fill = white,forget plot] coordinates {(6e-2,1.0e-9) (6e-2,0.5^4*1.0e-9) (3e-2,0.5^4*1.0e-9) (6e-2,1.0e-9)}; \node at (axis cs: 0.07, 2.0e-10){$4$};
            \legend{$k=0$,$k=1$,$k=2$,$k=3$,$k=4$,$k=5$,$k=6$,$k=7$,$k=8$,$k=9$}
            \end{loglogaxis}
            \end{tikzpicture}\\
            \tikzsetnextfilename{rkrk6_kmax4_convtest_IMEX_eps1}
            \begin{tikzpicture}[scale=0.65]
            \begin{loglogaxis}[cycle list name=green,xlabel={$\Delta t$},ylabel={$\|w-w_h\|_2$} ,grid=major,legend style={at={(0.98,0.02)},anchor= south east,font=\footnotesize},title={$\method{6}{3}$},label style={font=\large},title style={font=\large},legend cell align={left},ymin=1e-16,ymax=1e-1]
            \addplot table[skip first n=1,x expr={10.5/\thisrowno{0}}, y expr={\thisrowno{1}}] {../Code/results/results_imexmd_parallel_rkrk6_kmax4_convtest_IMEX_eps1.csv};
            \addplot table[skip first n=1,x expr={10.5/\thisrowno{0}}, y expr={\thisrowno{2}}] {../Code/results/results_imexmd_parallel_rkrk6_kmax4_convtest_IMEX_eps1.csv};
            \addplot table[skip first n=1,x expr={10.5/\thisrowno{0}}, y expr={\thisrowno{3}}] {../Code/results/results_imexmd_parallel_rkrk6_kmax4_convtest_IMEX_eps1.csv};
            \addplot table[skip first n=1,x expr={10.5/\thisrowno{0}}, y expr={\thisrowno{4}}] {../Code/results/results_imexmd_parallel_rkrk6_kmax4_convtest_IMEX_eps1.csv};
            \addplot[color=black,thick,fill = white,forget plot] coordinates {(6e-2,3.0e-5) (6e-2,0.5^2*3.0e-5) (3e-2,0.5^2*3.0e-5) (6e-2,3.0e-5)}; \node at (axis cs: 0.07, 1.5e-5){$2$};
            \addplot[color=black,thick,fill = white,forget plot] coordinates {(6e-2,1.5e-7) (6e-2,0.5^3*1.5e-7) (3e-2,0.5^3*1.5e-7) (6e-2,1.5e-7)}; \node at (axis cs: 0.07, 4.0e-8){$3$};
            \addplot[color=black,thick,fill = white,forget plot] coordinates {(6e-2,1.0e-9) (6e-2,0.5^4*1.0e-9) (3e-2,0.5^4*1.0e-9) (6e-2,1.0e-9)}; \node at (axis cs: 0.07, 2.0e-10){$4$};
            \legend{$k=0$,$k=1$,$k=2$,$k=3$,$k=4$,$k=5$,$k=6$,$k=7$,$k=8$,$k=9$}
            \end{loglogaxis}
            \end{tikzpicture}&
            \tikzsetnextfilename{rkrk6_kmax10_convtest_IMEX_eps1}
            \begin{tikzpicture}[scale=0.65]
            \begin{loglogaxis}[cycle list name=green,xlabel={$\Delta t$},grid=major,legend style={at={(0.98,0.02)},anchor= south east,font=\footnotesize},title={$\method{6}{9}$},label style={font=\large},title style={font=\large},legend cell align={left},ymin=1e-16,ymax=1e-1]
            \addplot table[skip first n=1,x expr={10.5/\thisrowno{0}}, y expr={\thisrowno{1}}] {../Code/results/results_imexmd_parallel_rkrk6_kmax10_convtest_IMEX_eps1.csv};
            \addplot table[skip first n=1,x expr={10.5/\thisrowno{0}}, y expr={\thisrowno{2}}] {../Code/results/results_imexmd_parallel_rkrk6_kmax10_convtest_IMEX_eps1.csv};
            \addplot table[skip first n=1,x expr={10.5/\thisrowno{0}}, y expr={\thisrowno{3}}] {../Code/results/results_imexmd_parallel_rkrk6_kmax10_convtest_IMEX_eps1.csv};
            \addplot table[skip first n=1,x expr={10.5/\thisrowno{0}}, y expr={\thisrowno{4}}] {../Code/results/results_imexmd_parallel_rkrk6_kmax10_convtest_IMEX_eps1.csv};
            \addplot table[skip first n=1,x expr={10.5/\thisrowno{0}}, y expr={\thisrowno{5}}] {../Code/results/results_imexmd_parallel_rkrk6_kmax10_convtest_IMEX_eps1.csv};
            \addplot table[skip first n=1,x expr={10.5/\thisrowno{0}}, y expr={\thisrowno{6}}] {../Code/results/results_imexmd_parallel_rkrk6_kmax10_convtest_IMEX_eps1.csv};
            \addplot table[skip first n=1,x expr={10.5/\thisrowno{0}}, y expr={\thisrowno{7}}] {../Code/results/results_imexmd_parallel_rkrk6_kmax10_convtest_IMEX_eps1.csv};
            \addplot table[skip first n=1,x expr={10.5/\thisrowno{0}}, y expr={\thisrowno{8}}] {../Code/results/results_imexmd_parallel_rkrk6_kmax10_convtest_IMEX_eps1.csv};
            \addplot table[skip first n=1,x expr={10.5/\thisrowno{0}}, y expr={\thisrowno{9}}] {../Code/results/results_imexmd_parallel_rkrk6_kmax10_convtest_IMEX_eps1.csv};
            \addplot table[skip first n=1,x expr={10.5/\thisrowno{0}}, y expr={\thisrowno{10}}] {../Code/results/results_imexmd_parallel_rkrk6_kmax10_convtest_IMEX_eps1.csv};
            \addplot[color=black,thick,fill = white,forget plot] coordinates {(6e-2,3.0e-5) (6e-2,0.5^2*3.0e-5) (3e-2,0.5^2*3.0e-5) (6e-2,3.0e-5)}; \node at (axis cs: 0.07, 1.5e-5){$2$};
            \addplot[color=black,thick,fill = white,forget plot] coordinates {(6e-2,1.5e-7) (6e-2,0.5^3*1.5e-7) (3e-2,0.5^3*1.5e-7) (6e-2,1.5e-7)}; \node at (axis cs: 0.07, 4.0e-8){$3$};
            \addplot[color=black,thick,fill = white,forget plot] coordinates {(6e-2,5.0e-9) (6e-2,0.5^4*5.0e-9) (3e-2,0.5^4*5.0e-9) (6e-2,5.0e-9)}; \node at (axis cs: 0.07, 1.0e-9){$4$};
            \addplot[color=black,thick,fill = white,forget plot] coordinates {(6e-2,1.5e-10) (6e-2,0.5^5*1.5e-10) (3e-2,0.5^5*1.5e-10) (6e-2,1.5e-10)}; \node at (axis cs: 0.07, 2.0e-11){$5$};
            \addplot[color=black,thick,fill = white,forget plot] coordinates {(1.2e-1,1.0e-12) (1.2e-1,0.5^6*1.0e-12) (6e-2,0.5^6*1.0e-12) (1.2e-1,1.0e-12)}; \node at (axis cs: 0.14, 9.0e-14){$6$};
            \legend{$k=0$,$k=1$,$k=2$,$k=3$,$k=4$,$k=5$,$k=6$,$k=7$,$k=8$,$k=9$}
            \end{loglogaxis}
            \end{tikzpicture}\\
            \tikzsetnextfilename{rkrk8_kmax4_convtest_IMEX_eps1}
            \begin{tikzpicture}[scale=0.65]
            \begin{loglogaxis}[cycle list name=green,xlabel={$\Delta t$},ylabel={$\|w-w_h\|_2$} ,grid=major,legend style={at={(0.98,0.02)},anchor= south east,font=\footnotesize},title={$\method{8}{3}$},label style={font=\large},title style={font=\large},legend cell align={left},ymin=1e-16,ymax=1e-1]
            \addplot table[skip first n=1,x expr={10.5/\thisrowno{0}}, y expr={\thisrowno{1}}] {../Code/results/results_imexmd_parallel_rkrk8_kmax4_convtest_IMEX_eps1.csv};
            \addplot table[skip first n=1,x expr={10.5/\thisrowno{0}}, y expr={\thisrowno{2}}] {../Code/results/results_imexmd_parallel_rkrk8_kmax4_convtest_IMEX_eps1.csv};
            \addplot table[skip first n=1,x expr={10.5/\thisrowno{0}}, y expr={\thisrowno{3}}] {../Code/results/results_imexmd_parallel_rkrk8_kmax4_convtest_IMEX_eps1.csv};
            \addplot table[skip first n=1,x expr={10.5/\thisrowno{0}}, y expr={\thisrowno{4}}] {../Code/results/results_imexmd_parallel_rkrk8_kmax4_convtest_IMEX_eps1.csv};
            \addplot[color=black,thick,fill = white,forget plot] coordinates {(6e-2,3.0e-5) (6e-2,0.5^2*3.0e-5) (3e-2,0.5^2*3.0e-5) (6e-2,3.0e-5)}; \node at (axis cs: 0.07, 1.5e-5){$2$};
            \addplot[color=black,thick,fill = white,forget plot] coordinates {(6e-2,1.5e-7) (6e-2,0.5^3*1.5e-7) (3e-2,0.5^3*1.5e-7) (6e-2,1.5e-7)}; \node at (axis cs: 0.07, 4.0e-8){$3$};
            \addplot[color=black,thick,fill = white,forget plot] coordinates {(6e-2,1.0e-9) (6e-2,0.5^4*1.0e-9) (3e-2,0.5^4*1.0e-9) (6e-2,1.0e-9)}; \node at (axis cs: 0.07, 2.0e-10){$4$};
            \legend{$k=0$,$k=1$,$k=2$,$k=3$,$k=4$,$k=5$,$k=6$,$k=7$,$k=8$,$k=9$}
            \end{loglogaxis}
            \end{tikzpicture}&
            \tikzsetnextfilename{rkrk8_kmax10_convtest_IMEX_eps1}
            \begin{tikzpicture}[scale=0.65]
            \begin{loglogaxis}[cycle list name=green,xlabel={$\Delta t$},grid=major,legend style={at={(0.98,0.02)},anchor= south east,font=\footnotesize},title={$\method{8}{9}$},label style={font=\large},title style={font=\large},legend cell align={left},ymin=1e-16,ymax=1e-1]
            \addplot table[skip first n=1,x expr={10.5/\thisrowno{0}}, y expr={\thisrowno{1}}] {../Code/results/results_imexmd_parallel_rkrk8_kmax10_convtest_IMEX_eps1.csv};
            \addplot table[skip first n=1,x expr={10.5/\thisrowno{0}}, y expr={\thisrowno{2}}] {../Code/results/results_imexmd_parallel_rkrk8_kmax10_convtest_IMEX_eps1.csv};
            \addplot table[skip first n=1,x expr={10.5/\thisrowno{0}}, y expr={\thisrowno{3}}] {../Code/results/results_imexmd_parallel_rkrk8_kmax10_convtest_IMEX_eps1.csv};
            \addplot table[skip first n=1,x expr={10.5/\thisrowno{0}}, y expr={\thisrowno{4}}] {../Code/results/results_imexmd_parallel_rkrk8_kmax10_convtest_IMEX_eps1.csv};
            \addplot table[skip first n=1,x expr={10.5/\thisrowno{0}}, y expr={\thisrowno{5}}] {../Code/results/results_imexmd_parallel_rkrk8_kmax10_convtest_IMEX_eps1.csv};
            \addplot table[skip first n=1,x expr={10.5/\thisrowno{0}}, y expr={\thisrowno{6}}] {../Code/results/results_imexmd_parallel_rkrk8_kmax10_convtest_IMEX_eps1.csv};
            \addplot table[skip first n=1,x expr={10.5/\thisrowno{0}}, y expr={\thisrowno{7}}] {../Code/results/results_imexmd_parallel_rkrk8_kmax10_convtest_IMEX_eps1.csv};
            \addplot table[skip first n=1,x expr={10.5/\thisrowno{0}}, y expr={\thisrowno{8}}] {../Code/results/results_imexmd_parallel_rkrk8_kmax10_convtest_IMEX_eps1.csv};
            \addplot table[skip first n=1,x expr={10.5/\thisrowno{0}}, y expr={\thisrowno{9}}] {../Code/results/results_imexmd_parallel_rkrk8_kmax10_convtest_IMEX_eps1.csv};
            \addplot table[skip first n=1,x expr={10.5/\thisrowno{0}}, y expr={\thisrowno{10}}] {../Code/results/results_imexmd_parallel_rkrk8_kmax10_convtest_IMEX_eps1.csv};
            \addplot[color=black,thick,fill = white,forget plot] coordinates {(6e-2,3.0e-5) (6e-2,0.5^2*3.0e-5) (3e-2,0.5^2*3.0e-5) (6e-2,3.0e-5)}; \node at (axis cs: 0.07, 1.5e-5){$2$};
            \addplot[color=black,thick,fill = white,forget plot] coordinates {(6e-2,1.5e-7) (6e-2,0.5^3*1.5e-7) (3e-2,0.5^3*1.5e-7) (6e-2,1.5e-7)}; \node at (axis cs: 0.07, 4.0e-8){$3$};
            \addplot[color=black,thick,fill = white,forget plot] coordinates {(6e-2,5.0e-9) (6e-2,0.5^4*5.0e-9) (3e-2,0.5^4*5.0e-9) (6e-2,5.0e-9)}; \node at (axis cs: 0.07, 1.0e-9){$4$};
            \addplot[color=black,thick,fill = white,forget plot] coordinates {(6e-2,1.5e-10) (6e-2,0.5^5*1.5e-10) (3e-2,0.5^5*1.5e-10) (6e-2,1.5e-10)}; \node at (axis cs: 0.07, 2.0e-11){$5$};
            \addplot[color=black,thick,fill = white,forget plot] coordinates {(6e-2,3.0e-12) (6e-2,0.5^6*3.0e-12) (3e-2,0.5^6*3.0e-12) (6e-2,3.0e-12)}; \node at (axis cs: 0.07, 3.0e-13){$6$};
            \addplot[color=black,thick,fill = white,forget plot] coordinates {(1.6e-1,4.0e-11) (1.6e-1,0.5^7*4.0e-11) (8e-2,0.5^7*4.0e-11) (1.6e-1,4.0e-11)}; \node at (axis cs: 0.14, 2.0e-12){$7$};
            \addplot[color=black,thick,fill = white,forget plot] coordinates {(4e-1,8.0e-12) (4e-1,0.5^8*8.0e-12) (2e-1,0.5^8*8.0e-12) (4e-1,8.0e-12)}; \node at (axis cs: 0.45, 4.0e-13){$8$};
            \legend{$k=0$,$k=1$,$k=2$,$k=3$,$k=4$,$k=5$,$k=6$,$k=7$,$k=8$,$k=9$}
            \end{loglogaxis}
            \end{tikzpicture}
        }
        {
          \includegraphics{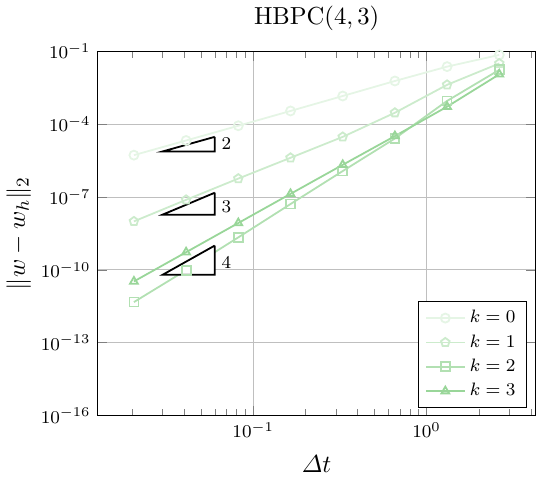}&
          \includegraphics{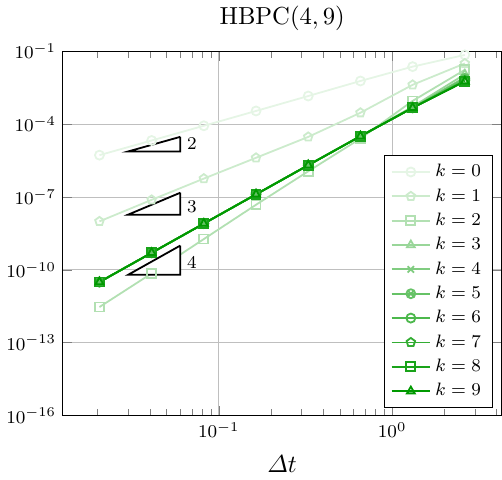}\\
          \includegraphics{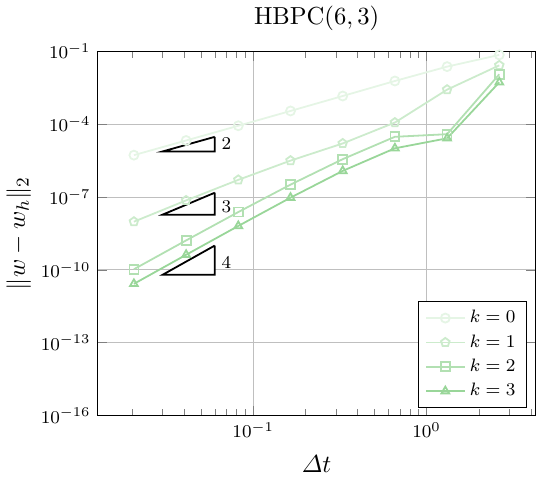}&
          \includegraphics{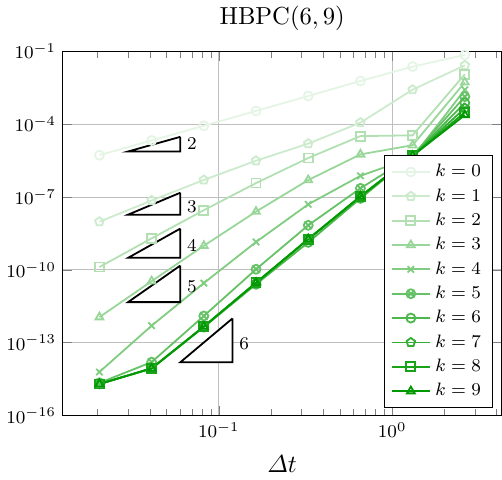}\\
          \includegraphics{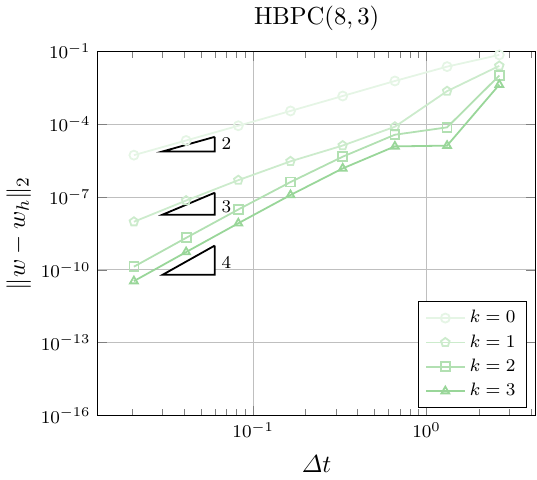}&
          \includegraphics{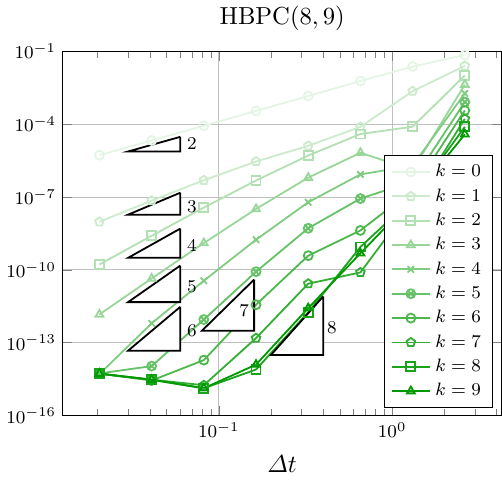}
        }
	\end{tabular}
	\caption{Error of the iterates of the parallel IMEX-MD schemes for the simple ODE problem $w' = -w^{-\frac5 2}$, see Eq.~\eqref{eq:ODE_conv}. The error is computed at $\Tend=0.25$. Different Runge-Kutta tables and $\km=3$ (left) and $\km=9$ (right) are used.}\label{fig:ConvTest}
\end{figure}

Fig.~\ref{fig:ConvTest} shows the convergence of the fourth-, sixth- and
eighth-order schemes after all correction steps. The figure shows that the
predictor step is second order accurate and the correction steps increase the
order of accuracy successively. With $\km=3$ fourth order of accuracy is
achievable and one can note that the last two correction steps are always of
the same order. Increasing $\km$ shows that with each correction step the
scheme picks up one order of accuracy - until the maximum order defined by the
quadrature rule is reached. Hence, the maximum achievable order is
$\min({\km+1},{\text{accuracy of quadrature rule}})$, see also Prop.~\ref{hyp:order}.

\subsection{Pareschi-Russo Problem}

Next, we consider the model problem introduced by Pareschi and Russo (PR)~\cite{pareschi2000implicit} to additionally consider the ability of treating stiff ODEs with the novel method. 
The system of ODEs is given by
\begin{align}\label{eq:PR_problem}
	 w_1'(t) &= -w_2, \qquad w_2'(t) = w_1 + \frac{\sin(w_1)-w_2}{\eps}, \qquad 
	 w_0 = \left(\frac\pi 2,1\right),
\end{align}
and is computed until $T_{\text{end}}=5$. To account for the stiff behavior, similar as it is done in~\cite{pareschi2000implicit}, we split the problem into a non-stiff, explicitly treated part $\Phie$ and a stiff, implicitly treated part $\Phii$
\begin{align*}
	\Phie=\begin{pmatrix}-w_2\\w_1\end{pmatrix},\qquad\Phii=\frac{1}{\eps}\begin{pmatrix}0\\\sin(w_1)-w_2\end{pmatrix}.
\end{align*}
We use this test problem to investigate the modifications of Alg.~\ref{alg:mdode} compared to the original one in~\cite{SealSchuetz19} and the straight-forward, but low-order parallelization variant. For an overview on the three different algorithms see Fig.~\ref{fig:timeparallel}.

\begin{figure}
	\centering
	\setlength{\tabcolsep}{0.1em}
	\begin{tabular}{cccc}
        \ifthenelse{\boolean{compilefromscratch}}{
        
        \tikzsetnextfilename{results_imexmd_serialrk4_kmax9_PR_IMEX_eps1}        
		\begin{tikzpicture}[scale=0.6]
		\begin{loglogaxis}[cycle list name=epslist,xlabel={$\Delta t$},ylabel={$\|w-w_h\|_2$} ,grid=major,legend style={at={(0.98,0.02)},anchor= south east,font=\footnotesize},
    title={ $\method{4}{9}$, $\varepsilon=10^0$},label style={font=\large},title style={font=\large},legend cell align={left},ymax=10,ymin=1e-14,x post scale=0.9]
		\addplot table[skip first n=1,x expr={5/\thisrowno{0}}, y expr={\thisrowno{1}}] {../Code/results/results_imexmd_serialrk4_kmax9_PR_IMEX_eps1.csv};
		\addplot table[skip first n=1,x expr={5/\thisrowno{0}}, y expr={\thisrowno{10}}] {../Code/results/results_imexmd_parallelrk4_kmax10_PR_IMEX_eps1.csv};
		\addplot table[skip first n=1,x expr={5/\thisrowno{0}}, y expr={\thisrowno{10}}] {../Code/results/results_imexmd_parallel_rkrk4_kmax10_PR_IMEX_eps1.csv};
		\addplot[color=black,thick,fill = white,forget plot] coordinates {(6e-2,1.5e-4) (6e-2,0.5^2*1.5e-4) (3e-2,0.5^2*1.5e-4) (6e-2,1.5e-4)}; \node at (axis cs: 0.075, 5.5e-5){$2$};
		\addplot[color=black,thick,fill = white,forget plot] coordinates {(6e-2,1.5e-8) (6e-2,0.5^4*1.5e-8) (3e-2,0.5^4*1.5e-8) (6e-2,1.5e-8)}; \node at (axis cs: 0.075, 3.5e-9){$4$};
		\legend{original, lo-parallel, ho-parallel}
		\end{loglogaxis}
		\end{tikzpicture}&
		\tikzsetnextfilename{results_imexmd_serialrk4_kmax9_PR_IMEX_eps0.01}
		\begin{tikzpicture}[scale=0.6]
		\begin{loglogaxis}[cycle list name=epslist,xlabel={$\Delta t$},grid=major,legend style={at={(0.98,0.02)},anchor= south east,font=\footnotesize},title={ $\method{4}{9}$, $\varepsilon=10^{-2}$},label style={font=\large},title style={font=\large},legend cell align={left},yticklabels={,,},ymax=10,ymin=1e-14,x post scale=0.9]
		\addplot table[skip first n=1,x expr={5/\thisrowno{0}}, y expr={\thisrowno{1}}] {../Code/results/results_imexmd_serialrk4_kmax9_PR_IMEX_eps0.01.csv};
		\addplot table[skip first n=1,x expr={5/\thisrowno{0}}, y expr={\thisrowno{10}}] {../Code/results/results_imexmd_parallelrk4_kmax10_PR_IMEX_eps0.01.csv};
		\addplot table[skip first n=1,x expr={5/\thisrowno{0}}, y expr={\thisrowno{10}}] {../Code/results/results_imexmd_parallel_rkrk4_kmax10_PR_IMEX_eps0.01.csv};
		\addplot[color=black,thick,fill = white,forget plot] coordinates {(6e-2,1.5e-3) (6e-2,0.5^2*1.5e-3) (3e-2,0.5^2*1.5e-3) (6e-2,1.5e-3)}; \node at (axis cs: 0.075, 5.5e-4){$2$};
		\addplot[color=black,thick,fill = white,forget plot] coordinates {(1.4e-1,1.5e-6) (1.4e-1,0.5^4*1.5e-6) (7e-2,0.5^4*1.5e-6) (1.4e-1,1.5e-6)}; \node at (axis cs: 0.16, 3.5e-7){$4$};
		\legend{original, lo-parallel, ho-parallel}
		\end{loglogaxis}
		\end{tikzpicture}&
		\tikzsetnextfilename{results_imexmd_serialrk4_kmax9_PR_IMEX_eps0.001}		
		\begin{tikzpicture}[scale=0.6]
		\begin{loglogaxis}[cycle list name=epslist,xlabel={$\Delta t$},grid=major,legend style={at={(0.98,0.02)},anchor= south east,font=\footnotesize},title={$\method{4}{9}$, $\varepsilon=10^{-3}$},label style={font=\large},title style={font=\large},legend cell align={left},yticklabels={,,},ymax=10,ymin=1e-14,x post scale=0.9]
		\addplot table[skip first n=1,x expr={5/\thisrowno{0}}, y expr={\thisrowno{1}}] {../Code/results/results_imexmd_serialrk4_kmax9_PR_IMEX_eps0.001.csv};
		\addplot table[skip first n=1,x expr={5/\thisrowno{0}}, y expr={\thisrowno{10}}] {../Code/results/results_imexmd_parallelrk4_kmax10_PR_IMEX_eps0.001.csv};
		\addplot table[skip first n=1,x expr={5/\thisrowno{0}}, y expr={\thisrowno{10}}] {../Code/results/results_imexmd_parallel_rkrk4_kmax10_PR_IMEX_eps0.001.csv};
		\addplot[color=black,thick,fill = white,forget plot] coordinates {(6e-2,1.5e-3) (6e-2,0.5^2*1.5e-3) (3e-2,0.5^2*1.5e-3) (6e-2,1.5e-3)}; \node at (axis cs: 0.075, 5.5e-4){$2$};
		\addplot[color=black,thick,fill = white,forget plot] coordinates {(2.2e-2,1.5e-7) (2.2e-2,0.5^4*1.5e-7) (1.1e-2,0.5^4*1.5e-7) (2.2e-2,1.5e-7)}; \node at (axis cs: 0.025, 3.5e-8){$4$};
		\legend{original, lo-parallel, ho-parallel}		
		\end{loglogaxis}
		\end{tikzpicture}\\
		\tikzsetnextfilename{results_imexmd_serialrk6_kmax9_PR_IMEX_eps1}	
		\begin{tikzpicture}[scale=0.6]
		\begin{loglogaxis}[cycle list name=epslist,xlabel={$\Delta t$},ylabel={$\|w-w_h\|_2$} ,grid=major,legend style={at={(0.98,0.02)},anchor= south east,font=\footnotesize},title={ $\method{6}{9}$, $\varepsilon=10^0$},label style={font=\large},title style={font=\large},legend cell align={left},ymax=10,ymin=1e-14,x post scale=0.9]
		\addplot table[skip first n=1,x expr={5/\thisrowno{0}}, y expr={\thisrowno{1}}] {../Code/results/results_imexmd_serialrk6_kmax9_PR_IMEX_eps1.csv};
		\addplot table[skip first n=1,x expr={5/\thisrowno{0}}, y expr={\thisrowno{10}}] {../Code/results/results_imexmd_parallelrk6_kmax10_PR_IMEX_eps1.csv};
		\addplot table[skip first n=1,x expr={5/\thisrowno{0}}, y expr={\thisrowno{10}}] {../Code/results/results_imexmd_parallel_rkrk6_kmax10_PR_IMEX_eps1.csv};
		\addplot[color=black,thick,fill = white,forget plot] coordinates {(6e-2,1.5e-4) (6e-2,0.5^2*1.5e-4) (3e-2,0.5^2*1.5e-4) (6e-2,1.5e-4)}; \node at (axis cs: 0.075, 5.5e-5){$2$};
		\addplot[color=black,thick,fill = white,forget plot] coordinates {(8e-1,6.5e-7) (8e-1,0.5^6*6.5e-7) (4.0e-1,0.5^6*6.5e-7) (8e-1,6.5e-7)}; \node at (axis cs: 0.95, 4.0e-8){$6$};
		\legend{original, lo-parallel, ho-parallel}
		\end{loglogaxis}
		\end{tikzpicture}&
		\tikzsetnextfilename{results_imexmd_serialrk6_kmax9_PR_IMEX_eps0.01}
		\begin{tikzpicture}[scale=0.6]
		\begin{loglogaxis}[cycle list name=epslist,xlabel={$\Delta t$},grid=major,legend style={at={(0.98,0.02)},anchor= south east,font=\footnotesize},title={ $\method{6}{9}$, $\varepsilon=10^{-2}$},label style={font=\large},title style={font=\large},legend cell align={left},yticklabels={,,},ymax=10,ymin=1e-14,x post scale=0.9]
		\addplot table[skip first n=1,x expr={5/\thisrowno{0}}, y expr={\thisrowno{1}}] {../Code/results/results_imexmd_serialrk6_kmax9_PR_IMEX_eps0.01.csv};
		\addplot table[skip first n=1,x expr={5/\thisrowno{0}}, y expr={\thisrowno{10}}] {../Code/results/results_imexmd_parallelrk6_kmax10_PR_IMEX_eps0.01.csv};
		\addplot table[skip first n=1,x expr={5/\thisrowno{0}}, y expr={\thisrowno{10}}] {../Code/results/results_imexmd_parallel_rkrk6_kmax10_PR_IMEX_eps0.01.csv};
		\addplot[color=black,thick,fill = white,forget plot] coordinates {(6e-2,1.5e-3) (6e-2,0.5^2*1.5e-3) (3e-2,0.5^2*1.5e-3) (6e-2,1.5e-3)}; \node at (axis cs: 0.075, 5.5e-4){$2$};
		\addplot[color=black,thick,fill = white,forget plot] coordinates {(4e-2,4.5e-9) (4e-2,0.5^3*4.5e-9) (2.0e-2,0.5^3*4.5e-9) (4e-2,4.5e-9)}; \node at (axis cs: 5e-2, 1.4e-9){$3$};
		\legend{original, lo-parallel, ho-parallel}
		\end{loglogaxis}
		\end{tikzpicture}&
		\tikzsetnextfilename{results_imexmd_serialrk6_kmax9_PR_IMEX_eps0.001}
		\begin{tikzpicture}[scale=0.6]
		\begin{loglogaxis}[cycle list name=epslist,xlabel={$\Delta t$},grid=major,legend style={at={(0.98,0.02)},anchor= south east,font=\footnotesize},title={ $\method{6}{9}$, $\varepsilon=10^{-3}$},label style={font=\large},title style={font=\large},legend cell align={left},yticklabels={,,},ymax=10,ymin=1e-14,x post scale=0.9]
		\addplot table[skip first n=1,x expr={5/\thisrowno{0}}, y expr={\thisrowno{1}}] {../Code/results/results_imexmd_serialrk6_kmax9_PR_IMEX_eps0.001.csv};
		\addplot table[skip first n=1,x expr={5/\thisrowno{0}}, y expr={\thisrowno{10}}] {../Code/results/results_imexmd_parallelrk6_kmax10_PR_IMEX_eps0.001.csv};
		\addplot table[skip first n=1,x expr={5/\thisrowno{0}}, y expr={\thisrowno{10}}] {../Code/results/results_imexmd_parallel_rkrk6_kmax10_PR_IMEX_eps0.001.csv};
		\addplot[color=black,thick,fill = white,forget plot] coordinates {(6e-2,1.5e-3) (6e-2,0.5^2*1.5e-3) (3e-2,0.5^2*1.5e-3) (6e-2,1.5e-3)}; \node at (axis cs: 0.075, 5.5e-4){$2$};
		\addplot[color=black,thick,fill = white,forget plot] coordinates {(6e-2,5.5e-7) (6e-2,0.5^3*5.5e-7) (3e-2,0.5^3*5.5e-7) (6e-2,5.5e-7)}; \node at (axis cs: 0.075, 1.5e-7){$3$};
		\addplot[color=black,thick,fill = white,forget plot] coordinates {(2.2e-2,4.5e-8) (2.2e-2,0.5^6*4.5e-8) (1.1e-2,0.5^6*4.5e-8) (2.2e-2,4.5e-8)}; \node at (axis cs: 0.025, 3.0e-9){$6$};
		\legend{original, lo-parallel, ho-parallel}
		\end{loglogaxis}
		\end{tikzpicture}\\
		
		\tikzsetnextfilename{results_imexmd_serialrk8_kmax9_PR_IMEX_eps1} 
		\begin{tikzpicture}[scale=0.6]
		\begin{loglogaxis}[cycle list name=epslist,xlabel={$\Delta t$},ylabel={$\|w-w_h\|_2$},grid=major,legend style={at={(0.98,0.02)},anchor= south east,font=\footnotesize},title={ $\method{8}{9}$, $\varepsilon=10^0$},label style={font=\large},title style={font=\large},legend cell align={left},ymax=10,ymin=1e-14,x post scale=0.9]
		\addplot table[skip first n=1,x expr={5/\thisrowno{0}}, y expr={\thisrowno{1}}] {../Code/results/results_imexmd_serialrk8_kmax9_PR_IMEX_eps1.csv};
		\addplot table[skip first n=1,x expr={5/\thisrowno{0}}, y expr={\thisrowno{10}}] {../Code/results/results_imexmd_parallelrk8_kmax10_PR_IMEX_eps1.csv};
		\addplot table[skip first n=1,x expr={5/\thisrowno{0}}, y expr={\thisrowno{10}}] {../Code/results/results_imexmd_parallel_rkrk8_kmax10_PR_IMEX_eps1.csv};
		\addplot[color=black,thick,fill = white,forget plot] coordinates {(6e-2,1.5e-4) (6e-2,0.5^2*1.5e-4) (3e-2,0.5^2*1.5e-4) (6e-2,1.5e-4)}; \node at (axis cs: 0.075, 5.5e-5){$2$};
		\addplot[color=black,thick,fill = white,forget plot] coordinates {(5.5e-1,1.5e-7) (5.5e-1,0.5^8*1.5e-7) (2.75e-1,0.5^8*1.5e-7) (5.5e-1,1.5e-7)}; \node at (axis cs: 0.48, 7.0e-9){$8$};
		\legend{original, lo-parallel, ho-parallel}
		\end{loglogaxis}
		\end{tikzpicture}&
		\tikzsetnextfilename{results_imexmd_serialrk8_kmax9_PR_IMEX_eps0.01}
		\begin{tikzpicture}[scale=0.6]
		\begin{loglogaxis}[cycle list name=epslist,xlabel={$\Delta t$},grid=major,legend style={at={(0.98,0.02)},anchor= south east,font=\footnotesize},title={ $\method{8}{9}$, $\varepsilon=10^{-2}$},label style={font=\large},title style={font=\large},legend cell align={left},yticklabels={,,},ymax=10,ymin=1e-14,x post scale=0.9]
		\addplot table[skip first n=1,x expr={5/\thisrowno{0}}, y expr={\thisrowno{1}}] {../Code/results/results_imexmd_serialrk8_kmax9_PR_IMEX_eps0.01.csv};
		\addplot table[skip first n=1,x expr={5/\thisrowno{0}}, y expr={\thisrowno{10}}] {../Code/results/results_imexmd_parallelrk8_kmax10_PR_IMEX_eps0.01.csv};
		\addplot table[skip first n=1,x expr={5/\thisrowno{0}}, y expr={\thisrowno{10}}] {../Code/results/results_imexmd_parallel_rkrk8_kmax10_PR_IMEX_eps0.01.csv};
		\addplot[color=black,thick,fill = white,forget plot] coordinates {(6e-2,1.5e-3) (6e-2,0.5^2*1.5e-3) (3e-2,0.5^2*1.5e-3) (6e-2,1.5e-3)}; \node at (axis cs: 0.075, 7.0e-4){$2$};
		\addplot[color=black,thick,fill = white,forget plot] coordinates {(5.5e-2,1.5e-8) (5.5e-2,0.5^3*1.5e-8) (2.75e-2,0.5^3*1.5e-8) (5.5e-2,1.5e-8)}; \node at (axis cs: 0.065, 5.0e-9){$3$};
		\legend{original, lo-parallel, ho-parallel}
		\end{loglogaxis}
		\end{tikzpicture}&
		\tikzsetnextfilename{results_imexmd_serialrk8_kmax9_PR_IMEX_eps0.001}
		\begin{tikzpicture}[scale=0.6]
		\begin{loglogaxis}[cycle list name=epslist,xlabel={$\Delta t$},grid=major,legend style={at={(0.98,0.02)},anchor= south east,font=\footnotesize},title={ $\method{8}{9}$, $\varepsilon=10^{-3}$},label style={font=\large},title style={font=\large},legend cell align={left},yticklabels={,,},ymax=10,ymin=1e-14,x post scale=0.9]
		\addplot table[skip first n=1,x expr={5/\thisrowno{0}}, y expr={\thisrowno{1}}] {../Code/results/results_imexmd_serialrk8_kmax9_PR_IMEX_eps0.001.csv};
		\addplot table[skip first n=1,x expr={5/\thisrowno{0}}, y expr={\thisrowno{10}}] {../Code/results/results_imexmd_parallelrk8_kmax10_PR_IMEX_eps0.001.csv};
		\addplot table[skip first n=1,x expr={5/\thisrowno{0}}, y expr={\thisrowno{10}}] {../Code/results/results_imexmd_parallel_rkrk8_kmax10_PR_IMEX_eps0.001.csv};
		\addplot[color=black,thick,fill = white,forget plot] coordinates {(6e-2,5.5e-4) (6e-2,0.5^2*5.5e-4) (3e-2,0.5^2*5.5e-4) (6e-2,5.5e-4)}; \node at (axis cs: 0.075, 1.5e-4){$2$};
		\addplot[color=black,thick,fill = white,forget plot] coordinates {(5e-2,1.0e-7) (5e-2,0.5^6*1.0e-7) (2.5e-2,0.5^6*1.0e-7) (5e-2,1.0e-7)}; \node at (axis cs: 0.06, 7.0e-9){$6$};
		\legend{original, lo-parallel, ho-parallel}		
		\end{loglogaxis}
		\end{tikzpicture}
		}
		{
    \includegraphics{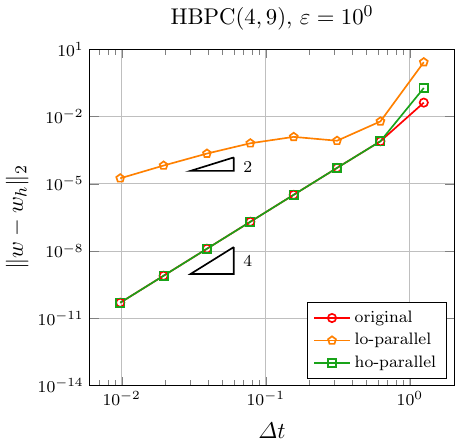}&
		\includegraphics{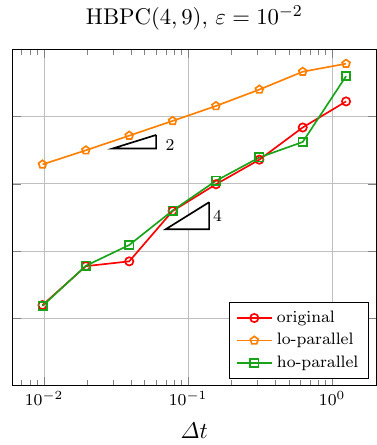}&
		\includegraphics{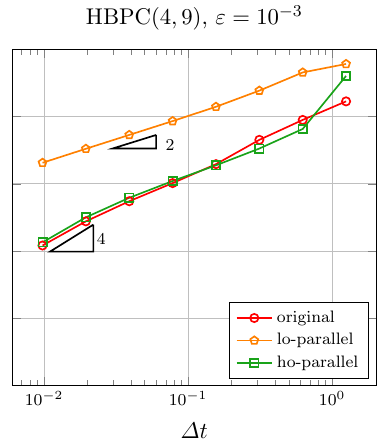}		\\
		\includegraphics{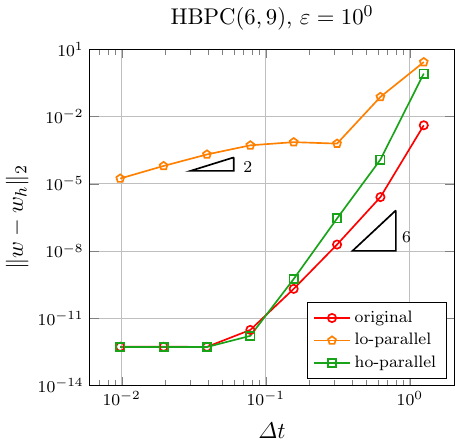}	&	
		\includegraphics{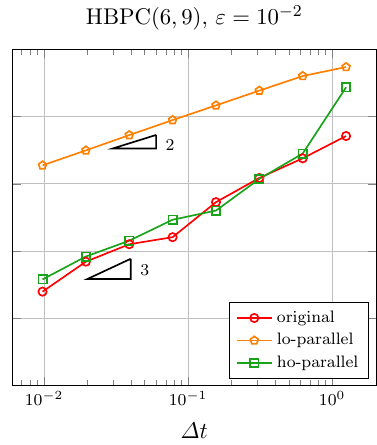}	&	
		\includegraphics{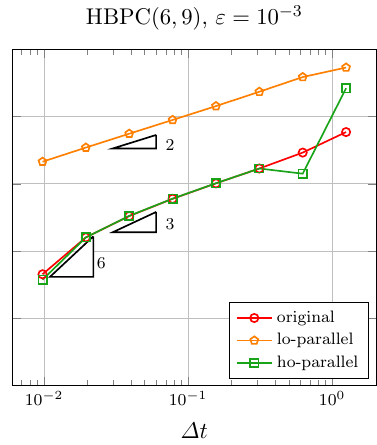} \\
		\includegraphics{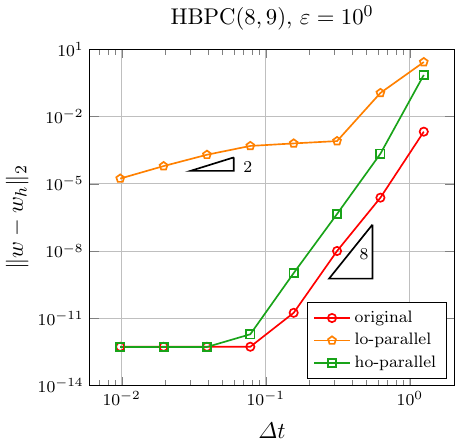} 	&	
		\includegraphics{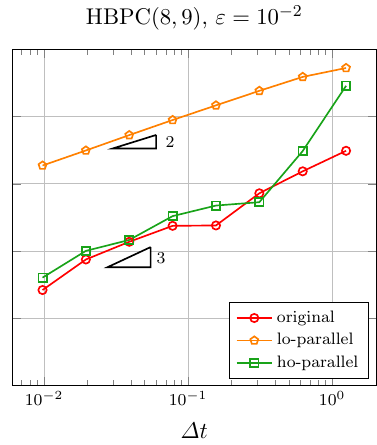}	&	
		\includegraphics{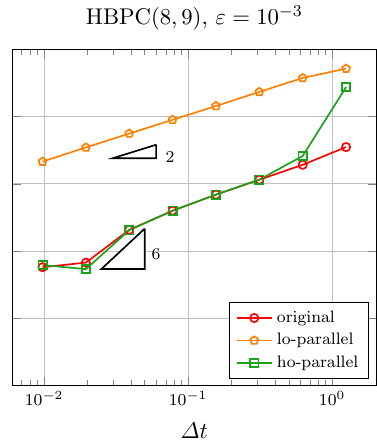}		
		}
	\end{tabular}
	\caption{Error for Pareschi-Russo IMEX problem~\cite{pareschi2000implicit}, see Eq.~\eqref{eq:PR_problem}, at $\Tend=5$ with $4^{th}$ (top), $6^{th}$ (middle) and $8^{th}$ (bottom) order IMEX-MD scheme with $\km=9$ with different stiffness parameters $\eps\in\{10^0,10^{-2},10^{-3}\}$.}\label{fig:PR_comparison}
\end{figure}

In order to investigate the different behaviors for stiff and non-stiff problems, $\eps$ in Eq.~\eqref{eq:PR_problem} is varied from $\eps=1$ to $\eps=10^{-3}$. The comparison is done for all three quadrature rules described in Eq.~\eqref{eq:butcher4}--\eqref{eq:butcher8}. A fixed $\km=9$ is used.
In Fig.~\ref{fig:PR_comparison} the results of the calculations are summarized. One can clearly see that the parallel low-order scheme is always second order accurate, regardless of the equipped quadrature rule and stiffness of the problem (orange lines). Additionally, the error of the parallel-low order scheme is always higher than with the other schemes. Considering the high-order parallel and the original serial scheme, the figure shows that the modifications introduced in Alg.~\ref{alg:mdode} compared to the original algorithm only slightly influence the obtained error and have no effect on the achieved order (compare red and green lines). Nevertheless, some differences between Alg.~\ref{alg:mdode} and the original algorithm are present, mostly showing up for large timesteps and non-stiff problems. Here, the original algorithm gives slightly better results. This is most probably due to the very few timesteps used for large $\Delta t$ (the coarsest resolution uses only four timesteps). Differently to the original algorithm, the information of the last corrector step has not yet reached the first corrector step in those settings, see Fig.~\ref{fig:timeparallel} for illustration. This reduces the accuracy of the algorithm if only a few timesteps are performed. Moreover, in the original algorithm a better $w^n$ is used for all steps $k\ne\km$ what naturally has a favorable influence on the solution.

One can see that while for the fourth order scheme no order reduction can be observed for stiff problems, $\method{6}{\km}$ and $\method{8}{\km}$ show order reduction for an increasing stiffness. In~\cite{SealSchuetz19} it has been shown, that increasing the number of correction steps can cure this problem. Therefore, we perform a convergence study with respect to the number of correction steps $\km$ for the stiffest setting $\eps=10^{-3}$ with $\method{6}{\km}$ and $\method{8}{\km}$.
Fig.~\ref{fig:PR_kmax} shows how the increase of corrector steps improves the accuracy of the obtained results. It is shown numerically that for an increasing $\km$, the predictor-corrector scheme described in Alg.~\ref{alg:mdode} converges towards the limit method, i.e. the fully coupled two-derivative Runge-Kutta method, see Def.~\ref{def:RKlimit}.

Summarizing, the presented simulations of the PR problem show that the modifications of the predictor-corrector scheme from~\cite{SealSchuetz19} described in Alg.~\ref{alg:mdode} have only a slight influence on the solution quality. Using the straightforward possibility to parallelize the scheme (Fig.~\ref{fig:timeparallel} right) is clearly inferior to the original method and Alg.~\ref{alg:mdode} even for stiff problems. Moreover, for the high order methods, it gets obvious that for stiff problems it can be beneficial to increase the number of corrector steps $\km$.
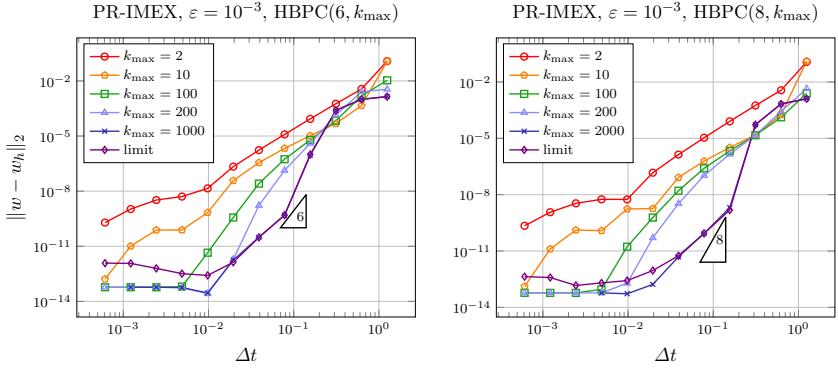
\begin{figure}
	\centering
	\begin{tabular}{cc}
        \ifthenelse{\boolean{compilefromscratch}}{
        
        \tikzsetnextfilename{results_imexmd_serialrk6_kmax2_PR_IMEX_eps0.001}
		\begin{tikzpicture}[scale=0.65]
		\begin{loglogaxis}[cycle list name=epslist,xlabel={$\Delta t$},ylabel={$\|w-w_h\|_2$} ,grid=major,legend style={at={(0.02,0.98)},anchor= north west,font=\footnotesize},title={PR-IMEX, $\varepsilon=10^{-3}$, $\method{6}{\km}$},label style={font=\large},title style={font=\large},legend cell align={left}]
		\addplot table[skip first n=1,x expr={5/\thisrowno{0}}, y expr={\thisrowno{1}}] {../Code/results/results_imexmd_serialrk6_kmax2_PR_IMEX_eps0.001.csv};
		\addplot table[skip first n=1,x expr={5/\thisrowno{0}}, y expr={\thisrowno{1}}] {../Code/results/results_imexmd_serialrk6_kmax10_PR_IMEX_eps0.001.csv};
		\addplot table[skip first n=1,x expr={5/\thisrowno{0}}, y expr={\thisrowno{1}}] {../Code/results/results_imexmd_serialrk6_kmax100_PR_IMEX_eps0.001.csv};
		\addplot table[skip first n=1,x expr={5/\thisrowno{0}}, y expr={\thisrowno{1}}] {../Code/results/results_imexmd_serialrk6_kmax200_PR_IMEX_eps0.001.csv};
		\addplot table[skip first n=1,x expr={5/\thisrowno{0}}, y expr={\thisrowno{1}}] {../Code/results/results_imexmd_serialrk6_kmax1000_PR_IMEX_eps0.001.csv};
		\addplot table[skip first n=1,x expr={5/\thisrowno{0}}, y expr={\thisrowno{1}}] {../Code/results/RungeKutta/results_imexmd_rkrk6_kmax1_PR_eps0.001.csv};
		\addplot[color=black,thick,fill = white,forget plot] coordinates {(1.4e-1,6.5e-9) (1.4e-1,0.5^6*6.5e-9) (7e-2,0.5^6*6.5e-9) (1.4e-1,6.5e-9)}; \node at (axis cs: 0.12, 4.0e-10){$6$};
		\legend{$\km=2$,$\km=10$,$\km=100$,$\km=200$,$\km=1000$,limit}
		\end{loglogaxis}
		\end{tikzpicture}&
		\tikzsetnextfilename{results_imexmd_serialrk8_kmax2_PR_IMEX_eps0.001}
		\begin{tikzpicture}[scale=0.65]
		\begin{loglogaxis}[cycle list name=epslist,xlabel={$\Delta t$},grid=major,legend style={at={(0.02,0.98)},anchor= north west,font=\footnotesize},title={PR-IMEX, $\varepsilon=10^{-3}$, $\method{8}{\km}$},label style={font=\large},title style={font=\large},legend cell align={left}]
		\addplot table[skip first n=1,x expr={5/\thisrowno{0}}, y expr={\thisrowno{1}}] {../Code/results/results_imexmd_serialrk8_kmax2_PR_IMEX_eps0.001.csv};
		\addplot table[skip first n=1,x expr={5/\thisrowno{0}}, y expr={\thisrowno{1}}] {../Code/results/results_imexmd_serialrk8_kmax10_PR_IMEX_eps0.001.csv};
		\addplot table[skip first n=1,x expr={5/\thisrowno{0}}, y expr={\thisrowno{1}}] {../Code/results/results_imexmd_serialrk8_kmax100_PR_IMEX_eps0.001.csv};
		\addplot table[skip first n=1,x expr={5/\thisrowno{0}}, y expr={\thisrowno{1}}] {../Code/results/results_imexmd_serialrk8_kmax200_PR_IMEX_eps0.001.csv};
		\addplot table[skip first n=1,x expr={5/\thisrowno{0}}, y expr={\thisrowno{1}}] {../Code/results/results_imexmd_serialrk8_kmax2000_PR_IMEX_eps0.001.csv};
		\addplot table[skip first n=1,x expr={5/\thisrowno{0}}, y expr={\thisrowno{1}}] {../Code/results/RungeKutta/results_imexmd_rkrk8_kmax1_PR_eps0.001.csv};
		\addplot[color=black,thick,fill = white,forget plot] coordinates {(1.4e-1,6.5e-10) (1.4e-1,0.5^8*6.5e-10) (7e-2,0.5^8*6.5e-10) (1.4e-1,6.5e-10)}; \node at (axis cs: 0.12, 4.0e-11){$8$};
		\legend{$\km=2$,$\km=10$,$\km=100$,$\km=200$,$\km=2000$,limit}
		\end{loglogaxis}
		\end{tikzpicture}
		}
		{
        \includegraphics{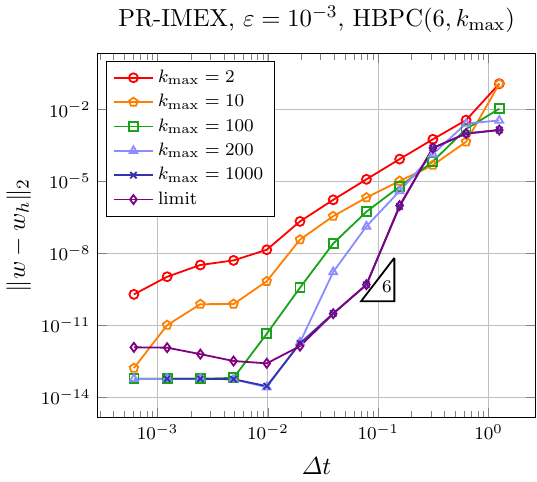}&
		\includegraphics{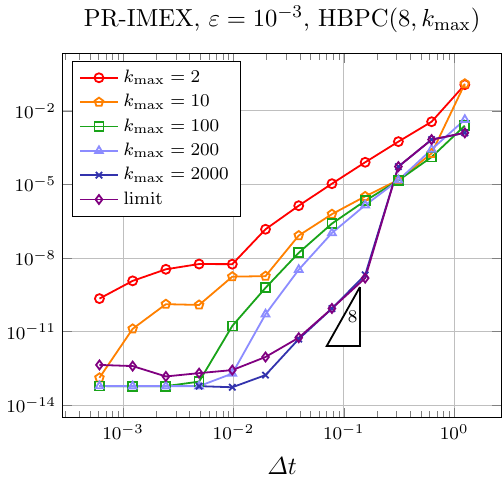}
		}
	\end{tabular}
	\caption{Influence of $\km$ iterations on error of PR-IMEX problem with $\varepsilon=10^{-3}$ for Alg.~\ref{alg:mdode} with $6^{th}$ order (left) and $8^{th}$ order (right) quadrature formula. Limit denotes the corresponding sixth- or eigth-order Runge-Kutta method as given in Def.~\ref{def:RKlimit}. 
	}\label{fig:PR_kmax}
\end{figure}

\subsection{Van-der-Pol Equation}
The van-der-Pol equation (vdP) allows us to study the convergence properties for very stiff problems in more detail. It is given by
\begin{align}\label{eq:Vdp}
	w_1'(t) &= w_2, \qquad w_2'(t) = \frac{(1-w_1^2)w_2 - w_1}{\eps}, \qquad w_0 = \left(2, \frac{-2}{3} + \frac{10}{81} \eps \right),
\end{align}
with $T_{\text{end}}=0.5$. Again, we split the equation into a non-stiff explicitly treated part $\Phie$ and a stiff, implicitly treated part $\Phii$ via
\begin{align*}
	\Phie=\begin{pmatrix}w_2\\0\end{pmatrix},\qquad\Phii=\frac{1}{\eps}\begin{pmatrix}0\\(1-w_1^2)w_2-w_1\end{pmatrix}.
\end{align*}
We use this test problem to study the convergence properties with a typical number of correction steps one would use for practical applications for stiff and non-stiff problems. The calculations are performed on one Intel skylake node with $36$ cores. We compare calculations with using all processors on the node ($\km=71$), to half the number of the processors on the node ($\km=35$).  We choose the minimum number of iterations to obtain the desired order for non-stiff problems, ($\km=3$, $\km=5$ and $\km=7$, respectively for the $4^{th}$, $6^{th}$ and $8^{th}$ order method).
Similar to the previously performed evaluations with the PR test problem, we compare the obtained results with $\km\rightarrow\infty$ and the limiting method from Def.~\ref{def:RKlimit}. 

In order to compute the limiting method as $\km \rightarrow \infty$, we compare the numerical errors w.r.t. the exact solution of simulations with $\km$ and $\km/2$. If the relative difference between those two errors differs by more than $1\%$, $\km$ is increased by a factor of two and the calculation is repeated. That means, repeat the simulation with $2\km$ if
\begin{align*}
	\frac{\left|\left(\|w-w_h\|_2\right)_{\km}-\left(\|w-w_h\|_2|\right)_{\km/2}\right|}{\left(\|w-w_h\|_2\right)_{\km}}>0.01,
\end{align*}
else use the solution with corresponding $\km$ as the ``converged'' solution. With this investigation we can experimentally show that our method converges towards the limiting Runge-Kutta method, i.e. the algorithm does not diverge.
\begin{remark}
	Note that a similar strategy can be pursued to obtain an adaptation criterion for the IMEX-MD scheme. Since the exact solution is not known for general problems, one could think about evaluating the change of the numerical solution after each iteration. If the iterates do not differ more than a user-defined threshold, no further correction steps are performed. Also timestep adaptation could be taken into account. We leave the detailed analysis of an adaptation of Alg.~\ref{alg:mdode} in this direction to future investigations.
\end{remark}

\begin{figure}
	\centering
	\ifthenelse{\boolean{compilefromscratch}}{
	\begin{tabular}{cccc}
        \tikzsetnextfilename{limit1}
		\begin{tikzpicture}[scale=0.5]
		\begin{loglogaxis}[cycle list name=epslist,xlabel={$\Delta t$},ylabel={$\|w-w_h\|_2$} ,grid=major,legend style={at={(0.02,0.98)},anchor= north west,font=\footnotesize},title={ $\method{6}{5}$},label style={font=\large},title style={font=\large},legend cell align={left},ymin=1e-13,ymax=1e-3]
		\addplot table[skip first n=1,x expr={0.5/\thisrowno{0}}, y expr={\thisrowno{6}}] {../Code/results/results_imexmd_parallel_rkrk6_kmax6_vdP_IMEX_eps0.1.csv};
		\addplot table[skip first n=1,x expr={0.5/\thisrowno{0}}, y expr={\thisrowno{6}}] {../Code/results/results_imexmd_parallel_rkrk6_kmax6_vdP_IMEX_eps0.01.csv};
		\addplot table[skip first n=1,x expr={0.5/\thisrowno{0}}, y expr={\thisrowno{6}}] {../Code/results/results_imexmd_parallel_rkrk6_kmax6_vdP_IMEX_eps0.001.csv};
		\addplot table[skip first n=1,x expr={0.5/\thisrowno{0}}, y expr={\thisrowno{6}}] {../Code/results/results_imexmd_parallel_rkrk6_kmax6_vdP_IMEX_eps0.0001.csv};
		\addplot table[skip first n=1,x expr={0.5/\thisrowno{0}}, y expr={\thisrowno{6}}] {../Code/results/results_imexmd_parallel_rkrk6_kmax6_vdP_IMEX_eps1e-05.csv};
		\addplot[color=black,thick,fill = white,forget plot] coordinates {(6e-2,1.0e-9) (6e-2,0.5^6*1.0e-9) (3e-2,0.5^6*1.0e-9) (6e-2,1.0e-9)}; \node at (axis cs: 0.075, 1.5e-10){$6$};
		\legend{$\eps=10^{-1}$,$\eps=10^{-2}$,$\eps=10^{-3}$,$\eps=10^{-4}$,$\eps=10^{-5}$}
		\end{loglogaxis}
		\end{tikzpicture}&
		\tikzsetnextfilename{limit2}
		\begin{tikzpicture}[scale=0.5]
		\begin{loglogaxis}[cycle list name=epslist,xlabel={$\Delta t$},grid=major,legend style={at={(0.02,0.98)},anchor= north west,font=\footnotesize},title={ $\method{6}{35}$},label style={font=\large},title style={font=\large},legend cell align={left},ymin=1e-13,ymax=1e-3]
		\addplot table[skip first n=1,x expr={0.5/\thisrowno{0}}, y expr={\thisrowno{36}}] {../Code/results/results_imexmd_parallel_rkrk6_kmax36_vdP_IMEX_eps0.1.csv};
		\addplot table[skip first n=1,x expr={0.5/\thisrowno{0}}, y expr={\thisrowno{36}}] {../Code/results/results_imexmd_parallel_rkrk6_kmax36_vdP_IMEX_eps0.01.csv};
		\addplot table[skip first n=1,x expr={0.5/\thisrowno{0}}, y expr={\thisrowno{36}}] {../Code/results/results_imexmd_parallel_rkrk6_kmax36_vdP_IMEX_eps0.001.csv};
		\addplot table[skip first n=1,x expr={0.5/\thisrowno{0}}, y expr={\thisrowno{36}}] {../Code/results/results_imexmd_parallel_rkrk6_kmax36_vdP_IMEX_eps0.0001.csv};
		\addplot table[skip first n=1,x expr={0.5/\thisrowno{0}}, y expr={\thisrowno{36}}] {../Code/results/results_imexmd_parallel_rkrk6_kmax36_vdP_IMEX_eps1e-05.csv};
		\addplot[color=black,thick,fill = white,forget plot] coordinates {(6e-2,1.0e-9) (6e-2,0.5^6*1.0e-9) (3e-2,0.5^6*1.0e-9) (6e-2,1.0e-9)}; \node at (axis cs: 0.075, 1.5e-10){$6$};
		\legend{$\eps=10^{-1}$,$\eps=10^{-2}$,$\eps=10^{-3}$,$\eps=10^{-4}$,$\eps=10^{-5}$}
		\end{loglogaxis}
		\end{tikzpicture}&
		\tikzsetnextfilename{limit3}
		\begin{tikzpicture}[scale=0.5]
		\begin{loglogaxis}[cycle list name=epslist,xlabel={$\Delta t$},grid=major,legend style={at={(0.02,0.98)},anchor= north west,font=\footnotesize},title={ $\method{6}{71}$},label style={font=\large},title style={font=\large},legend cell align={left},ymin=1e-13,ymax=1e-3]
		\addplot table[skip first n=1,x expr={0.5/\thisrowno{0}}, y expr={\thisrowno{72}}] {../Code/results/results_imexmd_parallel_rkrk6_kmax72_vdP_IMEX_eps0.1.csv};
		\addplot table[skip first n=1,x expr={0.5/\thisrowno{0}}, y expr={\thisrowno{72}}] {../Code/results/results_imexmd_parallel_rkrk6_kmax72_vdP_IMEX_eps0.01.csv};
		\addplot table[skip first n=1,x expr={0.5/\thisrowno{0}}, y expr={\thisrowno{72}}] {../Code/results/results_imexmd_parallel_rkrk6_kmax72_vdP_IMEX_eps0.001.csv};
		\addplot table[skip first n=1,x expr={0.5/\thisrowno{0}}, y expr={\thisrowno{72}}] {../Code/results/results_imexmd_parallel_rkrk6_kmax72_vdP_IMEX_eps0.0001.csv};
		\addplot table[skip first n=1,x expr={0.5/\thisrowno{0}}, y expr={\thisrowno{72}}] {../Code/results/results_imexmd_parallel_rkrk6_kmax72_vdP_IMEX_eps1e-05.csv};
		\addplot[color=black,thick,fill = white,forget plot] coordinates {(6e-2,1.0e-9) (6e-2,0.5^6*1.0e-9) (3e-2,0.5^6*1.0e-9) (6e-2,1.0e-9)}; \node at (axis cs: 0.075, 1.5e-10){$6$};
		\legend{$\eps=10^{-1}$,$\eps=10^{-2}$,$\eps=10^{-3}$,$\eps=10^{-4}$,$\eps=10^{-5}$}
		\end{loglogaxis}
		\end{tikzpicture}\\
		\tikzsetnextfilename{limit4}
		\begin{tikzpicture}[scale=0.5]
		\begin{loglogaxis}[cycle list name=epslist,xlabel={$\Delta t$},ylabel={$\|w-w_h\|_2$} ,grid=major,legend style={at={(0.02,0.98)},anchor= north west,font=\footnotesize},title={ $\method{6}{\infty}$},label style={font=\large},title style={font=\large},legend cell align={left},ymin=1e-13,ymax=1e-3]
		\addplot table[skip first n=1,x expr={0.5/\thisrowno{0}}, y expr={\thisrowno{1}}] {../Code/results/results_imexmd_serialrk6_limit_vdP_IMEX_eps0.1.csv};
		\addplot table[skip first n=1,x expr={0.5/\thisrowno{0}}, y expr={\thisrowno{1}}] {../Code/results/results_imexmd_serialrk6_limit_vdP_IMEX_eps0.01.csv};
		\addplot table[skip first n=1,x expr={0.5/\thisrowno{0}}, y expr={\thisrowno{1}}] {../Code/results/results_imexmd_serialrk6_limit_vdP_IMEX_eps0.001.csv};
		\addplot table[skip first n=1,x expr={0.5/\thisrowno{0}}, y expr={\thisrowno{1}}] {../Code/results/results_imexmd_serialrk6_limit_vdP_IMEX_eps0.0001.csv};
		\addplot table[skip first n=1,x expr={0.5/\thisrowno{0}}, y expr={\thisrowno{1}}] {../Code/results/results_imexmd_serialrk6_limit_vdP_IMEX_eps1e-05.csv};
		\addplot[color=black,thick,fill = white,forget plot] coordinates {(6e-2,1.0e-9) (6e-2,0.5^6*1.0e-9) (3e-2,0.5^6*1.0e-9) (6e-2,1.0e-9)}; \node at (axis cs: 0.075, 1.5e-10){$6$};
		\legend{$\eps=10^{-1}$,$\eps=10^{-2}$,$\eps=10^{-3}$,$\eps=10^{-4}$,$\eps=10^{-5}$}
		\end{loglogaxis}
		\end{tikzpicture}&
		\tikzsetnextfilename{limit5}
		\begin{tikzpicture}[scale=0.5]
		\begin{loglogaxis}[cycle list name=epslist,xlabel={$\Delta t$},grid=major,legend style={at={(0.02,0.98)},anchor= north west,font=\footnotesize},title={ $6^{th}$ order, limiting RK-method},label style={font=\large},title style={font=\large},legend cell align={left},ymin=1e-13,ymax=1e-3]
		\addplot table[skip first n=1,x expr={0.5/\thisrowno{0}}, y expr={\thisrowno{1}}] {../Code/results/RungeKutta/results_imexmd_rkrk6_kmax1_vdP_eps0.1.csv};
		\addplot table[skip first n=1,x expr={0.5/\thisrowno{0}}, y expr={\thisrowno{1}}] {../Code/results/RungeKutta/results_imexmd_rkrk6_kmax1_vdP_eps0.01.csv};
		\addplot table[skip first n=1,x expr={0.5/\thisrowno{0}}, y expr={\thisrowno{1}}] {../Code/results/RungeKutta/results_imexmd_rkrk6_kmax1_vdP_eps0.001.csv};
		\addplot table[skip first n=1,x expr={0.5/\thisrowno{0}}, y expr={\thisrowno{1}}] {../Code/results/RungeKutta/results_imexmd_rkrk6_kmax1_vdP_eps0.0001.csv};
		\addplot table[skip first n=1,x expr={0.5/\thisrowno{0}}, y expr={\thisrowno{1}}] {../Code/results/RungeKutta/results_imexmd_rkrk6_kmax1_vdP_eps1e-05.csv};
		\addplot[color=black,thick,fill = white,forget plot] coordinates {(6e-2,1.0e-9) (6e-2,0.5^6*1.0e-9) (3e-2,0.5^6*1.0e-9) (6e-2,1.0e-9)}; \node at (axis cs: 0.075, 1.5e-10){$6$};
		\legend{$\eps=10^{-1}$,$\eps=10^{-2}$,$\eps=10^{-3}$,$\eps=10^{-4}$,$\eps=10^{-5}$}
		\end{loglogaxis}
		\end{tikzpicture}
		\end{tabular}
        }
        {
		\begin{tabular}{ccc}
            \includegraphics{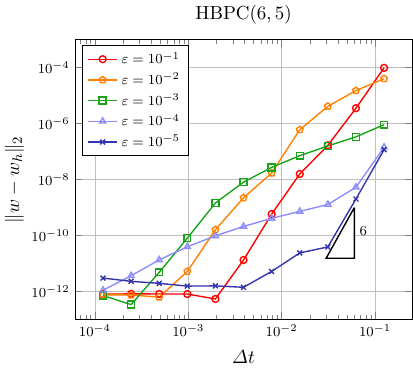}&
			\includegraphics{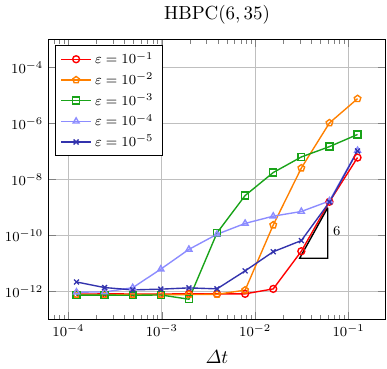}&
			\includegraphics{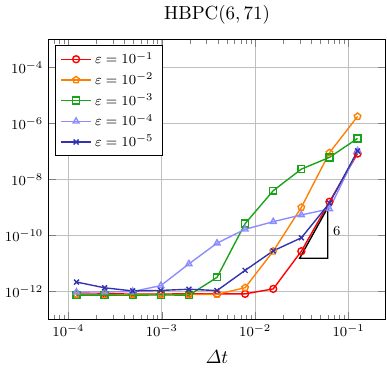}
		\end{tabular}
		\begin{tabular}{cc}
			\includegraphics{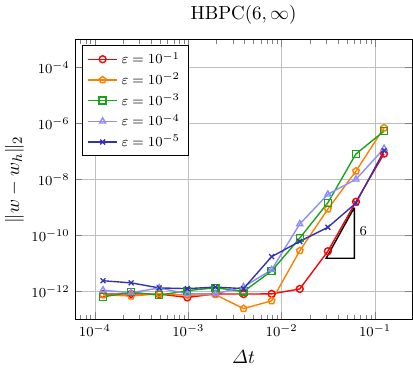}&
			\includegraphics{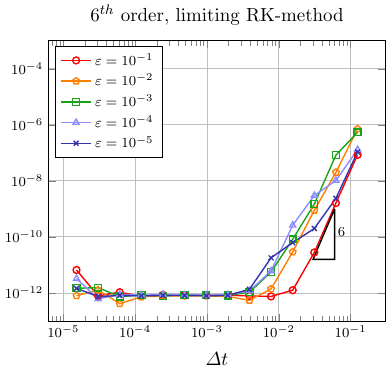}
		\end{tabular}
        }
	\caption{Error for van-der-Pol equation at $\Tend=0.5$ of simulations with $\method{6}{\km}$. Orders 4 and 6 are omitted for brevity, but show similar behavior.  In each image, we increase the level of stiffness of the problem by shrinking $\eps$, and in each frame we consider the impact of increasing the total number of correction steps $\km$.
	The last picture (second on the second line) shows the results of solving the fully coupled limiting Runge-Kutta method, (cf. Def.~\ref{def:RKlimit}). The first image on the second line shows the numerically determined limit if $\km \rightarrow \infty$. This is in excellent agreement with the underlying Runge-Kutta method in the right frame, as we expect to be the case when the iterates converge. }\label{fig:Vdp_eps}
\end{figure}
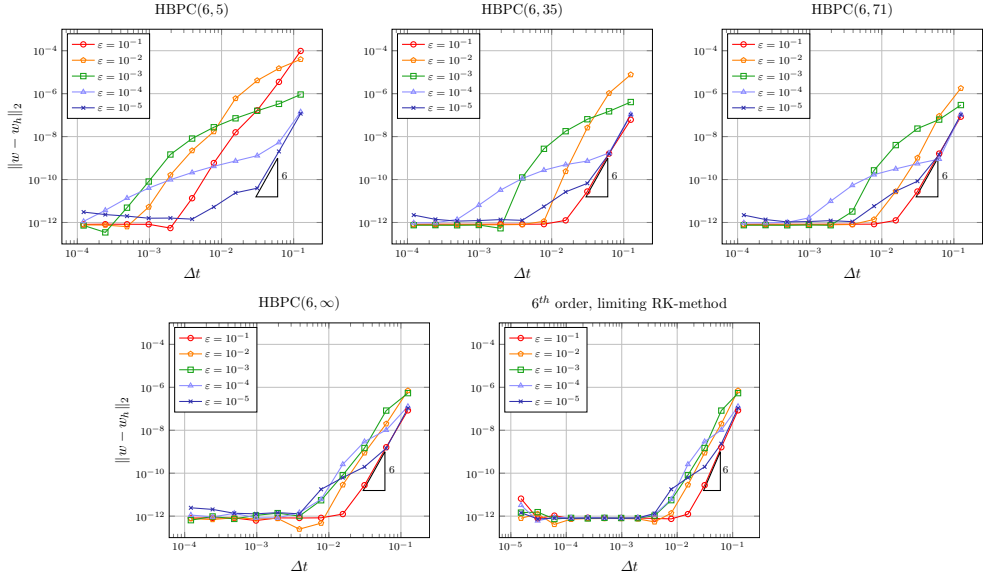



In Fig.~\ref{fig:Vdp_eps} the results of the simulation are summarized for the $6^{th}$-order method. For an increasing stiffness, order reduction can be observed. This can, at least partially, be cured by increasing the number of correction steps; perfect convergence cannot be observed. This is consistent with the limiting method and one can see that with $\km\rightarrow\infty$ the limiting method is reached. The results also extend to the $4^{th}$- and $8^{th}$-order scheme (not shown here for the sake of brevity): small values of $\km$ produce order reduction for some regimes. The effect can be mitigated by increasing the number of iterates with $\km \to \infty$. 
For the $4^{th}$-order limiting-scheme, order reduction is completely ruled out, which is not the case for $6^{th}$ and $8^{th}$-order schemes.

%% file: sec_results_modifications.tex
Given the flexibility of the solver, there are several opportunities to enhance the proposed Alg.~\ref{alg:mdode}. In this subsection, we present one such variation that improves the accuracy and scalability of the solver by making a total of two modifications:

\begin{itemize}
	\item \textbf{Improved predictor.}  The predictor in Alg.~\ref{alg:mdode} is completely independent of the correction steps, as the used initial condition is given by the previously computed predictor results.  A minor modification would be to use the more accurate result from the first (or higher) correction from the previous time step. This is advantageous whenever the main computational load lies on the predictor.  This could happen, e.g., if the number of Newton iterations for the predictor is higher than for the corrector, which is often observed in practice.  Moreover, the predictor could, in principle, be leveraged as the cornerstone for adaptive timestepping, which is crucial for many large-scale applications. 
	%
	Based on our observations, we find we can accomplish this and have a predictor that is overall  third-order accurete in time, which reduces the total number of corrections required.
	\item \textbf{Improved quadrature.}  The $l-$th stage of the $(k)$-th corrector with $l > 1$ uses a
	 quadrature rule $\II_l$ in Alg.~\ref{alg:mdode} that is still fully computed with values from correction step $k$, although for $\iota < l$, the values $\algs w n {k+1} \iota$ are readily available. If we consider a Gauss-Seidel type approach, similar to the one found in \cite{Crockatt2018}, we change this so that in $\II_l$, we replace $\algs w n {k} \iota$ by $\algs w n {k+1} \iota$ for $\iota < l$. Here, we show that that this not only leads to a more accurate scheme, but it is also lower storage.
\end{itemize}
We formalize the proposed changes to Alg.~\ref{alg:mdode} as its own algorithm.

\begin{algorithm}[$\methodt{q,\km}$]\label{alg:mdodejonas}
	To advance the solution to Eq.~\eqref{eq:ode} in time, we compute values $\algs w n {k } l$. The meaning of the parameters $n$, $k$ and $l$ is explained in Eq.~\eqref{eq:wnkl} and thereafter. 
	To account for the initial conditions, define 
	\begin{align*} 
	\algs w {-1} k s := w_0. 
	\end{align*}
	First, the values $\algs w n 0 l$ are filled using a straightforward second-order IMEX-Taylor method with an improved initial condition. 
	\begin{enumerate}
		\item \textbf{Predict.} Solve the following expression for $\algs w n 0 l$ and $1 \leq l \leq s$: 
		\begin{align} 
		\begin{split}
		\label{eq:predictmod}
		\algs w {n} 0 l &:= {\color{blue}\algs w {n-1} 1 s} \\
		&+ c_l\dt \left(\algs {\Phii} n 0 l + {\color{blue}\algs{\Phie} {n-1} 1 s}\right) 
		+ \frac{(c_l\dt)^2}{2} \left({\color{blue}\algs {\dPhie} {n-1} 1 s} - \algs {\dPhii} n 0 l\right).
		\end{split}
		\end{align}  
		The modified terms (${\algs w {n-1} 1 s}, {\algs{\Phie} {n-1} 1 s}$, and ${\algs{\dPhie} {n-1} 1 s}$, marked in blue) simply use one higher iterate than the one found in Alg.~\ref{alg:mdode}.  This makes the predictor third-order accurate due to its improved local truncation error.
		\item \textbf{Correct.} Next, the correction steps take place to fill the values of $\algs w n k l$ for $1 \leq k \leq \km$.  Solve the following for $\algs w {n} {k+1} l$, for each $2 \leq l \leq s$ and each $0 \leq k < k_{\max}$:
		\begin{align}
		\begin{split}
		\label{eq:correctmod}
		\algs w {n} {k+1} 1 &:= {\color{VioletRed}\algs w {n-1} {k+2} s}, \\
		\algs w {n} {k+1} l &:= {\color{VioletRed}\algs w {n-1} {k+2} s} \\ 
		&
		+ {\dt} \left( \algs {\Phii} {n} {k+1} l - \algs {\Phii} n {k} l\right) 
		- \frac{\dt^2}{2} \left(\algs {\dPhii} {n} {k+1} l- \algs {\dPhii} {n} {k} l\right)
		\\ 
		&+\; \II_l(\color{blue}\algs \Phi{n}{k+1} 0, \ldots,\algs \Phi n {k+1} {l-1},\color{black} \algs \Phi n {k} l, \ldots, \algs \Phi n {k} s)  
		\end{split}
		\end{align}
		The red term(s) ($\algs w {n-1} {k+2} s$) are the same modified values we implemented in Alg.~\ref{alg:mdode}.  
		The blue terms $\algs \Phi{n}{k+1} 0, \ldots,\algs \Phi n {k+1} {l-1}$ are evaluated at iterate $k+1$ instead of $k$.  This produces smaller errors and requires less storage.
		Again, if $k = k_{\max} -1$, then the $k+2$ in the red terms are replaced by $\km$ in order to close the recursion.
		
		\item \textbf{Update.} In order to retain a first-same-as-last property, we update the solution with
		\begin{align*} 
		w^{n+1} := \algs w n {k_{\max}} s.
		\end{align*}
	\end{enumerate}	
\end{algorithm}
The parallelization strategy suggested in Fig.~\ref{fig:timeparallel} remains largely unaltered due to the suggested grouping.  Given that the first thread operates on both the $k=0$ and the $k=1$ iterates, the predictor simply draws information from the $k=1$ iterate instead of the $k=0$ iterate.  The corrector also remains unaltered - in fact there is slightly looser coupling given that the quadrature rule uses information that is being updated.

In the following, we explore the impact of the further modifications to the algorithm numerically. We find that in most cases the modified algorithm behaves better than the standard one. Scaling results are therefore given on the basis of Alg.~\ref{alg:mdodejonas} only.

\subsection{Influence of the modifications}
We study the influence of the two modifications with previously considered Pareschi-Russo problem (Eq.~\eqref{eq:PR_problem}) with $\eps=1$. In Fig.~\ref{fig:PR_predictor} the results with Alg.~\ref{alg:mdode} and the modified Alg.~\ref{alg:mdodejonas} are compared side-by-side. Although the modification in the predictor and the corrector are presented simultaneously, we carry out a separated analysis of both changes in order to highlight the individual influence.
\pgfplotstableset{col sep=comma}
\begin{figure}
	\centering
	\begin{tabular}{cc}
        \ifthenelse{\boolean{compilefromscratch}}{
        
        \tikzsetnextfilename{results_imexmd_parallel_rkrk4_kmax10_PR_IMEX_eps1}
		\begin{tikzpicture}[scale=0.65]
			\begin{loglogaxis}[cycle list name=green,xlabel={$\Delta t$},ylabel={$\|w-w_h\|_2$} ,grid=major,legend style={at={(1.02,0.5)},anchor= west,font=\footnotesize},
			title={  $\method{4}{9}$, Alg.~1},
			label style={font=\large},title style={font=\large},legend cell align={left},ymin=1e-11,ymax=10]
			\addplot table[skip first n=1,x expr={2/\thisrowno{0}}, y expr={\thisrowno{1}}] {../Code/results/results_imexmd_parallel_rkrk4_kmax10_PR_IMEX_eps1.csv};
			\addplot table[skip first n=1,x expr={2/\thisrowno{0}}, y expr={\thisrowno{2}}] {../Code/results/results_imexmd_parallel_rkrk4_kmax10_PR_IMEX_eps1.csv};
			\addplot table[skip first n=1,x expr={2/\thisrowno{0}}, y expr={\thisrowno{3}}] {../Code/results/results_imexmd_parallel_rkrk4_kmax10_PR_IMEX_eps1.csv};
			\addplot table[skip first n=1,x expr={2/\thisrowno{0}}, y expr={\thisrowno{4}}] {../Code/results/results_imexmd_parallel_rkrk4_kmax10_PR_IMEX_eps1.csv};
			\addplot table[skip first n=1,x expr={2/\thisrowno{0}}, y expr={\thisrowno{5}}] {../Code/results/results_imexmd_parallel_rkrk4_kmax10_PR_IMEX_eps1.csv};
			\addplot table[skip first n=1,x expr={2/\thisrowno{0}}, y expr={\thisrowno{6}}] {../Code/results/results_imexmd_parallel_rkrk4_kmax10_PR_IMEX_eps1.csv};
			\addplot table[skip first n=1,x expr={2/\thisrowno{0}}, y expr={\thisrowno{7}}] {../Code/results/results_imexmd_parallel_rkrk4_kmax10_PR_IMEX_eps1.csv};
			\addplot table[skip first n=1,x expr={2/\thisrowno{0}}, y expr={\thisrowno{8}}] {../Code/results/results_imexmd_parallel_rkrk4_kmax10_PR_IMEX_eps1.csv};
			\addplot table[skip first n=1,x expr={2/\thisrowno{0}}, y expr={\thisrowno{9}}] {../Code/results/results_imexmd_parallel_rkrk4_kmax10_PR_IMEX_eps1.csv};
			\addplot table[skip first n=1,x expr={2/\thisrowno{0}}, y expr={\thisrowno{10}}] {../Code/results/results_imexmd_parallel_rkrk4_kmax10_PR_IMEX_eps1.csv};
			\addplot[color=black,thick,fill = white,forget plot] coordinates {(2e-2,4.5e-4) (2e-2,0.5^2*4.5e-4) (1e-2,0.5^2*4.5e-4) (2e-2,4.5e-4)}; \node at (axis cs: 2.3e-2, 2.0e-4){$2$};
			\addplot[color=black,thick,fill = white,forget plot] coordinates {(2e-2,1.5e-5) (2e-2,0.5^3*1.5e-5) (1e-2,0.5^3*1.5e-5) (2e-2,1.5e-5)}; \node at (axis cs: 1.8e-2, 4.0e-6){$3$};
			\addplot[color=black,thick,fill = white,forget plot] coordinates {(2e-2,6.5e-7) (2e-2,0.5^4*6.5e-7) (1e-2,0.5^4*6.5e-7) (2e-2,6.5e-7)}; \node at (axis cs: 1.8e-2, 1.2e-7){$4$};
			\addplot[color=black,thick,fill = white,forget plot] coordinates {(2e-2,1.5e-8) (2e-2,0.5^4*1.5e-8) (1e-2,0.5^4*1.5e-8) (2e-2,1.5e-8)}; \node at (axis cs: 1.8e-2, 3.0e-9){$4$};
			\end{loglogaxis}
		\end{tikzpicture}&
		\tikzsetnextfilename{results_imexmd_parallel_rkrk4_kmax10_PR_IMEX_eps1}
		\tikzsetnextfilename{results_alg2_imexmd_parallel_rkrk4_kmax10_PR_IMEX_eps1}
		\begin{tikzpicture}[scale=0.65]
			\begin{loglogaxis}[cycle list name=green,xlabel={$\Delta t$},grid=major,legend style={at={(1.02,0.5)},anchor= west,font=\footnotesize},
			title={  $\methodt{4}{9}$, Alg.~2},
			label style={font=\large},title style={font=\large},legend cell align={left},ymin=1e-11,ymax=10]
			\addplot table[skip first n=1,x expr={2/\thisrowno{0}}, y expr={\thisrowno{1}}] {../Code/results_alg2/results_imexmd_parallel_rkrk4_kmax10_PR_IMEX_eps1.csv};
			\addplot table[skip first n=1,x expr={2/\thisrowno{0}}, y expr={\thisrowno{2}}] {../Code/results_alg2/results_imexmd_parallel_rkrk4_kmax10_PR_IMEX_eps1.csv};
			\addplot table[skip first n=1,x expr={2/\thisrowno{0}}, y expr={\thisrowno{3}}] {../Code/results_alg2/results_imexmd_parallel_rkrk4_kmax10_PR_IMEX_eps1.csv};
			\addplot table[skip first n=1,x expr={2/\thisrowno{0}}, y expr={\thisrowno{4}}] {../Code/results_alg2/results_imexmd_parallel_rkrk4_kmax10_PR_IMEX_eps1.csv};
			\addplot table[skip first n=1,x expr={2/\thisrowno{0}}, y expr={\thisrowno{5}}] {../Code/results_alg2/results_imexmd_parallel_rkrk4_kmax10_PR_IMEX_eps1.csv};
			\addplot table[skip first n=1,x expr={2/\thisrowno{0}}, y expr={\thisrowno{6}}] {../Code/results_alg2/results_imexmd_parallel_rkrk4_kmax10_PR_IMEX_eps1.csv};
			\addplot table[skip first n=1,x expr={2/\thisrowno{0}}, y expr={\thisrowno{7}}] {../Code/results_alg2/results_imexmd_parallel_rkrk4_kmax10_PR_IMEX_eps1.csv};
			\addplot table[skip first n=1,x expr={2/\thisrowno{0}}, y expr={\thisrowno{8}}] {../Code/results_alg2/results_imexmd_parallel_rkrk4_kmax10_PR_IMEX_eps1.csv};
			\addplot table[skip first n=1,x expr={2/\thisrowno{0}}, y expr={\thisrowno{9}}] {../Code/results_alg2/results_imexmd_parallel_rkrk4_kmax10_PR_IMEX_eps1.csv};
			\addplot table[skip first n=1,x expr={2/\thisrowno{0}}, y expr={\thisrowno{10}}] {../Code/results_alg2/results_imexmd_parallel_rkrk4_kmax10_PR_IMEX_eps1.csv};
			\addplot[color=black,thick,fill = white,forget plot] coordinates {(2e-2,4.5e-6) (2e-2,0.5^3*4.5e-6) (1e-2,0.5^3*4.5e-6) (2e-2,4.5e-6)}; \node at (axis cs: 1.8e-2, 1.3e-6){$3$};
			\addplot[color=black,thick,fill = white,forget plot] coordinates {(2e-2,1.5e-8) (2e-2,0.5^4*1.5e-8) (1e-2,0.5^4*1.5e-8) (2e-2,1.5e-8)}; \node at (axis cs: 1.8e-2, 3.0e-9){$4$};
			\legend{$k=0$,$k=1$,$k=2$,$k=3$,$k=4$,$k=5$,$k=6$,$k=7$,$k=8$,$k=9$}
			\end{loglogaxis}
		\end{tikzpicture}\\
		\tikzsetnextfilename{results_imexmd_parallel_rkrk6_kmax10_PR_IMEX_eps1}
		\begin{tikzpicture}[scale=0.65]
			\begin{loglogaxis}[cycle list name=green,xlabel={$\Delta t$},ylabel={$\|w-w_h\|_2$} ,grid=major,legend style={at={(1.02,0.5)},anchor= west,font=\footnotesize},
			title={  $\method{6}{9}$, Alg.~1},
			label style={font=\large},title style={font=\large},legend cell align={left},ymin=1e-14,ymax=10]
			\addplot table[skip first n=1,x expr={2/\thisrowno{0}}, y expr={\thisrowno{1}}] {../Code/results/results_imexmd_parallel_rkrk6_kmax10_PR_IMEX_eps1.csv};
			\addplot table[skip first n=1,x expr={2/\thisrowno{0}}, y expr={\thisrowno{2}}] {../Code/results/results_imexmd_parallel_rkrk6_kmax10_PR_IMEX_eps1.csv};
			\addplot table[skip first n=1,x expr={2/\thisrowno{0}}, y expr={\thisrowno{3}}] {../Code/results/results_imexmd_parallel_rkrk6_kmax10_PR_IMEX_eps1.csv};
			\addplot table[skip first n=1,x expr={2/\thisrowno{0}}, y expr={\thisrowno{4}}] {../Code/results/results_imexmd_parallel_rkrk6_kmax10_PR_IMEX_eps1.csv};
			\addplot table[skip first n=1,x expr={2/\thisrowno{0}}, y expr={\thisrowno{5}}] {../Code/results/results_imexmd_parallel_rkrk6_kmax10_PR_IMEX_eps1.csv};
			\addplot table[skip first n=1,x expr={2/\thisrowno{0}}, y expr={\thisrowno{6}}] {../Code/results/results_imexmd_parallel_rkrk6_kmax10_PR_IMEX_eps1.csv};
			\addplot table[skip first n=1,x expr={2/\thisrowno{0}}, y expr={\thisrowno{7}}] {../Code/results/results_imexmd_parallel_rkrk6_kmax10_PR_IMEX_eps1.csv};
			\addplot table[skip first n=1,x expr={2/\thisrowno{0}}, y expr={\thisrowno{8}}] {../Code/results/results_imexmd_parallel_rkrk6_kmax10_PR_IMEX_eps1.csv};
			\addplot table[skip first n=1,x expr={2/\thisrowno{0}}, y expr={\thisrowno{9}}] {../Code/results/results_imexmd_parallel_rkrk6_kmax10_PR_IMEX_eps1.csv};
			\addplot table[skip first n=1,x expr={2/\thisrowno{0}}, y expr={\thisrowno{10}}] {../Code/results/results_imexmd_parallel_rkrk6_kmax10_PR_IMEX_eps1.csv};
			\addplot[color=black,thick,fill = white,forget plot] coordinates {(2e-2,4.5e-4) (2e-2,0.5^2*4.5e-4) (1e-2,0.5^2*4.5e-4) (2e-2,4.5e-4)}; \node at (axis cs: 2.3e-2, 2.0e-4){$2$};
			\addplot[color=black,thick,fill = white,forget plot] coordinates {(9e-2,1.1e-9) (9e-2,0.5^6*1.1e-9) (4.5e-2,0.5^6*1.1e-9) (9e-2,1.1e-9)}; \node at (axis cs: 0.08, 1.1e-10){$6$};
			\end{loglogaxis}
		\end{tikzpicture}&
		\tikzsetnextfilename{results_alg2_imexmd_parallel_rkrk6_kmax10_PR_IMEX_eps1}
		\begin{tikzpicture}[scale=0.65]
			\begin{loglogaxis}[cycle list name=green,xlabel={$\Delta t$},grid=major,legend style={at={(1.02,0.5)},anchor= west,font=\footnotesize},
			title={  $\methodt{6}{9}$, Alg.~2},
			label style={font=\large},title style={font=\large},legend cell align={left},ymin=1e-14,ymax=10]
			\addplot table[skip first n=1,x expr={2/\thisrowno{0}}, y expr={\thisrowno{1}}] {../Code/results_alg2/results_imexmd_parallel_rkrk6_kmax10_PR_IMEX_eps1.csv};
			\addplot table[skip first n=1,x expr={2/\thisrowno{0}}, y expr={\thisrowno{2}}] {../Code/results_alg2/results_imexmd_parallel_rkrk6_kmax10_PR_IMEX_eps1.csv};
			\addplot table[skip first n=1,x expr={2/\thisrowno{0}}, y expr={\thisrowno{3}}] {../Code/results_alg2/results_imexmd_parallel_rkrk6_kmax10_PR_IMEX_eps1.csv};
			\addplot table[skip first n=1,x expr={2/\thisrowno{0}}, y expr={\thisrowno{4}}] {../Code/results_alg2/results_imexmd_parallel_rkrk6_kmax10_PR_IMEX_eps1.csv};
			\addplot table[skip first n=1,x expr={2/\thisrowno{0}}, y expr={\thisrowno{5}}] {../Code/results_alg2/results_imexmd_parallel_rkrk6_kmax10_PR_IMEX_eps1.csv};
			\addplot table[skip first n=1,x expr={2/\thisrowno{0}}, y expr={\thisrowno{6}}] {../Code/results_alg2/results_imexmd_parallel_rkrk6_kmax10_PR_IMEX_eps1.csv};
			\addplot table[skip first n=1,x expr={2/\thisrowno{0}}, y expr={\thisrowno{7}}] {../Code/results_alg2/results_imexmd_parallel_rkrk6_kmax10_PR_IMEX_eps1.csv};
			\addplot table[skip first n=1,x expr={2/\thisrowno{0}}, y expr={\thisrowno{8}}] {../Code/results_alg2/results_imexmd_parallel_rkrk6_kmax10_PR_IMEX_eps1.csv};
			\addplot table[skip first n=1,x expr={2/\thisrowno{0}}, y expr={\thisrowno{9}}] {../Code/results_alg2/results_imexmd_parallel_rkrk6_kmax10_PR_IMEX_eps1.csv};
			\addplot table[skip first n=1,x expr={2/\thisrowno{0}}, y expr={\thisrowno{10}}] {../Code/results_alg2/results_imexmd_parallel_rkrk6_kmax10_PR_IMEX_eps1.csv};
			\addplot[color=black,thick,fill = white,forget plot] coordinates {(2e-2,4.5e-6) (2e-2,0.5^3*4.5e-6) (1e-2,0.5^3*4.5e-6) (2e-2,4.5e-6)}; \node at (axis cs: 1.8e-2, 1.3e-6){$3$};
			\addplot[color=black,thick,fill = white,forget plot] coordinates {(9e-2,1.1e-9) (9e-2,0.5^6*1.1e-9) (4.5e-2,0.5^6*1.1e-9) (9e-2,1.1e-9)}; \node at (axis cs: 0.08, 1.1e-10){$6$};
			\legend{$k=0$,$k=1$,$k=2$,$k=3$,$k=4$,$k=5$,$k=6$,$k=7$,$k=8$,$k=9$}
			\end{loglogaxis}
		\end{tikzpicture}\\
		\tikzsetnextfilename{results_imexmd_parallel_rkrk8_kmax10_PR_IMEX_eps1} 
		\begin{tikzpicture}[scale=0.65]
			\begin{loglogaxis}[cycle list name=green,xlabel={$\Delta t$},ylabel={$\|w-w_h\|_2$} ,grid=major,legend style={at={(1.02,0.5)},anchor= west,font=\footnotesize},
			title={  $\method{8}{9}$, Alg.~1},
			label style={font=\large},title style={font=\large},legend cell align={left},ymin=1e-14,ymax=10]
			\addplot table[skip first n=1,x expr={2/\thisrowno{0}}, y expr={\thisrowno{1}}] {../Code/results/results_imexmd_parallel_rkrk8_kmax10_PR_IMEX_eps1.csv};
			\addplot table[skip first n=1,x expr={2/\thisrowno{0}}, y expr={\thisrowno{2}}] {../Code/results/results_imexmd_parallel_rkrk8_kmax10_PR_IMEX_eps1.csv};
			\addplot table[skip first n=1,x expr={2/\thisrowno{0}}, y expr={\thisrowno{3}}] {../Code/results/results_imexmd_parallel_rkrk8_kmax10_PR_IMEX_eps1.csv};
			\addplot table[skip first n=1,x expr={2/\thisrowno{0}}, y expr={\thisrowno{4}}] {../Code/results/results_imexmd_parallel_rkrk8_kmax10_PR_IMEX_eps1.csv};
			\addplot table[skip first n=1,x expr={2/\thisrowno{0}}, y expr={\thisrowno{5}}] {../Code/results/results_imexmd_parallel_rkrk8_kmax10_PR_IMEX_eps1.csv};
			\addplot table[skip first n=1,x expr={2/\thisrowno{0}}, y expr={\thisrowno{6}}] {../Code/results/results_imexmd_parallel_rkrk8_kmax10_PR_IMEX_eps1.csv};
			\addplot table[skip first n=1,x expr={2/\thisrowno{0}}, y expr={\thisrowno{7}}] {../Code/results/results_imexmd_parallel_rkrk8_kmax10_PR_IMEX_eps1.csv};
			\addplot table[skip first n=1,x expr={2/\thisrowno{0}}, y expr={\thisrowno{8}}] {../Code/results/results_imexmd_parallel_rkrk8_kmax10_PR_IMEX_eps1.csv};
			\addplot table[skip first n=1,x expr={2/\thisrowno{0}}, y expr={\thisrowno{9}}] {../Code/results/results_imexmd_parallel_rkrk8_kmax10_PR_IMEX_eps1.csv};
			\addplot table[skip first n=1,x expr={2/\thisrowno{0}}, y expr={\thisrowno{10}}] {../Code/results/results_imexmd_parallel_rkrk8_kmax10_PR_IMEX_eps1.csv};
			\addplot[color=black,thick,fill = white,forget plot] coordinates {(2e-2,4.5e-4) (2e-2,0.5^2*4.5e-4) (1e-2,0.5^2*4.5e-4) (2e-2,4.5e-4)}; \node at (axis cs: 2.3e-2, 2.0e-4){$2$};
			\addplot[color=black,thick,fill = white,forget plot] coordinates {(1.2e-1,2.5e-9) (1.2e-1,0.5^8*2.5e-9) (6.0e-2,0.5^8*2.5e-9) (1.2e-1,2.5e-9)}; \node at (axis cs: 0.10, 1.0e-10){$8$};
			\end{loglogaxis}
		\end{tikzpicture}&
		\tikzsetnextfilename{results_alg2_imexmd_parallel_rkrk8_kmax10_PR_IMEX_eps1} 
		\begin{tikzpicture}[scale=0.65]
			\begin{loglogaxis}[cycle list name=green,xlabel={$\Delta t$},grid=major,legend style={at={(1.02,0.5)},anchor= west,font=\footnotesize},
			title={  $\methodt{8}{9}$, Alg.~2},
			label style={font=\large},title style={font=\large},legend cell align={left},ymin=1e-14,ymax=10]
			\addplot table[skip first n=1,x expr={2/\thisrowno{0}}, y expr={\thisrowno{1}}]
			{../Code/results_alg2/results_imexmd_parallel_rkrk8_kmax10_PR_IMEX_eps1.csv};
			\addplot table[skip first n=1,x expr={2/\thisrowno{0}}, y expr={\thisrowno{2}}] {../Code/results_alg2/results_imexmd_parallel_rkrk8_kmax10_PR_IMEX_eps1.csv};
			\addplot table[skip first n=1,x expr={2/\thisrowno{0}}, y expr={\thisrowno{3}}] {../Code/results_alg2/results_imexmd_parallel_rkrk8_kmax10_PR_IMEX_eps1.csv};
			\addplot table[skip first n=1,x expr={2/\thisrowno{0}}, y expr={\thisrowno{4}}] {../Code/results_alg2/results_imexmd_parallel_rkrk8_kmax10_PR_IMEX_eps1.csv};
			\addplot table[skip first n=1,x expr={2/\thisrowno{0}}, y expr={\thisrowno{5}}] {../Code/results_alg2/results_imexmd_parallel_rkrk8_kmax10_PR_IMEX_eps1.csv};
			\addplot table[skip first n=1,x expr={2/\thisrowno{0}}, y expr={\thisrowno{6}}] {../Code/results_alg2/results_imexmd_parallel_rkrk8_kmax10_PR_IMEX_eps1.csv};
			\addplot table[skip first n=1,x expr={2/\thisrowno{0}}, y expr={\thisrowno{7}}] {../Code/results_alg2/results_imexmd_parallel_rkrk8_kmax10_PR_IMEX_eps1.csv};
			\addplot table[skip first n=1,x expr={2/\thisrowno{0}}, y expr={\thisrowno{8}}] {../Code/results_alg2/results_imexmd_parallel_rkrk8_kmax10_PR_IMEX_eps1.csv};
			\addplot table[skip first n=1,x expr={2/\thisrowno{0}}, y expr={\thisrowno{9}}] {../Code/results_alg2/results_imexmd_parallel_rkrk8_kmax10_PR_IMEX_eps1.csv};
			\addplot table[skip first n=1,x expr={2/\thisrowno{0}}, y expr={\thisrowno{10}}] {../Code/results_alg2/results_imexmd_parallel_rkrk8_kmax10_PR_IMEX_eps1.csv};
			\addplot[color=black,thick,fill = white,forget plot] coordinates {(2e-2,4.5e-6) (2e-2,0.5^3*4.5e-6) (1e-2,0.5^3*4.5e-6) (2e-2,4.5e-6)}; \node at (axis cs: 1.8e-2, 1.3e-6){$3$};
			\addplot[color=black,thick,fill = white,forget plot] coordinates {(1.2e-1,2.5e-9) (1.2e-1,0.5^8*2.5e-9) (6.0e-2,0.5^8*2.5e-9) (1.2e-1,2.5e-9)}; \node at (axis cs: 0.10, 1.0e-10){$8$};
			\legend{$k=0$,$k=1$,$k=2$,$k=3$,$k=4$,$k=5$,$k=6$,$k=7$,$k=8$,$k=9$}
			\end{loglogaxis}
		\end{tikzpicture}
		}
		{
            \includegraphics{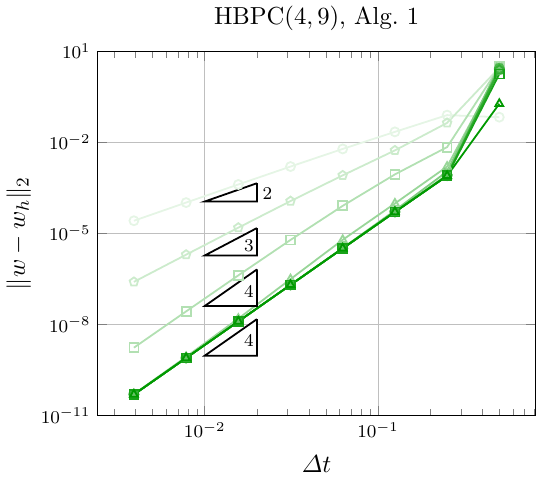}&
            \includegraphics{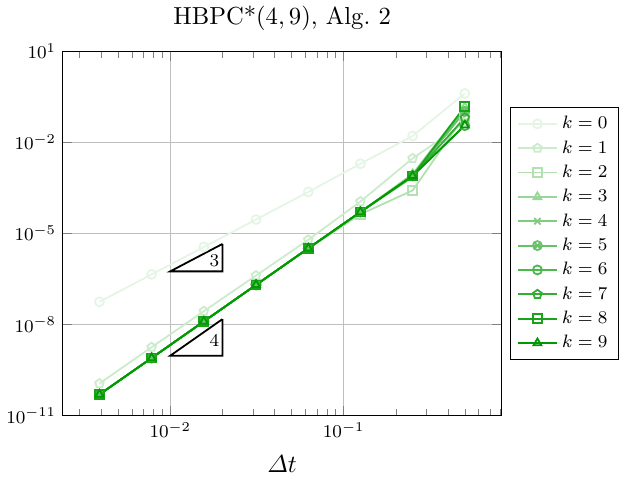}\\
            \includegraphics{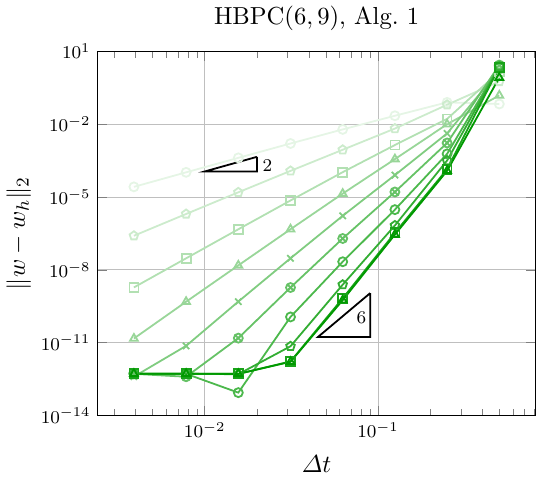}&
            \includegraphics{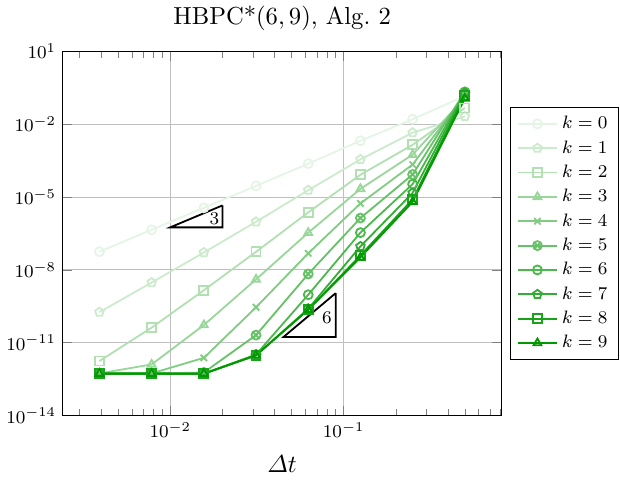} \\
            \includegraphics{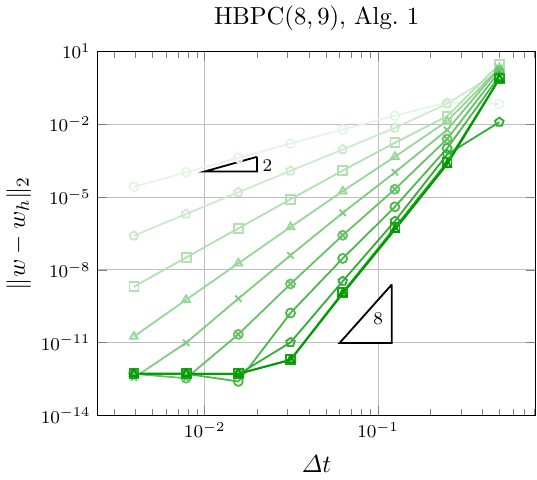}&
            \includegraphics{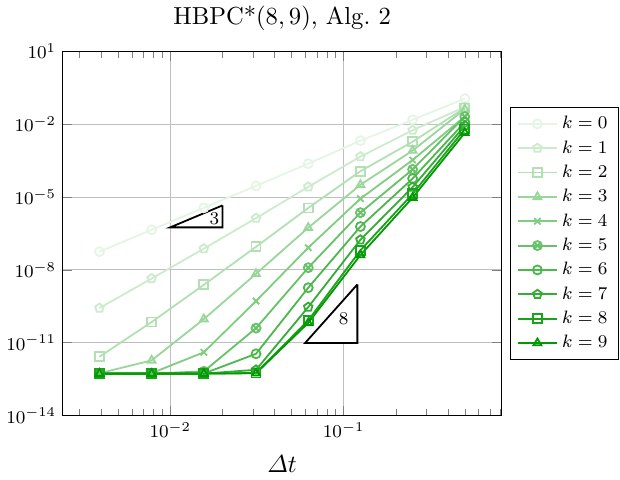} 
		}
	\end{tabular}
	\caption{Error for Pareschi-Russo IMEX problem~\cite{pareschi2000implicit} with $\eps=1$ at $\Tend=5$ with our IMEX-MD schemes. We use a fixed $\km = 9$ for all the simulations.  We display the impact of increasing the iterate number $k$ and the fact that each additional correction improves the overall order of accuracy by one, save the final one.  We compare side-by-side the analytically investigated Alg.~\ref{alg:mdode} (Left column) and the modified Alg.~\ref{alg:mdodejonas} (Right column).  The modifications made for Alg.~\ref{alg:mdodejonas} increase the order of the predictor and therefore the following correction steps by one until the maximum order of convergence is reached.  Moreover, the errors especially for large timesteps are significantly reduced with Alg.~\ref{alg:mdodejonas}.}\label{fig:PR_predictor}
\end{figure}
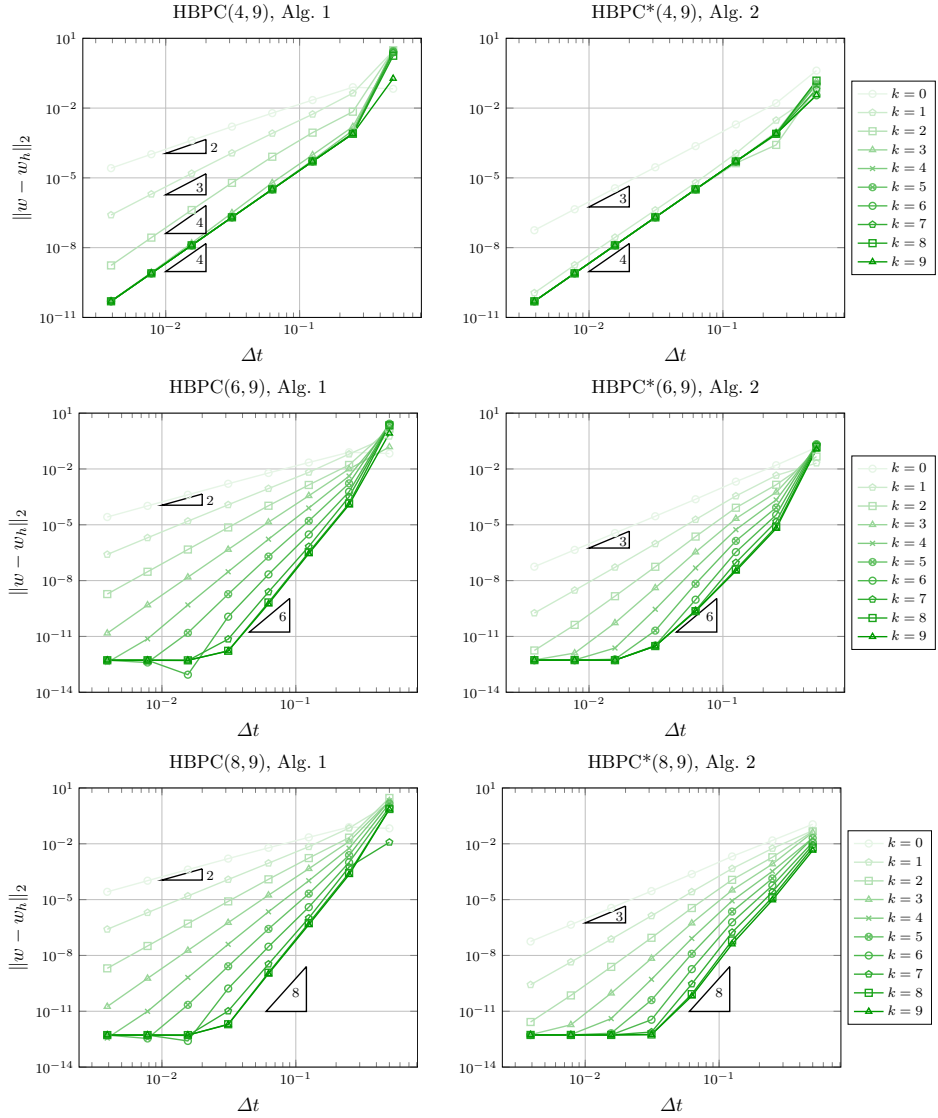

We make the following observations concerning the additional modifications:
\begin{itemize}
\item \textbf{Modification of the predictor.}
The modification of the predictor has the drawback that the predictor loses its independence from the correction steps as it now uses the solution of the first corrector step as the solution at $t^n$.
In Fig.~\ref{fig:PR_predictor} we see the influence of modifying the predictor step: the predictor's order of accuracy is increased from second to third order. This results in an additional order of the following correction steps, of course only until the maximum achievable order is reached. For the $6^{th}$- and $8^{th}$-order method we additionally observe improvements in the error found after the last iteration.

\item \textbf{Modification of the corrector.}
The primary advantage in the modification in the corrector is that less overall storage is required for the solver. Simply put, as the stage values of the previous iteration are replaced with the stage values from the current iteration $k+1$, no additional storage array for the old stage values is required. Moreover, potentially better values are used for the quadrature rule. We call this modification ``Gau\ss-Seidel style" as ``better" information is used as soon as it is available~\cite{Saad2003}. Note that similar ideas have been pursued for an integral deferred correction method in~\cite{Crockatt2018}.
\end{itemize}

In Fig.~\ref{fig:PR_predictor} we can also observe the influence of using the Gau\ss-Seidel style corrector by comparing the iterates to each other.  The improved quadrature significantly decreased the error. This effect is especially visible for large timesteps since each correction step improves the solution with $\mathcal{O}(\Delta t)$ (of course only until the maximum order of convergence is reached).  Although not shown here, this effect is even more pronounced for stiff problems. Moreover, the required amount of Newton iterations is significantly reduced for such cases that require large timesteps.  This could certainly be beneficial for multiscale problems.  Hence, using values of the predictor that are ``one correction step better'' as soon as they are available facilitates the solution of the non-linear problem and at the same time allows for more accurate results.

\subsection{Parallel performance}

We use the improved Alg.~\ref{alg:mdodejonas} to test the parallel performance of the proposed method.  In order to do this, we run Alg.~\ref{alg:mdodejonas} in both serial and parallel settings. The method is implemented in \texttt{MATLAB} using the parallel computing toolbox \cite{ParallelToolbox}. Calculations are done on one Intel skylake node with 2 Xeon Gold 6140 CPUs@2.3 GHz, with 18 cores each, provided by the Vlaams Supercomputing Centrum (VSC). 
The algorithm necessitates the solution of a non-linear system of equations 
\begin{align*}
F(\tilde{w}):=f(\tilde{w})-\text{rhs}=0,  
\end{align*}
with
\begin{align*}
	\tilde{w}&=\algs{w}{n}{0}{l},\quad f(\tilde{w})=\algs{w}{n}{0}{l}-c_l\Delta t\left(\algs{\Phii}{n}{0}{l}\right)+\frac{(c_l\Delta t)^2}{2}\left(\algs{\dPhii}{n}{0}{l}\right),\\
	\text{rhs}&=\algs{w}{n-1}{1}{s}+c_l\Delta t\left(\algs{\Phie}{n-1}{1}{s}\right)+\frac{(c_l\Delta t)^2}{2}\left(\algs{\dPhie}{n-1}{1}{s}\right),
\end{align*}
for the predictor and
\begin{align*}
	\tilde{w}&=\algs{w}{n}{k+1}{l},\quad f(\tilde{w})=\algs{w}{n}{k+1}{l}-\Delta t\left(\algs{\Phii}{n}{k+1}{l}\right)+\frac{\Delta t^2}{2}\left(\algs{\dPhii}{n}{k+1}{l}\right),\\
	\text{rhs}&=\algs{w}{n-1}{k+2}{s}-\Delta t\left(\algs{\Phii}{n}{k}{l}\right)+\frac{\Delta t^2}{2}\left(\algs{\dPhii}{n}{k}{l}\right)+\mathcal{I}_l,
\end{align*}
for the corrector. We choose to solve this using a damped Newton's method with starting point $\tilde{w}_0=\algs{w}{n-1}{1}{s}$ for the predictor and $\tilde{w}_0=\algs{w}{n-1}{k+2}{s}$ for the corrector. We define a relative convergence criterion of
\begin{align*}
	\frac{\|F(\tilde{w})\|_2}{\|F(\tilde{w}_0)\|_2}\le\eps_{\text{Newton}}=10^{-6}
\end{align*}
and an absolute convergence criterion of $\|F(\tilde{w})\|_2\le\eps'_{\text{Newton}}=10^{-14}$. A maximum of $1000$ iterations is allowed. Starting with a damping factor of $1$, the damping factor is halved if one Newton step's residual exceeds $0.9$ times the previous Newton step's residual. If none of the convergence criteria is met after $1000$ iterations, we mark the solver as converged. Typically, when this happens it is then due to machine accuracy, and not to a lack of convergence for Newton's method. 

The hierarchical structure of the method allows us to improve the iterative scheme by replacing the initial $\tilde{w}$ with $\algs{w}{n}{k}{l}$ for the correction steps. This improves the starting point of Newton's method and can reduce the amount of Newton iterations. We could also think about replacing $\tilde{w}_0$ with $\algs{w}{n}{k}{l}$. This would have an influence on the results presented below. Nevertheless, we observe that these changes do not significantly alter the reported results. Hence, we can only conclude that the proposed method is quite robust against such algorithmic variants.

\ifthenelse{\boolean{compilefromscratch}}{
\pgfplotstableread[col sep = comma] {../Code/results_alg2/results_imexmd_parallel_rkrk4_kmax4_PR_IMEX_eps1.csv}\PRaa
\pgfplotstablecreatecol[create col/copy column from table= {../Code/results_alg2/results_imexmd_serialrk4_kmax3_PR_IMEX_eps1.csv}{wallclocktime}]{wallclocktime_serial}{\PRaa}
\pgfplotstableread[col sep = comma] {../Code/results_alg2/results_imexmd_parallel_rkrk6_kmax4_PR_IMEX_eps1.csv}\PRab
\pgfplotstablecreatecol[create col/copy column from table= {../Code/results_alg2/results_imexmd_serialrk6_kmax3_PR_IMEX_eps1.csv}{wallclocktime}]{wallclocktime_serial}{\PRab}
\pgfplotstableread[col sep = comma] {../Code/results_alg2/results_imexmd_parallel_rkrk8_kmax4_PR_IMEX_eps1.csv}\PRac
\pgfplotstablecreatecol[create col/copy column from table= {../Code/results_alg2/results_imexmd_serialrk8_kmax3_PR_IMEX_eps1.csv}{wallclocktime}]{wallclocktime_serial}{\PRac}
\pgfplotstableread[col sep = comma] {../Code/results_alg2/results_imexmd_parallel_rkrk4_kmax4_vdP_IMEX_eps0.001.csv}\vdPaa
\pgfplotstablecreatecol[create col/copy column from table= {../Code/results_alg2/results_imexmd_serialrk4_kmax3_vdP_IMEX_eps0.001.csv}{wallclocktime}]{wallclocktime_serial}{\vdPaa}
\pgfplotstableread[col sep = comma] {../Code/results_alg2/results_imexmd_parallel_rkrk6_kmax4_vdP_IMEX_eps0.001.csv}\vdPab
\pgfplotstablecreatecol[create col/copy column from table= {../Code/results_alg2/results_imexmd_serialrk6_kmax3_vdP_IMEX_eps0.001.csv}{wallclocktime}]{wallclocktime_serial}{\vdPab}
\pgfplotstableread[col sep = comma] {../Code/results_alg2/results_imexmd_parallel_rkrk8_kmax4_vdP_IMEX_eps0.001.csv}\vdPac
\pgfplotstablecreatecol[create col/copy column from table= {../Code/results_alg2/results_imexmd_serialrk8_kmax3_vdP_IMEX_eps0.001.csv}{wallclocktime}]{wallclocktime_serial}{\vdPac}

\pgfplotstableread[col sep = comma] {../Code/results_alg2/results_imexmd_parallel_rkrk4_kmax8_PR_IMEX_eps1.csv}\PRba
\pgfplotstablecreatecol[create col/copy column from table= {../Code/results_alg2/results_imexmd_serialrk4_kmax7_PR_IMEX_eps1.csv}{wallclocktime}]{wallclocktime_serial}{\PRba}
\pgfplotstableread[col sep = comma] {../Code/results_alg2/results_imexmd_parallel_rkrk6_kmax8_PR_IMEX_eps1.csv}\PRbb
\pgfplotstablecreatecol[create col/copy column from table= {../Code/results_alg2/results_imexmd_serialrk6_kmax7_PR_IMEX_eps1.csv}{wallclocktime}]{wallclocktime_serial}{\PRbb}
\pgfplotstableread[col sep = comma] {../Code/results_alg2/results_imexmd_parallel_rkrk8_kmax8_PR_IMEX_eps1.csv}\PRbc
\pgfplotstablecreatecol[create col/copy column from table= {../Code/results_alg2/results_imexmd_serialrk8_kmax7_PR_IMEX_eps1.csv}{wallclocktime}]{wallclocktime_serial}{\PRbc}
\pgfplotstableread[col sep = comma] {../Code/results_alg2/results_imexmd_parallel_rkrk4_kmax8_vdP_IMEX_eps0.001.csv}\vdPba
\pgfplotstablecreatecol[create col/copy column from table= {../Code/results_alg2/results_imexmd_serialrk4_kmax7_vdP_IMEX_eps0.001.csv}{wallclocktime}]{wallclocktime_serial}{\vdPba}
\pgfplotstableread[col sep = comma] {../Code/results_alg2/results_imexmd_parallel_rkrk6_kmax8_vdP_IMEX_eps0.001.csv}\vdPbb
\pgfplotstablecreatecol[create col/copy column from table= {../Code/results_alg2/results_imexmd_serialrk6_kmax7_vdP_IMEX_eps0.001.csv}{wallclocktime}]{wallclocktime_serial}{\vdPbb}
\pgfplotstableread[col sep = comma] {../Code/results_alg2/results_imexmd_parallel_rkrk8_kmax8_vdP_IMEX_eps0.001.csv}\vdPbc
\pgfplotstablecreatecol[create col/copy column from table= {../Code/results_alg2/results_imexmd_serialrk8_kmax7_vdP_IMEX_eps0.001.csv}{wallclocktime}]{wallclocktime_serial}{\vdPbc}

\pgfplotstableread[col sep = comma] {../Code/results_alg2/results_imexmd_parallel_rkrk4_kmax36_PR_IMEX_eps1.csv}\PRca
\pgfplotstablecreatecol[create col/copy column from table= {../Code/results_alg2/results_imexmd_serialrk4_kmax35_PR_IMEX_eps1.csv}{wallclocktime}]{wallclocktime_serial}{\PRca}
\pgfplotstableread[col sep = comma] {../Code/results_alg2/results_imexmd_parallel_rkrk6_kmax36_PR_IMEX_eps1.csv}\PRcb
\pgfplotstablecreatecol[create col/copy column from table= {../Code/results_alg2/results_imexmd_serialrk6_kmax35_PR_IMEX_eps1.csv}{wallclocktime}]{wallclocktime_serial}{\PRcb}
\pgfplotstableread[col sep = comma] {../Code/results_alg2/results_imexmd_parallel_rkrk8_kmax36_PR_IMEX_eps1.csv}\PRcc
\pgfplotstablecreatecol[create col/copy column from table= {../Code/results_alg2/results_imexmd_serialrk8_kmax35_PR_IMEX_eps1.csv}{wallclocktime}]{wallclocktime_serial}{\PRcc}
\pgfplotstableread[col sep = comma] {../Code/results_alg2/results_imexmd_parallel_rkrk4_kmax36_vdP_IMEX_eps0.001.csv}\vdPca
\pgfplotstablecreatecol[create col/copy column from table= {../Code/results_alg2/results_imexmd_serialrk4_kmax35_vdP_IMEX_eps0.001.csv}{wallclocktime}]{wallclocktime_serial}{\vdPca}
\pgfplotstableread[col sep = comma] {../Code/results_alg2/results_imexmd_parallel_rkrk6_kmax36_vdP_IMEX_eps0.001.csv}\vdPcb
\pgfplotstablecreatecol[create col/copy column from table= {../Code/results_alg2/results_imexmd_serialrk6_kmax35_vdP_IMEX_eps0.001.csv}{wallclocktime}]{wallclocktime_serial}{\vdPcb}
\pgfplotstableread[col sep = comma] {../Code/results_alg2/results_imexmd_parallel_rkrk8_kmax36_vdP_IMEX_eps0.001.csv}\vdPcc
\pgfplotstablecreatecol[create col/copy column from table= {../Code/results_alg2/results_imexmd_serialrk8_kmax35_vdP_IMEX_eps0.001.csv}{wallclocktime}]{wallclocktime_serial}{\vdPcc}

\pgfplotstableread[col sep = comma] {../Code/results_alg2/results_imexmd_parallel_rkrk4_kmax72_PR_IMEX_eps1.csv}\PRda
\pgfplotstablecreatecol[create col/copy column from table= {../Code/results_alg2/results_imexmd_serialrk4_kmax71_PR_IMEX_eps1.csv}{wallclocktime}]{wallclocktime_serial}{\PRda}
\pgfplotstableread[col sep = comma] {../Code/results_alg2/results_imexmd_parallel_rkrk6_kmax72_PR_IMEX_eps1.csv}\PRdb
\pgfplotstablecreatecol[create col/copy column from table= {../Code/results_alg2/results_imexmd_serialrk6_kmax71_PR_IMEX_eps1.csv}{wallclocktime}]{wallclocktime_serial}{\PRdb}
\pgfplotstableread[col sep = comma] {../Code/results_alg2/results_imexmd_parallel_rkrk8_kmax72_PR_IMEX_eps1.csv}\PRdc
\pgfplotstablecreatecol[create col/copy column from table= {../Code/results_alg2/results_imexmd_serialrk8_kmax71_PR_IMEX_eps1.csv}{wallclocktime}]{wallclocktime_serial}{\PRdc}
\pgfplotstableread[col sep = comma] {../Code/results_alg2/results_imexmd_parallel_rkrk4_kmax72_vdP_IMEX_eps0.001.csv}\vdPda
\pgfplotstablecreatecol[create col/copy column from table= {../Code/results_alg2/results_imexmd_serialrk4_kmax71_vdP_IMEX_eps0.001.csv}{wallclocktime}]{wallclocktime_serial}{\vdPda}
\pgfplotstableread[col sep = comma] {../Code/results_alg2/results_imexmd_parallel_rkrk6_kmax72_vdP_IMEX_eps0.001.csv}\vdPdb
\pgfplotstablecreatecol[create col/copy column from table= {../Code/results_alg2/results_imexmd_serialrk6_kmax71_vdP_IMEX_eps0.001.csv}{wallclocktime}]{wallclocktime_serial}{\vdPdb}
\pgfplotstableread[col sep = comma] {../Code/results_alg2/results_imexmd_parallel_rkrk8_kmax72_vdP_IMEX_eps0.001.csv}\vdPdc
\pgfplotstablecreatecol[create col/copy column from table= {../Code/results_alg2/results_imexmd_serialrk8_kmax71_vdP_IMEX_eps0.001.csv}{wallclocktime}]{wallclocktime_serial}{\vdPdc}
}

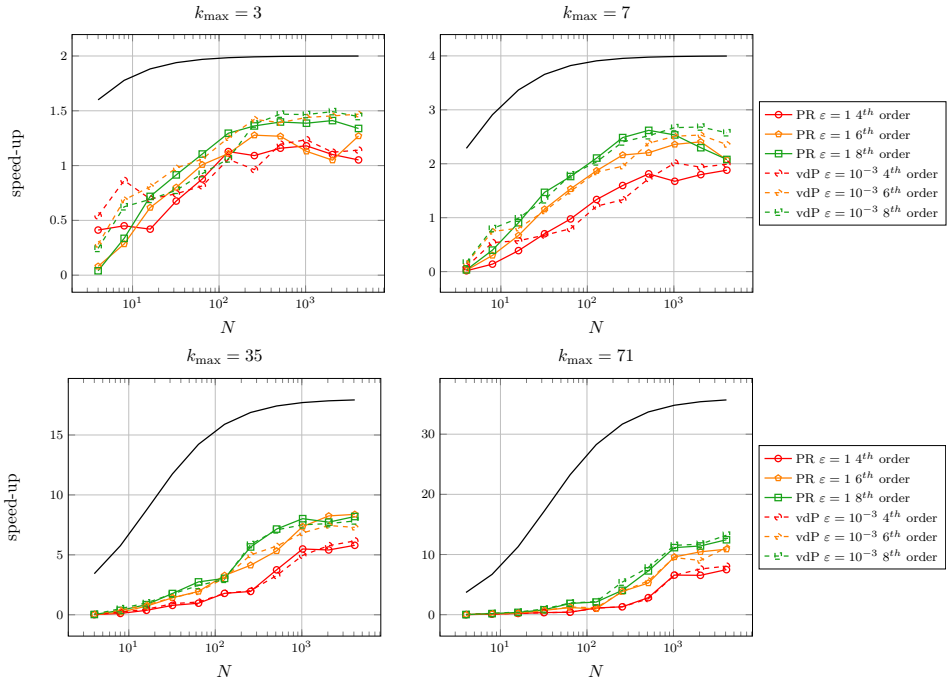
\begin{figure}
	\centering
	\begin{tabular}{cc}
        \ifthenelse{\boolean{compilefromscratch}}{
        
        \tikzsetnextfilename{Scaling1}
		\begin{tikzpicture}[scale=0.6]
			\begin{semilogxaxis}[cycle list name=epslist,xlabel={$N$},ylabel={speed-up} ,grid=major,legend style={at={(0.98,0.02)},anchor= south east,font=\footnotesize},title={$k_\text{max}=3$},label style={font=\large},title style={font=\large},legend cell align={left}]
			
			\addplot table [skip first n=1,x expr={\thisrowno{0}}, y expr={(\thisrow{wallclocktime_serial}/ \thisrow{wallclocktime})}] {\PRaa};
			\addplot table [skip first n=1,x expr={\thisrowno{0}}, y expr={(\thisrow{wallclocktime_serial}/ \thisrow{wallclocktime})}] {\PRab};
			\addplot table [skip first n=1,x expr={\thisrowno{0}}, y expr={(\thisrow{wallclocktime_serial}/ \thisrow{wallclocktime})}] {\PRac};
			\pgfplotsset{cycle list shift=3}
			\addplot table [skip first n=1,x expr={\thisrowno{0}}, y expr={(\thisrow{wallclocktime_serial}/ \thisrow{wallclocktime})}] {\vdPaa};
			\addplot table [skip first n=1,x expr={\thisrowno{0}}, y expr={(\thisrow{wallclocktime_serial}/ \thisrow{wallclocktime})}] {\vdPab};
			\addplot table [skip first n=1,x expr={\thisrowno{0}}, y expr={(\thisrow{wallclocktime_serial}/ \thisrow{wallclocktime})}] {\vdPac};
			\addplot[black,thick] table [skip first n=1,x expr={\thisrowno{0}}, y expr={\thisrowno{0}*(3+1)/(2*\thisrowno{0}+3-1)}] {\PRaa};
			\end{semilogxaxis}
		\end{tikzpicture}&
		\tikzsetnextfilename{Scaling2}
		\begin{tikzpicture}[scale=0.6]
			\begin{semilogxaxis}[cycle list name=epslist,xlabel={$N$},grid=major,legend style={at={(1.02,0.5)},anchor= west,font=\footnotesize},title={$k_\text{max}=7$},label style={font=\large},title style={font=\large},legend cell align={left}]
			\addplot table [skip first n=1,x expr={\thisrowno{0}}, y expr={(\thisrow{wallclocktime_serial}/ \thisrow{wallclocktime})}] {\PRba};
			\addplot table [skip first n=1,x expr={\thisrowno{0}}, y expr={(\thisrow{wallclocktime_serial}/ \thisrow{wallclocktime})}] {\PRbb};
			\addplot table [skip first n=1,x expr={\thisrowno{0}}, y expr={(\thisrow{wallclocktime_serial}/ \thisrow{wallclocktime})}] {\PRbc};
			\pgfplotsset{cycle list shift=3}
			\addplot table [skip first n=1,x expr={\thisrowno{0}}, y expr={(\thisrow{wallclocktime_serial}/ \thisrow{wallclocktime})}] {\vdPba};
			\addplot table [skip first n=1,x expr={\thisrowno{0}}, y expr={(\thisrow{wallclocktime_serial}/ \thisrow{wallclocktime})}] {\vdPbb};
			\addplot table [skip first n=1,x expr={\thisrowno{0}}, y expr={(\thisrow{wallclocktime_serial}/ \thisrow{wallclocktime})}] {\vdPbc};			
			
			\addplot[black,thick] table [skip first n=1,x expr={\thisrowno{0}}, y expr={\thisrowno{0}*(7+1)/(2*\thisrowno{0}+7-1)}] {\PRba};
			\legend{PR $\eps=1$ $4^{th}$ order,PR $\eps=1$ $6^{th}$ order,PR $\eps=1$ $8^{th}$ order,vdP $\eps=10^{-3}$ $4^{th}$ order,vdP $\eps=10^{-3}$ $6^{th}$ order,vdP $\eps=10^{-3}$ $8^{th}$ order}
			\end{semilogxaxis}
		\end{tikzpicture}\\
		\tikzsetnextfilename{Scaling3}
		\begin{tikzpicture}[scale=0.6]
			\begin{semilogxaxis}[cycle list name=epslist,xlabel={$N$},ylabel={speed-up} ,grid=major,legend style={at={(0.02,0.98)},anchor= north west,font=\footnotesize},title={$k_\text{max}=35$},label style={font=\large},title style={font=\large},legend cell align={left}]
		
			\addplot table [skip first n=1,x expr={\thisrowno{0}}, y expr={(\thisrow{wallclocktime_serial}/ \thisrow{wallclocktime})}] {\PRca};
			\addplot table [skip first n=1,x expr={\thisrowno{0}}, y expr={(\thisrow{wallclocktime_serial}/ \thisrow{wallclocktime})}] {\PRcb};
			\addplot table [skip first n=1,x expr={\thisrowno{0}}, y expr={(\thisrow{wallclocktime_serial}/ \thisrow{wallclocktime})}] {\PRcc};
			\pgfplotsset{cycle list shift=3}
			\addplot table [skip first n=1,x expr={\thisrowno{0}}, y expr={(\thisrow{wallclocktime_serial}/ \thisrow{wallclocktime})}] {\vdPca};
			\addplot table [skip first n=1,x expr={\thisrowno{0}}, y expr={(\thisrow{wallclocktime_serial}/ \thisrow{wallclocktime})}] {\vdPcb};
			\addplot table [skip first n=1,x expr={\thisrowno{0}}, y expr={(\thisrow{wallclocktime_serial}/ \thisrow{wallclocktime})}] {\vdPcc};			
			
			\addplot[black,thick] table [skip first n=1,x expr={\thisrowno{0}}, y expr={\thisrowno{0}*(35+1)/(2*\thisrowno{0}+35-1)}] {\PRca};
			\end{semilogxaxis}
		\end{tikzpicture}&
		\tikzsetnextfilename{Scaling4}
		\begin{tikzpicture}[scale=0.6]
			\begin{semilogxaxis}[cycle list name=epslist,xlabel={$N$},grid=major,legend style={at={(1.02,0.5)},anchor= west,font=\footnotesize},title={$k_\text{max}=71$},label style={font=\large},title style={font=\large},legend cell align={left}]
			
			\addplot table [skip first n=1,x expr={\thisrowno{0}}, y expr={(\thisrow{wallclocktime_serial}/ \thisrow{wallclocktime})}] {\PRda};
			\addplot table [skip first n=1,x expr={\thisrowno{0}}, y expr={(\thisrow{wallclocktime_serial}/ \thisrow{wallclocktime})}] {\PRdb};
			\addplot table [skip first n=1,x expr={\thisrowno{0}}, y expr={(\thisrow{wallclocktime_serial}/ \thisrow{wallclocktime})}] {\PRdc};
			\pgfplotsset{cycle list shift=3}
			\addplot table [skip first n=1,x expr={\thisrowno{0}}, y expr={(\thisrow{wallclocktime_serial}/ \thisrow{wallclocktime})}] {\vdPda};
			\addplot table [skip first n=1,x expr={\thisrowno{0}}, y expr={(\thisrow{wallclocktime_serial}/ \thisrow{wallclocktime})}] {\vdPdb};
			\addplot table [skip first n=1,x expr={\thisrowno{0}}, y expr={(\thisrow{wallclocktime_serial}/ \thisrow{wallclocktime})}] {\vdPdc};			
			
			\addplot[black,thick] table [skip first n=1,x expr={\thisrowno{0}}, y expr={\thisrowno{0}*(71+1)/(2*\thisrowno{0}+71-1)}] {\PRda};
			\legend{PR $\eps=1$ $4^{th}$ order,PR $\eps=1$ $6^{th}$ order,PR $\eps=1$ $8^{th}$ order,vdP $\eps=10^{-3}$ $4^{th}$ order,vdP $\eps=10^{-3}$ $6^{th}$ order,vdP $\eps=10^{-3}$ $8^{th}$ order}
			\end{semilogxaxis}
		\end{tikzpicture}
		}
		{
            \includegraphics{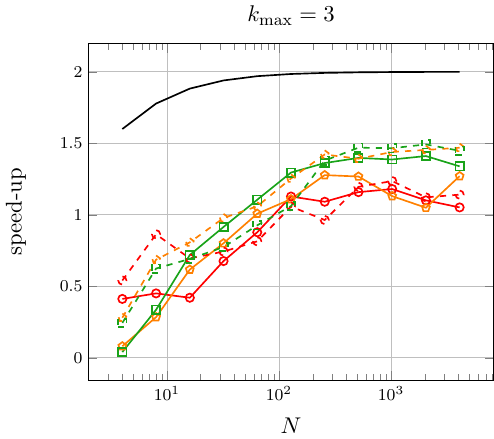}&
            \includegraphics{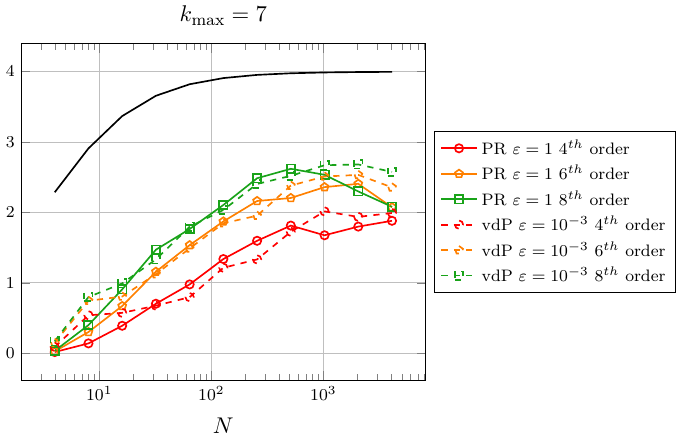}\\
            \includegraphics{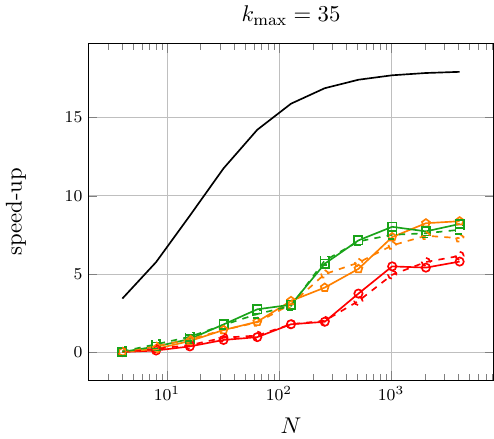}&
            \includegraphics{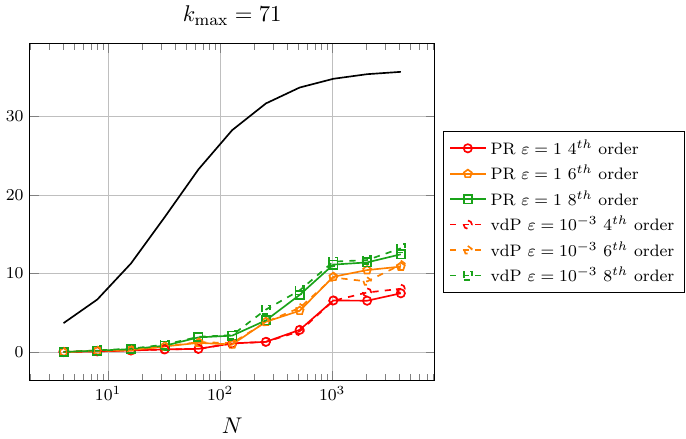}
		}
	\end{tabular}
	\caption{Speed-up over number of timesteps for different $k_\text{max}$ with non-stiff Pareschi-Russo problem (solid lines) and stiff van-der-Pol equation (dashed lines) for $\methodt{4}{\km}$ (red), $\methodt{6}{\km}$ (orange) and $\methodt{8}{\km}$ (green) with the more efficient Alg.~\ref{alg:mdodejonas}. Black line indicates theoretically achievable speed-up $\frac{N\cdot(\km+1)}{2N+\km-1}$.}\label{fig:Speedup}
\end{figure}

The results regarding speed-up are reported in Fig.~\ref{fig:Speedup} for an increasing number of correction steps $\km$ and hence also an increasing number of processors ranging from $\#\text{procs}=2$ ($\km=3$) to $\#\text{procs}=36$ ($\km=71$). We choose a non-stiff (PR-IMEX with $\eps=1$) and a stiff (vdP-IMEX with $\eps=10^{-3}$) problem to see if this has an influence on the parallel performance. All calculations are done with the $4^{th}$, $6^{th}$ and $8^{th}$ order method. The figure shows that the achieved speed-ups are quite similar for both considered problems, but differently for the three different quadrature rules: with $\methodt{8}{\km}$ the highest speed-up can be achieved. This is most probably caused by the more `processor-local' work due to the more stages compared to the other considered schemes. As the communication introduces some overhead, the scaling properties of the algorithm are increased by performing more processor-local operations for an almost similar amount of communication. This is a typical observation made when parallelizing numerical methods. It has, e.g., also been observed for high-order discontinuous Galerkin methods where a better parallel performance can be achieved if higher order ansatz polynomials are chosen as they increase the amount of processor-local work~\cite{MunzDG12}. As the parallel performance results are almost the same for the stiff and non-stiff problem, we anticipate that the obtained performance gain is transferable to other problems.

One can see that if too few timesteps are used, no speed-up is achieved and a deceleration of the simulation is obtained. This can be caused by the overhead introduced by establishing the communication. Nevertheless, if more timesteps are used, a significant speed-up can be achieved with the proposed parallelization strategy. Considering for example the case with $36$ processors ($\km=71$), one can achieve a speed-up of up to a factor of $\approx13$ for $\methodt{8}{\km}$, $\approx11$ for $\methodt{6}{\km}$ and up to $\approx8$ for $\methodt{4}{\km}$. These values are in the same range as reported in the review paper by Ong and Schroder~\cite{OngSchroder2020}. In terms of their paper, our method belongs to the class of direct time-parallel methods. 

The maximum theoretical speed-up is not achieved for all cases. There can be two reasons for this: First, the overhead introduced by the communication slows down the computations. Second, the processor-local work is distributed unevenly. The first point is strongly influenced by the current implementation and architecture and is therefore beyond the scope of this work. To gain more insight into the amount of processor-local work we consider the amount of Newton iterations performed by each processor.

\begin{figure}
	\centering
    \begin{tabular}{cc}
	\ifthenelse{\boolean{compilefromscratch}}{
        
    \tikzsetnextfilename{Newton1}
		\begin{tikzpicture}[scale=0.6]
		\begin{loglogaxis}[cycle list name=epslist,xlabel={$N$},ylabel={Newton iterations} ,grid=major,legend style={at={(0.98,0.02)},anchor= south east,font=\footnotesize},title={PR-IMEX, $\methodt{8}{71}$, $\eps=1$},label style={font=\large},title style={font=\large},legend cell align={left}]
		\addplot table[skip first n=1,x expr={\thisrowno{0}}, y expr={\thisrowno{74}}] {../Code/results_alg2/results_imexmd_parallel_rkrk8_kmax72_PR_IMEX_eps1.csv};
		\addplot table[skip first n=1,x expr={\thisrowno{0}}, y expr={\thisrowno{75}}] {../Code/results_alg2/results_imexmd_parallel_rkrk8_kmax72_PR_IMEX_eps1.csv};
		\addplot table[skip first n=1,x expr={\thisrowno{0}}, y expr={\thisrowno{76}}] {../Code/results_alg2/results_imexmd_parallel_rkrk8_kmax72_PR_IMEX_eps1.csv};
		\addplot table[skip first n=1,x expr={\thisrowno{0}}, y expr={\thisrowno{77}}] {../Code/results_alg2/results_imexmd_parallel_rkrk8_kmax72_PR_IMEX_eps1.csv};
		\addplot table[skip first n=1,x expr={\thisrowno{0}}, y expr={\thisrowno{78}}] {../Code/results_alg2/results_imexmd_parallel_rkrk8_kmax72_PR_IMEX_eps1.csv};
		\addplot table[skip first n=1,x expr={\thisrowno{0}}, y expr={\thisrowno{79}}] {../Code/results_alg2/results_imexmd_parallel_rkrk8_kmax72_PR_IMEX_eps1.csv};
		\addplot table[skip first n=1,x expr={\thisrowno{0}}, y expr={\thisrowno{109}}] {../Code/results_alg2/results_imexmd_parallel_rkrk8_kmax72_PR_IMEX_eps1.csv};
		\legend{$k=0\&1$,$k=2\&3$,$k=4\&5$,$k=6\&7$,$k=8\&9$,$k=10\&11$,$k=70\&71$}
		\end{loglogaxis}
		\end{tikzpicture}&
		\tikzsetnextfilename{Newton2}
		\begin{tikzpicture}[scale=0.6]
		\begin{loglogaxis}[cycle list name=epslist,xlabel={$N$},grid=major,legend style={at={(0.98,0.02)},anchor= south east,font=\footnotesize},title={vdP-IMEX, $\methodt{8}{71}$, $\eps=10^{-3}$},label style={font=\large},title style={font=\large},legend cell align={left}]
		\addplot table[skip first n=1,x expr={\thisrowno{0}}, y expr={\thisrowno{74}}] {../Code/results_alg2/results_imexmd_parallel_rkrk8_kmax72_vdP_IMEX_eps0.001.csv};
		\addplot table[skip first n=1,x expr={\thisrowno{0}}, y expr={\thisrowno{75}}] {../Code/results_alg2/results_imexmd_parallel_rkrk8_kmax72_vdP_IMEX_eps0.001.csv};
		\addplot table[skip first n=1,x expr={\thisrowno{0}}, y expr={\thisrowno{76}}] {../Code/results_alg2/results_imexmd_parallel_rkrk8_kmax72_vdP_IMEX_eps0.001.csv};
		\addplot table[skip first n=1,x expr={\thisrowno{0}}, y expr={\thisrowno{77}}] {../Code/results_alg2/results_imexmd_parallel_rkrk8_kmax72_vdP_IMEX_eps0.001.csv};
		\addplot table[skip first n=1,x expr={\thisrowno{0}}, y expr={\thisrowno{78}}] {../Code/results_alg2/results_imexmd_parallel_rkrk8_kmax72_vdP_IMEX_eps0.001.csv};
		\addplot table[skip first n=1,x expr={\thisrowno{0}}, y expr={\thisrowno{79}}] {../Code/results_alg2/results_imexmd_parallel_rkrk8_kmax72_vdP_IMEX_eps0.001.csv};
		\addplot table[skip first n=1,x expr={\thisrowno{0}}, y expr={\thisrowno{109}}] {../Code/results_alg2/results_imexmd_parallel_rkrk8_kmax72_vdP_IMEX_eps0.001.csv};
		\legend{$k=0\&1$,$k=2\&3$,$k=4\&5$,$k=6\&7$,$k=8\&9$,$k=10\&11$,$k=70\&71$}
		\end{loglogaxis}
		\end{tikzpicture}
		}
		{
            \includegraphics{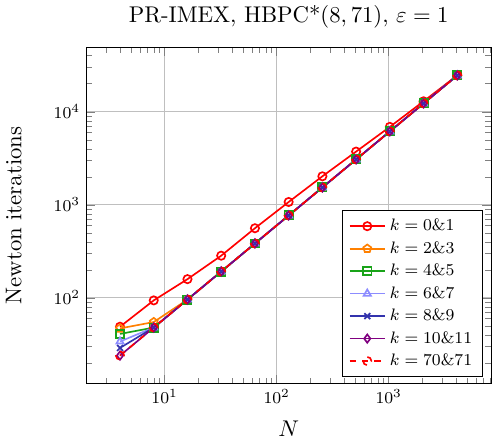}&
            \includegraphics{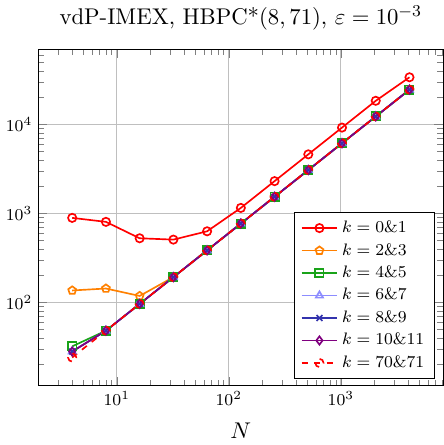}
		}
	\end{tabular}
	\caption{Required amount of Newton iterations for different correction steps (grouped by the processor on which they are performed) and two selected test setups (left: PR-IMEX with $\eps=1$; right: vdP-IMEX with $\eps=10^{-3}$) over the number of timesteps.}\label{fig:Newton_iterations}
\end{figure}
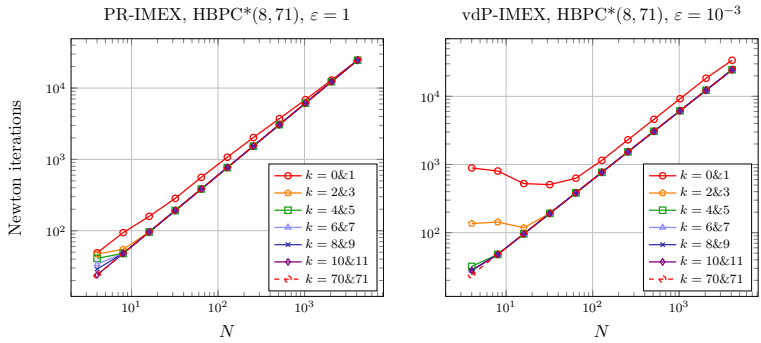

In Fig.~\ref{fig:Newton_iterations} the Newton iterations on the first $6$ and the last processor are shown for the $\methodt{8}{\km}$ scheme for the non-stiff PR-IMEX problem and the stiff vdP-IMEX equation. It is shown that the required amount of iterations decreases with an increasing index of the correction step, i.e. with the index of the processor. Especially the first processor which is responsible for the predictor and the first corrector step requires significantly more iterations than the other processors. This trend is even enhanced for the considered stiff problem. To cure this imbalance one can either think about a different distribution of the processors than presented in Fig.~\ref{fig:timeparallel}, or a termination criterion for Newton's method depending on the current index of the correction step.

\subsection{Arenstorf Orbit}
Finally, we use the Arenstorf orbit problem~\cite{HaiWan1,OngSpiteri} for further illustration of the proposed method's capabilities. The Arenstorf orbit problem describes a three body problem, where the movement of a light object is influenced by two heavy objects. This can, for example, be a satellite being influenced by two planets. The problem is actually a second-order differential equation, here formulated as a system of first-order ODEs:
\begin{align*}
	w'=\begin{pmatrix}w_3\\w_4\\w_1+2w_4-\mu'\frac{w_1+\mu}{D_1}-\mu\frac{w_1-\mu'}{D_2}\\w_2-2w_3-\mu'\frac{w_2}{D_1}-\mu\frac{w_2}{D_2}\end{pmatrix}.
\end{align*}
Here, we have defined
\begin{align*}
	D_1&:=\left((w_1+\mu)^2+w_2^2\right)^{\frac{3}{2}},\quad D_2:=\left((w_1-\mu')^2+w_2^2\right)^{\frac{3}{2}},\\ \mu&:=0.012277471,\quad\mu':=1-\mu,
\end{align*}
and the initial conditions
\begin{align*}
	w_0=\left(0.994,~0,~0,~-2.001585106379\right)^T.
\end{align*}
Although there is no multi-scale character of the problem, we artificially split the equation into an explicit and an implicit part. For that purpose, all parts of the equation which are divided by $D_1$ or $D_2$ are simply treated implicitly, the remaining parts are treated explicitly.

The solution is a closed orbit with a period of $17.065216560159$. Similar to~\cite{OngSpiteri}, we choose $10^5$ equidistant timesteps to simulate one period of the problem. We choose the $\methodt{8}{\km}$ method with Alg.~\ref{alg:mdodejonas} for our calculations, which results in an error after one period of $\|w-w_0\|_2=1.7818\cdot10^{-9}$ for $\km=71$.

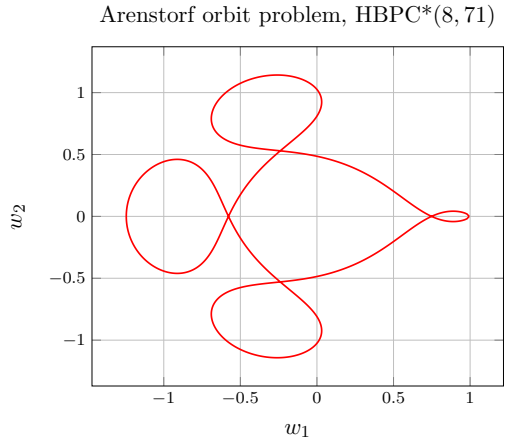
\begin{figure}[ht]
	\centering
    \ifthenelse{\boolean{compilefromscratch}}{
        
    \tikzsetnextfilename{Arenstorf1}
	\begin{tikzpicture}[scale=0.79]
		\begin{axis}[cycle list name=epslist,xlabel={$w_1$},ylabel={$w_2$} ,grid=major,legend style={at={(0.02,0.02)},anchor= south west,font=\footnotesize},title={Arenstorf orbit problem, $\methodt{8}{71}$},label style={font=\large},title style={font=\large},legend cell align={left}]
		\addplot[mark=none,thick,red] table [skip first n=1,x expr={\thisrowno{1}}, y expr={\thisrowno{2}}] {./Arenstorf_Orbit_Solution.csv};
		\end{axis}
	\end{tikzpicture}
	}
	{
	  \includegraphics{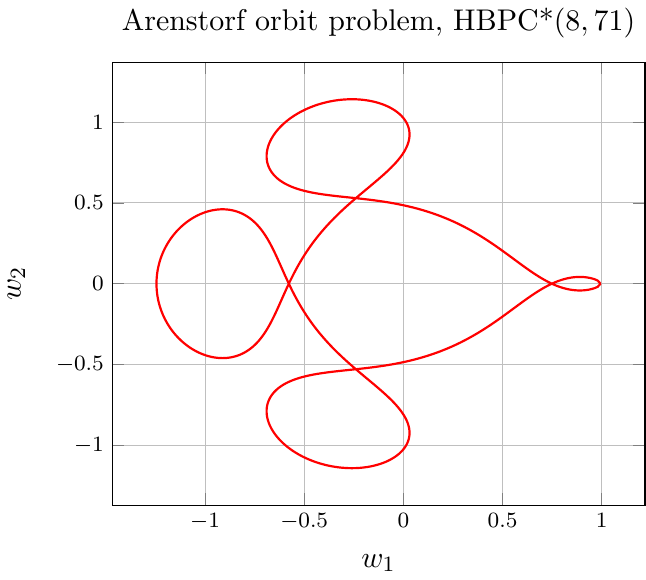}
	}
	\caption{Solution of the Arenstorf orbit problem for the two spatial variables. $10^5$ timesteps of the $8^{th}$ order IMEX-MD scheme described in Alg.~\ref{alg:mdodejonas} are used to simulate one period.}\label{fig:Arenstorf}
\end{figure}

In Fig.~\ref{fig:Arenstorf} the solution of $w_1$ and $w_2$ for the Arenstorf orbit problem is shown. One can see that after one period, the initial condition is reached with good accuracy and a periodic orbit is obtained. With $\km=71$, and hence choosing $36$ processors, the speed-up is $\approx14$.

To obtain a comparison in a ``real-world" setting we consider the scenario of having a serial code and one performs $7$ correction steps, what is a reasonable $\km$ for the $8^{th}$ order method. Having now the possibility to use a parallel scheme on one compute node with $36$ processors, one might use $\km=71$. Comparing those two computing times, one still obtains a speed-up of $\approx1.5$ and at the same time a potentially better solution.

As a very last example, we consider the Arenstorf orbit problem with relatively few timesteps, $N=5{,}000$. (Note that in this work, we only use timesteps that are uniformly spaced. The Arenstorf orbit is an example of a difficult problem that asks for adaptive timestepping, which is beyond the scope of this work.) We use the $8^{th}$ order scheme and $\km = 7$. In Fig.~\ref{fig:Arenstorf_PredCorr} we plot the orbits for the predictor and the last correction step, once for Alg.~\ref{alg:mdode} (left) and once for Alg.~\ref{alg:mdodejonas} (right). All other parameters are exactly the same. It is clear that at least the predictor of Alg.~\ref{alg:mdode} is rather useless here, and also the last correction value is not a closed orbit. On the other hand, the modifications done in Alg.~\ref{alg:mdodejonas} lead to a better solution, where both predictor and the last correction step are, in the `eyeball-norm', closed orbits. This once again shows the superiority of Alg.~\ref{alg:mdodejonas} over its more straightforward counterpart.

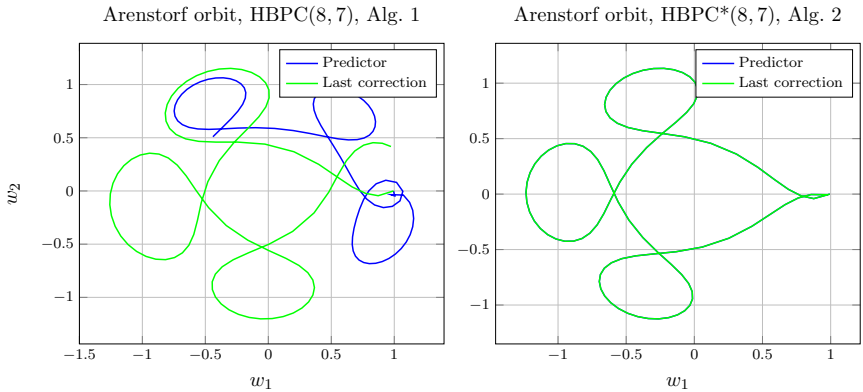
\begin{figure}[ht]
	\centering
    \ifthenelse{\boolean{compilefromscratch}}{
        
    \tikzsetnextfilename{Arenstorf2}
	\begin{tikzpicture}[scale=0.7]
		\begin{axis}[cycle list name=epslist,xlabel={$w_1$},ylabel={$w_2$} ,grid=major,legend style={at={(0.98,0.98)},anchor= north east,font=\footnotesize},
		title={Arenstorf orbit, $\method{8}{7}$, Alg.~1},
		label style={font=\large},title style={font=\large},legend cell align={left}]
		\addplot[mark=none,thick,blue]  table [x expr={\thisrowno{0}}, y expr={\thisrowno{1}}] {../Code/Arenstorf/Arenstorf_Alg1_RK8_5e3.csv};
		\addplot[mark=none,thick,green] table [x expr={\thisrowno{2}}, y expr={\thisrowno{3}}] {../Code/Arenstorf/Arenstorf_Alg1_RK8_5e3.csv};
		\legend{Predictor,Last correction}
		\end{axis}
	\end{tikzpicture}
	\tikzsetnextfilename{Arenstorf3}
	\begin{tikzpicture}[scale=0.7]
		\begin{axis}[cycle list name=epslist,xlabel={$w_1$} ,grid=major,legend style={at={(0.98,0.98)},anchor= north east,font=\footnotesize},
		title={Arenstorf orbit, $\methodt{8}{7}$, Alg.~2},
		label style={font=\large},title style={font=\large},legend cell align={left}]
		\addplot[mark=none,thick,blue]  table [x expr={\thisrowno{0}}, y expr={\thisrowno{1}}] {../Code/Arenstorf/Arenstorf_Alg2_RK8_5e3.csv};
		\addplot[mark=none,thick,green] table [x expr={\thisrowno{2}}, y expr={\thisrowno{3}}] {../Code/Arenstorf/Arenstorf_Alg2_RK8_5e3.csv};
		\legend{Predictor,Last correction}
		\end{axis}
	\end{tikzpicture}	
	}
	{
	  \includegraphics{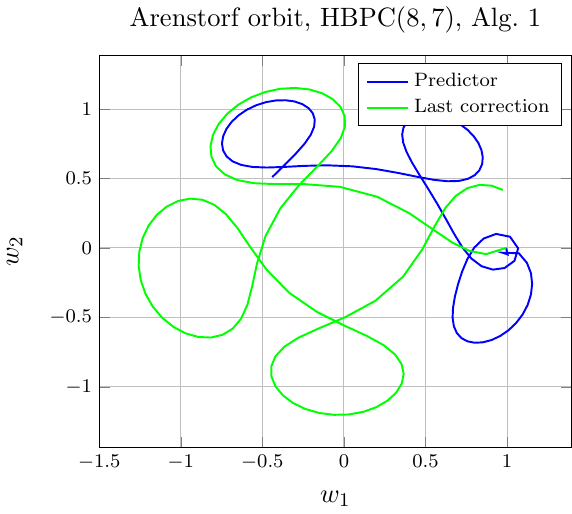}
	  \includegraphics{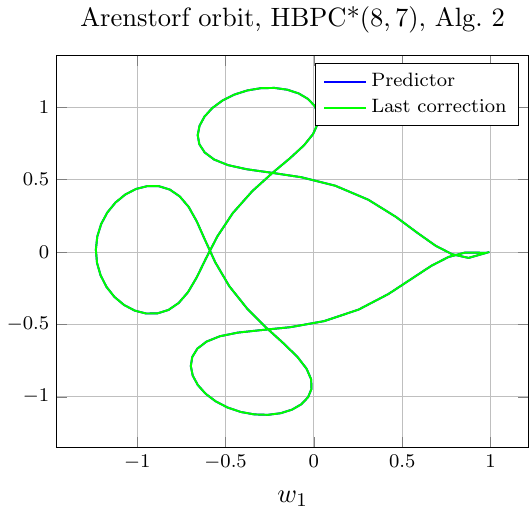}
	}
	\caption{Solution of the Arenstorf orbit problem for the two spatial variables. A total of $N = 5{,}000$ timesteps of the $8^{th}$ order IMEX-MD scheme described in Alg.~\ref{alg:mdode} (left) and in Alg.~\ref{alg:mdodejonas} (right) are used to simulate until final time $T=17.065216560159$, which should correspond to one period.  The improved algorithm produces something that is close to what is expected to be a closed orbit, whereas the first algorithm is far from the exact solution, with the predictor being wildly off.}
	\label{fig:Arenstorf_PredCorr}
\end{figure}

%% file: sec_conclusion.tex
In this work, we have developed a novel IMEX solvers for the numerical treatment of ODEs that operate on multiple derivatives of the ODE's flux function. The algorithm has been specifically designed to be 
\begin{itemize}
 \item of high-order (we showed results until order eight, higher orders are easily achievable by increasing the number of collocation points or the number of derivatives used),
 \item and parallelizable in time. 
\end{itemize}
There are two flavors of the method, one (Alg. \ref{alg:mdode}) is a softer modification to our previous result \cite{SealSchuetz19} that is more ammenable to analysis, and the second one (Alg. \ref{alg:mdodejonas}) is more storage efficient and has better scaling results. We have shown numerical results demonstrating the behavior of the algorithms. In the light of~\cite{OngSchroder2020}, the scaling results seem to be very much in line with other state-of-the-art methods. 

There are multiple important open areas that need to be addressed with future work. The most accessible problem to consider would be to increase the total number of derivatives used, or work on different integration points $c$.  Of notable interest would be to explore collocation solvers described in Ex.~\ref{expl:rk} constructed from the so-called Gauss-Turan-type~\cite{StroudStancu1965} quadrature rules.  These solvers would ideally offer low-storage alternatives to the present solvers.  Next, it would be interesting to explore versions of this solver that make use of arbitrary multiderivative Runge-Kutta methods, not of the collocation variety.  In addition, it is imperative that we further explore implementations for PDEs.  To date, multiderivative methods have largely been sidelined, even though they can outperform lauded and highly optimized \texttt{MATLAB} routines such as the builtin \texttt{ode15s} (cf. \cite{AbdiConte2020} for one such case study).  One reason for this lack of broader interest is arguably the difficult computation of the second (and third,~...) derivatives in practical applications. For some ODEs stemming from a semi-discrete PDE, this can be done rather elegantly~\cite{MultiDerHDG2015}, but for many, it is a tedious task, which requires leveraging Lax-Wendroff type time discretizations.  We therefore suggest exploring the possibilities of using finite-difference approximations, as done in~\cite{BAEZA201887,ZorioEtAl} in the context of explicit Taylor methods. It is not yet clear how this interacts with the stability properties of the overall method, in particular for highly stiff problems, as these type of discretizations require fully discrete stability analyses, but is certainly worth looking into.  Furthermore, the extension and testing of this time-stepping option to the classical IMEX application areas remains to be explored.